\pdfoutput=1
\documentclass{amsart}
\usepackage{calc}
\usepackage[nomessages]{fp}
\usepackage{tikz,amsmath,amssymb}
\usepackage{pgfplots}
\usepackage{quiver}
\usepackage{cancel}
\usetikzlibrary{matrix,arrows,backgrounds,shapes.misc,shapes.geometric,patterns,calc,positioning,intersections}
\usetikzlibrary{decorations.pathmorphing, angles, quotes}
\usepgflibrary{patterns}
\pgfplotsset{compat=1.15}
\usepackage{mathrsfs}
\usepackage{rotating}
\usepackage{tkz-euclide} 
\usetikzlibrary{decorations.markings}
\usetikzlibrary{shapes,fit}
\usepackage[english]{babel}
\usepackage[hidelinks]{hyperref}
\usepackage{amsmath,amsfonts,amssymb}
\usepackage{amsthm}
\usepackage[all]{xy}
\usepackage{ stmaryrd }
\usepackage{ mathrsfs }
\theoremstyle{plain} 

\usepackage{epsfig}
\usepackage[T1]{fontenc}
\usepackage{lmodern}
\usepackage{eurosym}
\usepackage{graphicx}
\usepackage{enumitem}
\usepackage{amsbsy}
\usepackage{rotating}
\usepackage[noabbrev,capitalise]{cleveref}
\usepackage[labelsep=period]{caption}
\usepackage{multirow}
\usepackage{xcolor}
\usepackage{nicematrix}

\usepgfplotslibrary{fillbetween}
\makeatletter
\tikzset{use path/.code=\tikz@addmode{\pgfsyssoftpath@setcurrentpath#1}}
\makeatother
\pgfdeclarelayer{bg}
\pgfsetlayers{bg,main}

\usepackage{stackengine,scalerel}
\stackMath
\def\Resleq{\ThisStyle{\mathrel{%
  \stackinset{r}{.75pt+.15\LMpt}{t}{.1\LMpt}{\rule{.3pt}{1.1\LMex+.2ex}}{\SavedStyle\leqslant}%
}}}

\def\Accordleq{\ThisStyle{\mathrel{%
  \stackinset{r}{.75pt+.15\LMpt}{t}{.1\LMpt}{\rule{.3pt}{1.1\LMex+.2ex}}{\SavedStyle\preccurlyeq}%
}}}

\def\Accordlneq{\ThisStyle{\mathrel{%
  \stackinset{r}{.75pt+.15\LMpt}{t}{.1\LMpt}{\rule{.3pt}{1.1\LMex+.2ex}}{\SavedStyle\prec}%
}}}

\definecolor{midnightblue}{rgb}{0.1, 0.1, 0.44}
\definecolor{plum}{rgb}{0.56, 0.27, 0.52}
\definecolor{Plum}{rgb}{0.56, 0.27, 0.52}
\definecolor{patriarch}{rgb}{0.5, 0.0, 0.5}
\definecolor{darkgreen}{rgb}{0.0, 0.2, 0.13}
\definecolor{darkcerulean}{rgb}{0.03, 0.27, 0.49}
\definecolor{jade}{rgb}{0.0, 0.66, 0.42}

\newcommand{\barredUpsilon}{\addbar{0.05}{}{\Upsilon}}

\newcommand{\Anti}{\operatorname{Anti}}

\newcommand{\ArExt}{\operatorname{ArExt}}
\newcommand{\OvExt}{\operatorname{OvExt}}
\newcommand{\s}{\operatorname{\pmb{s}}}
\renewcommand{\t}{\operatorname{\pmb{t}}}

\makeatletter
\newcommand{\addbar}[3]{{\vphantom{#3}\mathpalette\add@bar{{#1}{#2}{#3}}}}
\newcommand{\add@bar}[2]{\add@@bar{#1}#2}
\newcommand{\add@@bar}[4]{%
  \begingroup
  \sbox\z@{$\m@th#1#4$}%
  \ooalign{%
    \hidewidth\kern#2\wd\z@\add@@@bar{#1}{#3}\hidewidth\cr
    \box\z@\cr
  }%
  \endgroup
}
\newcommand{\add@@@bar}[2]{%
  \sbox\tw@{$\m@th#1\newmcodes@\if\relax#2\relax-\else\bm{-}\fi$}%
  \raisebox{\dimexpr(\ht\z@-\ht\tw@)/2}{\usebox\tw@}%
}
\makeatother

\newcommand\precdot{\mathrel{\ooalign{$\prec$\cr
  \hidewidth\raise0ex\hbox{$\cdot\mkern0.5mu$}\cr}}}

\hypersetup{
	colorlinks=true,
	linkcolor=darkcerulean,
	citecolor=jade,
	filecolor=magenta,      
	urlcolor=cyan,
	pdfpagemode=FullScreen,
}

\renewcommand{\mod}{\operatorname{mod}}

\newcommand{\rep}{\operatorname{rep}}
\newcommand{\add}{\operatorname{add}}
\newcommand{\Ker}{\operatorname{Ker}}

\newcommand{\Hom}{\operatorname{Hom}}

\newcommand{\Ext}{\operatorname{Ext}}

\newcommand{\ind}{\operatorname{ind}}
\newcommand{\proj}{\operatorname{proj}}

\newcommand{\Res}{\operatorname{Res}}
\newcommand{\ResOrd}{\operatorname{\pmb{Res}}}

\newcommand{\opQ}{\operatorname{\mathbf{Q}}}
\newcommand{\opR}{\operatorname{\mathbf{R}}}

\newcommand{\opResAc}{\operatorname{\pmb{\mathscr{R}}}}
\newcommand{\Prj}{\operatorname{Prj}}
\newcommand{\MM}{\operatorname{M}}
\newcommand{\NP}{\operatorname{N}_{\proj}}
\newcommand{\Syg}{\operatorname{Syg}}
\newcommand{\Surf}{\operatorname{\pmb{\mathcal{S}}}}

\newcommand{\tc}{t_{\operatorname{cell}}}
\renewcommand{\sc}{s_{\operatorname{cell}}}
\renewcommand{\Ker}{\operatorname{Ker}}
\newcommand{\cl}{\operatorname{\mathsf{cl}}}
\newcommand{\JI}{\operatorname{JIrr}}
\newcommand{\col}{\operatorname{\mathsf{col}}}
\newcommand{\D}{\mathbb{D}}
\newcommand{\Ups}{\barredUpsilon}
\newcommand{\Accord}{\operatorname{\pmb{\mathscr{A}}}}
\newcommand{\KE}{\operatorname{\mathbf{KE}}}
\newcommand{\CoZ}{\operatorname{Co-\mathbf{Z}}}
\newcommand{\V}{\operatorname{\mathbf{V}}}
\newcommand{\X}{\operatorname{\mathbf{X}}}
\newcommand{\All}{\operatorname{\mathbf{All}}}
\newcommand{\Smov}{\operatorname{\mathbf{S}}}
\newcommand{\W}{\operatorname{\mathbf{W}}}
\newcommand{\AR}{\operatorname{AR}}

\newcommand{\bleq}{\mathrel{\mathpalette\bleqinn\relax}}
\newcommand{\bleqinn}[2]{%
  \ooalign{%
    \raisebox{.2ex}{$#1\blacktriangleleft$}\cr
    $#1\leqslant$\cr
  }%
}

\author[B.~Dequêne]{Benjamin Dequêne}
\address[B.~Dequêne]{School of Mathematics, University of Leeds}
\email{B.D.Dequene@leeds.ac.uk}

\author[M.~Schoonheere]{Michael Schoonheere}
\address[M.~Schoonheere]{LAMFA, Université de Picardie Jules Verne}
\email{michael.schoonheere@u-picardie.fr}

\title[Resolving subcategories for gentle algebras II]{Resolving subcategories for gentle algebras II: Resolving subcategories for gentle trees}

\date{\today}

\usepackage{thmtools}
\declaretheorem[numberwithin=section,name=Theorem,
refname={Theorem,Theorems},
Refname={Theorem,Theorems}]{theorem}
\declaretheorem[numberlike=theorem,name=Lemma,
refname={Lemma,Lemmas},
Refname={Lemma,Lemmas}]{lemma}
\declaretheorem[numberlike=theorem,name=Proposition,
refname={Proposition,Propositions},
Refname={Proposition,Propositions}]{prop}
\declaretheorem[numberlike=theorem,name=Corollary,
refname={Corollary,Corollaries},
Refname={Corollary,Corollaries}]{cor}

\declaretheorem[style=definition,numberlike=theorem,name=Definition,
refname={Definition,Definitions},
Refname={Definition,Definitions}]{definition}
\declaretheorem[style=definition,numberlike=theorem,name=Convention,
refname={Convention,Conventions},
Refname={Convention,Conventions}]{conv}
\declaretheorem[style=definition,numberlike=theorem,name=Algorithm,
refname={Algorithm,Algorithms},
Refname={Algorithm,Algorithms}]{algo}

\declaretheorem[style=definition,numberlike=theorem,name=Example,
refname={Example,Examples},
Refname={Example,Examples}]{ex}

\declaretheorem[style=remark,numberlike=theorem,name=Remark,
refname={Remark,Remarks},
Refname={Remark,Remarks}]{remark}

\usepackage{todonotes}

\definecolor{aquamarine}{rgb}{0.5, 1.0, 0.83}

\newcommand{\new}[1]{\textit{\textbf{\color{patriarch}{#1}}}}

\definecolor{dark-green}{RGB}{14,150,2}
\newcommand{\gpoint}{{\color{dark-green}{\circ}}}
\newcommand{\rpoint}{\textcolor{red}{\bullet}}
\definecolor{darkgray}{rgb}{0.66, 0.66, 0.66}
\definecolor{darkpink}{rgb}{0.91, 0.33, 0.5}
\newcommand{\gsquare}{\color{dark-green}{\pmb{\square}}}
\newcommand{\rsquare}{\color{red}{{\blacksquare}}}
\newcommand{\osquare}{\color{orange}{\pmb{\boxtimes}}}
\newcommand{\psquare}{\color{darkpink}{\pmb{\times}}}

\definecolor{mypurple}{rgb}{0.63, 0.36, 0.94}

\usepackage{pict2e,picture}
\makeatletter
\DeclareRobustCommand{\bbDelta}{{\mathpalette\bb@Delta\relax}}
\newcommand{\bb@Delta}[2]{%
  \begingroup
  \sbox\z@{$\m@th#1\Delta$}%
  \dimendef\Dht=6 \dimendef\Dwd=8
  \setlength{\Dwd}{\wd\z@}%
  \setlength{\Dht}{\ht\z@}%
  \begin{picture}(\Dwd,\Dht)
  \put(0,0){$\m@th#1\Delta$}
  \put(.42\Dwd,.7\Dht){\line(10,-26){.25\Dht}}
  \end{picture}%
  \endgroup
}
\makeatother

\begin{document}

\begin{abstract}
This paper is the second part of a series that intends to study the resolving subcategories for gentle algebras over an algebraically closed field $\mathbb{K}$. 

As in the first part, we continue to focus on gentle quivers $(Q,R)$, where $Q$ is a directed tree, known as gentle trees. In our previous work, via a modified surface model for gentle algebras with finite global dimension, we studied the join-irreducible elements of the lattice of resolving subcategories of $\mathbb{K}Q/\langle R \rangle - \text{mod}$, which happen to be those generated by a non-projective indecomposable object. 

In this paper, we notice that this lattice is not semidistributive in general and, accordingly, introduce a so-called upper join-decomposition, replacing the canonical one.  Together with the techniques we develop in our geometric model, it allows us to describe the resolving subcategories of any gentle tree. These same techniques let us explicitly construct the resolving subcategory generated by any collection of indecomposable $\mathbb{K}Q/\langle R\rangle$-modules.
\end{abstract}
	\maketitle

%\ibenj{Cet article fait parties d'une série se focalisant sur l'étude des sous-catégories res de mod(Q,R), avec (Q,R) aimable.Dans le précédent article, nous avons étudié les join-irreductibles du poset des catégories résolvantes -- qui se trouvent être exactement les monogènes...Dans cet article, en notant que ces treillis ne sont pas semi-distributifs, nous introduisons une décomposition dite supérieure, analogue à la décomposition canonique. Conjugué aux techniques que nous développons dans le modèle géometrique associé, cela nous permettra de décrire les catégories résolvantes de tout arbre aimable.Cette même technique permet d'expliciter la catégorie résolvante engendré par une collection d'indécomposables.}
    
\tableofcontents

	\section{Introduction}
	\label{sec:Intro}
	\pagestyle{plain}

The notion of a resolving subcategory dates back, at least, to \cite{Auslander1969}, who used this notion to describe the subcategory of modules of Gorenstein dimension 0. These subcategories contain the projective modules and are closed under taking sums, summands, extensions, and kernels of epimorphisms. 
The later works \cite{Auslander1991, Adachi2022} provide a more precise study of the functorially finite resolving subcategories. They prove that functorially finite resolving subcategories are in bijection with hereditary cotorsion pairs. This bijection preserves inclusion, and thus, the study of the poset of resolving subcategories gives another way to study the poset of hereditary cotorsion pairs. In these two papers, under some stronger assumptions on the categories, they prove that resolving subcategories are in bijection with tilting (resp. silting) objects in the case of abelian (resp. extriangulated) categories. Such objects appear naturally in the context of exceptional sequences in \cite{Ringel91} and their characteristic tilting, which is deeply linked with quasi-hereditary algebras.
To obtain such information on the tilting objects, one must compute all the resolving subcategories. A way to achieve this is to calculate the smallest subcategory containing a set of objects $\mathcal{X}$. Takahashi \cite{T09} addressed this issue, providing an algorithm to compute such a resolving subcategory. This category is called the resolving closure of $\mathcal{X}$. The algorithm given works in the broad setup of the module category of an algebra of finite representation type and thus does not benefit from the combinatorial insights offered by the algebra. 

Gentle algebras are algebras with a strong combinatorial structure. They arise from quivers with relations, appearing first in \cite{Assem1981,Assem1982}. The combinatorial behavior of gentle algebras is widely studied: the extensions of indecomposable objects have at most two indecomposable summands \cite{BS80,WW85,GP68,SW83,DS87}. The hom-spaces \cite{BR87} and extensions \cite{S99,BDMTY19,CPS21} are endowed with a basis coming from string combinatorics. To complete the combinatorics, geometric models consisting of dissected surfaces are in bijection with the module category over the path algebra, the derived category, and the two-term homotopy category of projectives \cite{BCS21,OPS18,APS19,PPP18,PPP182}. These models offer more straightforward categorical computations.

In a previous work \cite{DS251} we started the computation of the lattice of resolving subcategories of $\mod{A}$ for $A$ a gentle algebra arising from a disc or equivalently gentle algebras whose underlying quiver is a tree. The first step was to prove the following result :

\begin{theorem}[\cite{DS251}]
The join-irreducible elements in the poset of resolving subcategories in $\rep(Q,R)$ ordered by inclusion are precisely given by the smallest resolving subcategories containing  $X$, for $X$ an indecomposable non-projective object.
\end{theorem}

We call monogeneous any resolving subcategory of the form $\Res(X)$ for $X$ an indecomposable non-projective object. We computed all the monogeneous resolving subcategories using the geometric model in the following way. The dissected disc is endowed with a so-called set of red points on the boundary: one in each cell of the dissected surface. These points are the endpoints of the accordions, which are curves on the dissected disc in bijection with indecomposable modules of the associated gentle algebra. To each accordion $\delta$ we associate the set $\NP(\delta)$ of accordions in bijection with indecomposable projective modules appearing in the projective resolution of $\MM(\delta)$ or having  nontrivial higher extension with $\MM(\delta)$. The set of endpoints of curves in $\NP(\delta)$ are colored red, orange, or green according to a combinatorial rule explained in \cref{def:colourendpoints}. We can now state the main result of the previous article.

\begin{theorem}[\cite{DS251}]
Let $(Q,R)$ be a gentle tree, and $\Surf(Q,R) = (\pmb{\Sigma}, \mathcal{M}, \Delta^{\gpoint})$ be the dissected surface associated with $(Q,R)$. For any accordion $\delta$ the monogeneous resolving closure  $\Res(\MM(\delta))$ is the additive closure of the union of the indecomposable projectives and of $\{\MM(\varsigma) \mid \varsigma \in \opResAc(\delta)\}$ where the set $\opResAc(\delta)$ is the set of accordions whose sources are colored either red or orange and targets are colored either orange or green.
\end{theorem}

We aim at describing all the resolving subcategories through these monogeneous resolving subcategories. Even though the lattice is not semi-distributive, we obtain an upper join decomposition (see \cref{def:upperjoindec} ) of all the elements of the lattice as joins of join-irreducible objects.

\begin{theorem} \label{thm:Resclosalgosucceedintro}
Let $(Q,R)$ be a gentle tree, and $\Surf(Q,R) = (\pmb{\Sigma}, \mathcal{M}, \Delta^{\gpoint}$) be its associated $\gpoint$-dissected marked disc. For any antichain $\mathfrak{\mathfrak{D}}$ in $(\Accord, \Accordleq)$, \cref{algo:resclosalgo} ends, and returns the unique antichain $\mathfrak{E}$ such that \[\opResAc(\mathfrak{D}) = \opResAc(\mathfrak{E}) = \bigcup_{\delta \in \mathfrak{E}} \opResAc(\delta).\]
\end{theorem}

This paper is organized as follows: \cref{sec:Res} contains all the preliminaries. In \cref{ss:RecallResolv}. We define the notion of resolving subcategories and state the general algorithm for computing resolving closures. In \cref{ss:RecallGentleGeom}, we recall what a gentle algebra is and some tools from their geometric models that will be used in computations. \cref{ss:recalltrees} recalls some specificities of gentle trees. Then, the results of \cite{DS251} on the monogeneous resolving subcategories are stated in \cref{ss:RecallMono}. In \cref{sec:ResPoset} we study closure operators (\cref{ss:closop}) and decompositions in lattices through their join-irreducible (\cref{ss:latticeJI}) and then apply it to the lattice of resolving subcategories in \cref{ss:rescat} and describe results for the resolving order relation in  \cref{ss:geometric}. In \cref{sec:Matchable} we go through the difficulty of fixing a common coloration on the sets of vertices of two sets of neighbouring projectives. To obtain such a coloration we define in \cref{ss:Matchable} the notion of matchability. It follows naturally from the definition that matchability has a strong and a weak case. We study the second one in \cref{ss:weakmatchabilityconfig}. We attend to the case were colorations are non matchable and reduce it to the matchable case in \cref{ss:nomatchablemoves}.  In \cref{sec:Moves}, we list all the operations required to transform any antichain into the canonical one that describes the resolving subcategory. To do so, we refine information on the colorations and endpoints of curves in \cref{ss:exclusive}. The rest of the section properly lists the operations on antichains we must study. Then, in \cref{sec:TreePart2}, we compute all the resolving closures of sets of indecomposable objects. \cref{ss:Algo} states the transformation of antichains into the upper join decomposition and \cref{ss:examplePart1} gives a detailed set of examples.

\begin{conv} \label{conv:orders} Throughout this paper, we fix the notations of main order relations enumerated in \cref{table:orders}.
\begin{table}[!ht]
\centering
\begin{tabular}{|c|c|c|}
\hline
$\preccurlyeq_v$ & \begin{tabular}{c}
   Counter-clockwise order around \\
   the marked point $v \in \mathcal{M}_{\rpoint}$ on\\
   $\rpoint$-arcs with $v$ as an endpoint
\end{tabular}  & p. \pageref{order:clock} \\ \hline
$\Resleq$ & Res-relation on $\pmb{\ind \setminus \proj}(Q,R)$ & p. \pageref{order:resindec} \\ \hline
$\bleq$ & \begin{tabular}{c}
   Order on antichains induced \\
   by the inclusion order on \\
   order ideals of a given poset
\end{tabular} & p. \pageref{order:antichain} \\ \hline
$\Accordleq$ & Res-relation on $\Accord'(\pmb{\Sigma}, \mathcal{M}, \Delta^{\gpoint})$  & p. \pageref{def:GeomResrelation}  \\ \hline 
\end{tabular}
\caption{\label{table:orders} Main order relations used in this paper. The last column provides the page where we define each term.}
\end{table}
\end{conv}
	
	\section{Background}
	\label{sec:Res}
	\pagestyle{plain}

In this section, we recall some relevant definitions, notations, and results from our previous work that will be useful in the following. We allow ourselves not to recall all the precise terminology we used previously, and we refer the reader to \cite{DS251} for more details if needed.

\subsection{Resolving subcategories}
\label{ss:RecallResolv}

Let $\mathscr{C}$ be a Krull--Schmidt abelian category with enough projective objects.
\begin{conv} \label{conv:add}
A subcategory is said to be \emph{additively closed} if it is closed under direct sums and summands.
\end{conv}

\begin{definition} \label{def:resolv}
A full subcategory $\mathscr{R} \subseteq \mathscr{C}$ is called \new{resolving} if it satisfies the following conditions: 
\begin{enumerate}[label=$(\mathsf{R \arabic*})$, itemsep=1mm]
\setcounter{enumi}{-1}
\item \label{R0} $\mathscr{R}$ is additively closed
\item \label{R1} $\mathscr{R}$ generates $\mathscr{C}$,
\item \label{R2} $\mathscr{R}$ is closed under extensions, and,
\item \label{R3} $\mathscr{R}$ is closed under taking kernels of epimorphisms
\end{enumerate}
\end{definition}

We can simplify some conditions to check that a subcategory is resolving under the hypotheses we enforce on $\mathscr{C}$.

\begin{lemma} \label{lem:othercharactresolv}
Let $\mathscr{R} \subseteq \mathscr{C}$ be a full additive subcategory. We have the following equivalences:
\begin{enumerate}[label=$(\roman*)$, itemsep=1mm]
\item \label{ires} $\mathscr{R}$ satisfies \ref{R1} if, and only if, $\proj(\mathscr{C}) \subseteq \mathscr{R}$;
\item $\mathscr{R}$ is resolving if, and only if,  $\mathscr{R}$ satisfies \ref{R0}, \ref{R1} \ref{R2}, and $\mathscr{R}$ is closed under taking syzygies.
\end{enumerate}
\end{lemma}

As resolving subcategories are additively closed and must contain all the projective objects, they are characterized by their indecomposable non-projective objects. For any resolving subcategory $\mathscr{R} \subset \mathscr{C}$, we write $\pmb{\ind \setminus \proj}(\mathscr{R})$ for the set of indecomposable non-projective objects in $\mathscr{R}$ up to isomorphism.

By the fact that resolving subcategories are closed under arbitrary intersections, given any $\mathcal{X} \subseteq \pmb{\ind \setminus \proj}(\mathscr{C})$, we can define $\Res(\mathcal{X})$ the \new{resolving closure} of $\mathcal{X}$ as the smallest resolving subcategory containing  $\mathcal{X}$. For $X,Y \in \ind(Q,R)$, we write $\Res(X)$ for $\Res(\{X\})$ and $\Res(X,Y)$ for $\Res(\{X,Y\})$. We define an analogous operator on subsets of $\pmb{\ind \setminus \proj}(\mathscr{C})$, denoted $(-)^{\Res}$, which returns $\pmb{\ind \setminus \proj}(\Res(\mathcal{X}))$ for any $\mathcal{X} \subseteq \pmb{\ind \setminus \proj}(\mathscr{C})$. One easily checks that $(-)^{\Res}$ is a closure operator on $\pmb{\ind \setminus \proj}(\mathscr{C})$.

Using \cite[Theorem 2.11]{DS251} allows one to check if an additive subcategory is resolving by doing some calculations on indecomposable representations. We update the algorithm of Takahashi \cite{T09} by only focusing on indecomposable non-projective objects in $\mathscr{C}$ (see \cite[Algorithm 2.14]{DS251}). Even though this algorithm is less effective than the one Takahashi gave, we support the fact that we earn in precision to construct step by step the set of non-projective indecomposable objects $\mathcal{X}^{\Res}$ that generates additively $\Res(\mathcal{X})$, for all collections $\mathcal{X} \subseteq \pmb{\ind \setminus \proj}(\mathscr{C})$. We have already discussed some conditions to ensure that this algorithm terminates.

In the following, we restrict ourselves to the study of module categories of a subfamily of gentle algebras, refining the algorithm in this setting to facilitate explicit calculations. Consequently, it will allow us to describe all the resolving subcategories of those categories explicitly.

\subsection{Geometric model for gentle quivers}
\label{ss:RecallGentleGeom}

Fix an algebraically closed field $\mathbb{K}$.
Let $(Q,R)$ be a gentle quiver. The category $\rep(Q,R)$ of finite-dimensional representations of $(Q,R)$ (over $\mathbb{K}$) is abelian, Krull--Schmidt, and has enough projective and injective objects. Write $\ind(Q,R)$ for the (isomorphism classes of) indecomposable representations of $(Q,R)$. We can describe $\rep(Q,R)$ using string combinatorics \cite{BR87,CB89} by focusing on representation-finite gentle quivers. This allows us to describe morphisms, and thus syzygies, as extensions between objects in $\ind(Q,R)$. In particular, Schröer \cite{S99}, \c{C}anak\c{c}i, Pauksztello, and Schroll \cite{CPS21} and Brüstle, Douville Mousavand, Thomas , Y\i ld\i r\i m \cite{BDMTY19} give an explicit basis of $\Ext^1(M,N)$ in terms of arrow and overlap extensions, for any $(M,N) \in \ind(Q,R)^2$. 

The article \cite{BCS21} introduced a geometric model (see also \cite{OPS18,PPP18}) for gentle quivers. Let us first state some conventions and definitions.
\begin{conv}
We recall that whenever we talk about \new{surfaces}, we always mean an oriented compact surface $\pmb{\Sigma}$ with a finite number (which may be zero) of open discs removed. Write $\partial \pmb{\Sigma}$ for the boundary of $\pmb{\Sigma}$. Such a surface is determined, up to homeomorphism, by its genus and by the number of connected components of $\partial \pmb{\Sigma}$; we will refer to these connected components as the \new{boundary components} of $\pmb{\Sigma}$.
\end{conv}

\begin{definition} \label{def:marksurf}
A \new{marked surface} is a pair $(\pmb{\Sigma},\mathcal{M})$, where $\pmb{\Sigma}$ is a surface and $\mathcal{M}$ is a finite set of \new{marked points} of $\pmb{\Sigma}$ such that $\mathcal{M}$ admits a bipartition $\{\mathcal{M}_{\gpoint},\mathcal{M}_{\rpoint}\}$ such that, on each boundary component of $\pmb{\Sigma}$: 
\begin{enumerate}[label = $\bullet$, itemsep=0.1em]
	\item there is at least one marked point of each color;
	\item marked points in $\mathcal{M}_{\gpoint}$ and $\mathcal{M}_{\rpoint}$ alternate.
\end{enumerate}
\end{definition}

\begin{definition} \label{def:arcsanddissec} Let $(\pmb{\Sigma},\mathcal{M})$ be a marked surface.
\begin{enumerate}[label=$\bullet$,itemsep=1mm]
\item A \new{$\gpoint$-arc} is a curve on $(\pmb{\Sigma},\mathcal{M})$ joining two points in $\mathcal{M}_{\gpoint}$; it is a continuous map from the closed interval $[0,1]$ to $\pmb{\Sigma}$ with endpoints in $\mathcal{M}_{\gpoint}$, and such that the image of its interior is disjoint from $\mathcal{M}$.

\item A $\gpoint$-arc is said to be \new{simple} if it does not intersect itself (except perhaps at its endpoints). 

\item A \new{$\gpoint$-dissection} of $(\pmb{\Sigma}, \mathcal{M})$ is a collection $\Delta^{\gpoint}$ of pairwise non-intersecting simple $\gpoint$-arcs which cut the surface into polygons (that is to say, into simply-connected regions), called the \new{cells} of the dissection. The triplet $(\pmb{\Sigma}, \mathcal{M}, \Delta^{\gpoint})$  is called a  \new{$\gpoint$-dissected marked surface}.

%\item A \new{$\gpoint$-dissection} $\Delta^{\gpoint}$ of $(\pmb{\Sigma}, \mathcal{M})$ is said to be \new{admissible} if each cell of $\Delta^{\gpoint}$ contains at least one marked point in $\mathcal{M}_{\rpoint}$. In the case where each cells of $\Delta^{\gpoint}$ contains exactly one marked point in $\mathcal{M}_{\rpoint}$, we say that $\Delta^{\gpoint}$ is \new{dualizable}.

\item We define \new{$\rpoint$-arcs} and \new{$\rpoint$-dissection} similarly.

\item Given a $\gpoint$-dissection or a $\rpoint$-dissection $\Delta$ of a marked surface $(\pmb{\Sigma}, \mathcal{M})$, we denote by $\pmb{\Gamma}(\Delta)$ the set of cells of $(\pmb{\Sigma}, \mathcal{M},\Delta)$. 
\end{enumerate}
\end{definition}

The articles \cite{BCS21,OPS18,PPP18} gave a one-to-one correspondence from gentle quivers to $\gpoint$-dissected marked surfaces, uniquely determined up to oriented homeomorphisms and homotopies of $\gpoint$-arcs. We write $\Surf(Q,R) = (\pmb{\Sigma}, \mathcal{M}, \Delta^{\gpoint})$ for the $\gpoint$-dissected marked surface associated to the gentle quiver $(Q,R)$. Moreover, we can translate information from $\rep(Q,R)$ to the geometric model and vice-versa.

\begin{theorem}[\cite{BCS21,OPS18,PPP182}] \label{thm:GeomandRep} Let $(\pmb{\Sigma}, \mathcal{M}, \Delta^{\gpoint})$ a $\gpoint$-dissected marked surface such that $\mathcal{M}_{\gpoint} \subset \partial \pmb{\Sigma}$, and let $(\opQ(\Delta^{\gpoint}),\opR(\Delta^{\gpoint}))$ be its associated gentle quiver. Then we have a one-to-one correspondence from indecomposable representations of  $(\opQ(\Delta^{\gpoint}),\opR(\Delta^{\gpoint}))$, up to isomorphism, to a subfamily of $\rpoint$-curves, called \emph{accordions} on $(\pmb{\Sigma}, \mathcal{M}, \Delta^{\gpoint})$, up to homotopy.
\end{theorem}

\begin{definition}[{\cite{HKK17,OPS18}}] \label{def:accordions}
	    A $\rpoint$-arc $\delta$ in $(\pmb{\Sigma}, \mathcal{M},\Delta^{\gpoint})$ is an \new{accordion} if it satisfies the following conditions: whenever~$\delta$ enters a cell of~$\Delta^{\gpoint}$ by crossing an $\gpoint$-arc~$\eta$,
		    \begin{enumerate}[label=$(\alph*)$,itemsep=0.5em]
			\item if it leaves the cell, it leaves it by crossing an $\gpoint$-arc~$\zeta$ adjacent to~$\eta$;
			\item the relevant segments of the arcs~$\eta$,~$\zeta$ and~$\delta$ bound a disk that does not contain the unique marked point in~$\mathcal{M}_{\rpoint}$ belonging to the cell.
		    \end{enumerate}
		See \cref{fig:rulesclosedaccord} for a picture illustrating the rules $(a)$ and $(b)$. 
		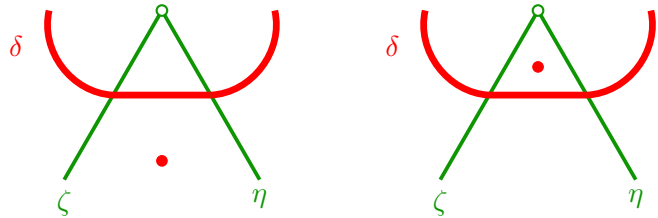
\begin{figure}[!ht] 
			\begin{center}
			    \begin{tikzpicture}[mydot/.style={
						circle,
						thick,
						fill=white,
						draw,
						outer sep=0.5pt,
						inner sep=1pt
					}]
					
				\begin{scope}[xshift=0cm]
					\tkzDefPoint(0,0){O}\tkzDefPoint(0,1.5){1}
					\tkzDefPointsBy[rotation=center O angle 120](1,2){2,3}
					\tkzDefMidPoint(1,2)
					\tkzGetPoint{i}
					\tkzDefMidPoint(1,3)
					\tkzGetPoint{j}
					\tkzDefPoint(-1.5,1.5){4}
					\tkzDefPoint(1.5,1.5){5}
					
					\draw[line width=0.5mm,dark-green](1) edge (2);
					\draw[line width=0.5mm,dark-green](1) edge (3);
					\draw[line width=0.9mm,red] (i) edge (j);
					\draw[line width=0.9mm, bend right=50,red] (4) edge (i);
					\draw[line width=0.9mm,bend left=50,red] (5) edge (j);
					\tkzDrawPoints[size=4,color=dark-green,mydot](1);
					
					\tkzDefPoint(0,-0.5){6}
					\tkzDrawPoints[size=4,color=red](6);
					
					\tkzLabelPoint[below left=0.3](4){{\large $\color{red}{\delta}$}}
					\tkzLabelPoint[below](3){{\large $\color{dark-green}{\eta}$}}
					\tkzLabelPoint[below](2){{\large $\color{dark-green}{\zeta}$}}
				\end{scope} 
				\begin{scope}[xshift=5cm]
					\tkzDefPoint(0,0){O}\tkzDefPoint(0,1.5){1}
					\tkzDefPointsBy[rotation=center O angle 120](1,2){2,3}
					\tkzDefMidPoint(1,2)
					\tkzGetPoint{i}
					\tkzDefMidPoint(1,3)
					\tkzGetPoint{j}
					\tkzDefPoint(-1.5,1.5){4}
					\tkzDefPoint(1.5,1.5){5}
					
					\draw[line width=0.5mm,dark-green](1) edge (2);
					\draw[line width=0.5mm,dark-green](1) edge (3);
					\draw[line width=0.9mm,red] (i) edge (j);
					\draw[line width=0.9mm, bend right=50,red] (4) edge (i);
					\draw[line width=0.9mm,bend left=50,red] (5) edge (j);
					
					\tkzDrawPoints[size=4,color=dark-green,mydot](1);
					
					\tkzDefPoint(0,0.75){6}
					\tkzDrawPoints[size=4,color=red](6);
					
					\tkzLabelPoint[below left=0.3](4){{\large $\color{red}{\delta}$}}
					\tkzLabelPoint[below](3){{\large $\color{dark-green}{\eta}$}}
					\tkzLabelPoint[below](2){{\large $\color{dark-green}{\zeta}$}}
				\end{scope}
				\end{tikzpicture}
				\caption{\label{fig:rulesclosedaccord} Drawings representing the rules that an accordion must satisfy: on the left, $\delta$ satisfies the rule $(a)$ and $(b)$; on the right, $\delta$ does not satisfy the rule $(b)$.} \end{center} \end{figure}
	\end{definition} 
    %\ibenj{Y a-t-il vraiment besoin des dessins ?}
    %\iMicka{pas bien sûr ^^'}

    \begin{remark} \label{rem:BCSOPS} In \cite{HKK17,OPS18}, those $\rpoint$-arcs are called \emph{graded arcs}. We use here the terminology from \cite{PPP18}
    \end{remark}

Write $\Accord = \mathscr{A}(\pmb{\Sigma}, \mathcal{M}, \Delta^{\gpoint})$ for the set of accordions on $(\pmb{\Sigma}, \mathcal{M}, \Delta^{\gpoint})$. Recall that:
\begin{enumerate}[label=$\bullet$, itemsep=1mm]
    \item for any $M \in \ind(Q,R)$, we write $\gamma_{(M)}$ for its associated accordion up to homotopy; and,
    \item for any $\delta \in \Accord$, we denote by $\MM(\delta)$ its associated indecomposable representation up to isomorphism.
\end{enumerate} We can also describe morphisms and extensions between indecomposable representations as geometric configurations between accordions. We recall the precise statements.

\begin{prop}[\cite{BR87}] \label{prop:HomCrossing} Let $(\pmb{\Sigma},\mathcal{M},\Delta^{\gpoint})$ be a $\gpoint$-dissected marked surface with $\mathcal{M}_{\gpoint} \subset \partial{\pmb{\Sigma}}$. For any pair $(\delta, \eta) \in \Accord$, a basis of elements of $\Hom(\MM(\delta), \MM(\eta))$ is given by crossings as depicted in \cref{fig:HomCrossing}.
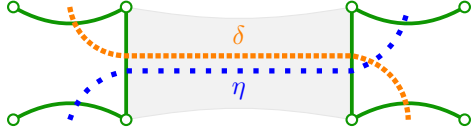
\begin{figure}[ht!]
    \centering
    \begin{tikzpicture}[mydot/.style={
				circle,
				thick,
				fill=white,
				draw,
				outer sep=0.5pt,
				inner sep=1pt
			}, fl/.style={->,>=latex}]
			\tkzDefPoint(0,0){O}\tkzDefPoint(0,1.5){1} 
			\tkzDefPoint(-1.5,1.5){2} 
			\tkzDefPoint(-1.5,0){3} 
			
			\tkzDefPoint(3,0){4}\tkzDefPoint(3,1.5){5} 
			\tkzDefPoint(4.5,1.5){6} 
			\tkzDefPoint(4.5,0){7} 
			
			\filldraw [fill=gray,opacity=0.1] (O) to (1) to [bend right=10] (5) to (4) to [bend right=10] cycle;
			
			\draw[line width=0.5mm,dark-green](1) edge (O);
			\draw[line width=0.5mm,dark-green, bend left=30](1) edge (2);
			\draw[line width=0.5mm,dark-green,bend right=30](O) edge (3);
			
			\draw[line width=0.5mm,dark-green](4) edge (5);
			\draw[line width=0.5mm,dark-green, bend left=30](6) edge (5);
			\draw[line width=0.5mm,dark-green,bend right=30](7) edge (4);
			
			\draw[line width=0.7mm,orange,densely dotted](-0.75,1.5) to [bend right=40] (0,0.85) to [bend left=0] node[above]{$\delta$}  (3,0.85) to [bend left=40] (3.75,0);
			\draw[line width=0.7mm,blue, bend left=30, loosely dotted](-0.75,0) to [bend left=30]  (0,0.65) to [bend left=0] node[below]{$\eta$} (3,0.65) to [bend right=30] (3.75,1.5);
			
			\tkzDrawPoints[size=4,color=dark-green,mydot](O,1,2,3,4,5,6,7);
		\end{tikzpicture}
    \caption{\label{fig:HomCrossing} Illustration of a crossing corresponding to a basis element of $\Hom(\MM(\delta), \MM(\eta))$. The shaded part is a part where all the segments of $\delta$ and all the ones of $\eta$, given by the cutting of $\pmb{\Sigma}$ with $\pmb{\Gamma}(\Delta^{\gpoint})$, are homotopic.}
\end{figure}
\end{prop}

\begin{conv}
Whenever we say that two arcs cross, they cross in their relative interior.
\end{conv}
	
\begin{prop}[\cite{DS251}]\label{prop:geo_mor}
Let $(\pmb{\Sigma},\mathcal{M},\Delta^{\gpoint})$ be a  $\gpoint$-dissected marked surface with $\mathcal{M}_{\gpoint} \subset \partial{\pmb{\Sigma}}$. Consider $p \in \mathbb{N}^*$ and $\delta_1,\ldots, \delta_p \in \Accord$. Let $\eta \in \Accord$ such that there exists a minimal epimorphism \[ \begin{tikzcd}
	{\displaystyle f: \bigoplus_{i=1}^p \MM(\delta_i)} & \MM(\eta)
	\arrow[two heads, from=1-1, to=1-2]
\end{tikzcd}\]
such that $f_{|\MM(\delta_i)}$ is given by exactly one crossing between $\delta_i$ and $\eta$. Then \[\Ker(f) = \bigoplus_{i=0}^p \MM(\kappa_i)\] where $\kappa_0,\ldots,\kappa_p \in \Accord$  are constructed from $\delta_1, \ldots, \delta_p$, and $\eta$ as depicted in \cref{fig:kerepi}. Note that $\kappa_0$ is defined in a dual way to $\kappa_p$ and thus the drawing of $\kappa_p$ is dual to the one of $\kappa_0 $.
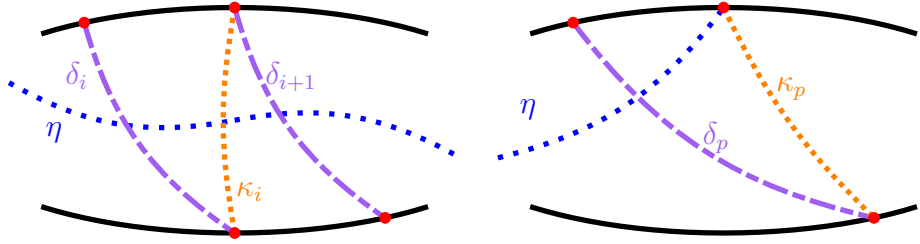
\begin{figure}[!ht]
\centering 
     \begin{tikzpicture}[mydot/.style={
					circle,
					thick,
					fill=white,
					draw,
					outer sep=0.5pt,
					inner sep=1pt
				}, scale = 1]
		\tikzset{
		osq/.style={
        rectangle,
        thick,
        fill=white,
        append after command={
            node [
                fit=(\tikzlastnode),
                orange,
                line width=0.3mm,
                inner sep=-\pgflinewidth,
                cross out,
                draw
            ] {}}}}
		\draw [line width=0.7mm,domain=50:130] plot ({4*cos(\x)}, {1.5*sin(\x)});
        \draw [line width=0.7mm,domain=230:310] plot ({4*cos(\x)}, {1.5*sin(\x)});
		\foreach \X in {0,1}
		{
		\tkzDefPoint(4*cos(pi/6*\X +pi/2),1.5*sin(pi/6*\X + pi/2)){\X};
		};
		\foreach \X in {2,3}
		{
		\tkzDefPoint(4*cos(pi/6*(\X-2) +3*pi/2),1.5*sin(pi/6*(\X-2) + 3*pi/2)){\X};
		};
		
		\draw[line width=0.7mm ,bend right=20,blue, loosely dotted](-3,0.5) edge (0,0);
		
		\draw[line width=0.7mm ,bend left=20,blue, loosely dotted](0,0) edge (3,-0.5);
		
		\draw[line width=0.7mm ,bend right=10,orange,dotted](0) edge (2);

		\draw [line width=0.7mm, mypurple,dash pattern={on 10pt off 2pt on 5pt off 2pt}, bend right=20] (1) edge (2);
		\draw [line width=0.7mm, mypurple,dash pattern={on 10pt off 2pt on 5pt off 2pt}, bend right=20] (0) edge (3);
		
		\foreach \X in {0,1,2,3}
		{
		\tkzDrawPoints[fill =red,size=4,color=red](\X);
		};
		
		%\foreach \X in {0}
		%{
		%\tkzDrawPoints[rectangle,fill =red,size=6,color=red](\X);
		%};
		
		%\foreach \X in {2}
		%{
		%\tkzDrawPoints[rectangle,size=6,color=dark-green,thick,fill=white](\X);
		%};
		%\foreach \X in {0,2}
		%{
		%\tkzDrawPoints[size=6,orange,osq](\X);
		%};
		%\foreach \X in {23}
		%{
		%\tkzDrawPoints[size=6,darkpink,line width=0.5mm,cross out, draw](\X);
		%};

		\tkzDefPoint(-2.4,0.1){gammaM};
		\tkzLabelPoint[blue](gammaM){\Large $\eta$}
		\tkzDefPoint(0.2,-0.7){gammaP};
		\tkzLabelPoint[orange](gammaP){\Large $\kappa_i$}
		%\tkzDefPoint(-0.4,-1.6){v};
		%\tkzLabelPoint[dark-green](v){\Large $v$}
		%\tkzDefPoint(0,1.9){w};
		%\tkzLabelPoint[red](w){\Large $w$}
		\tkzDefPoint(-2.1,0.9){deltav};
		\tkzLabelPoint[mypurple](deltav){\Large $\delta_i$}
		\tkzDefPoint(0.75,0.9){deltaw};
		\tkzLabelPoint[mypurple](deltaw){\Large $\delta_{i+1}$}
		
		\begin{scope}[xshift = 6.5cm]
		\draw [line width=0.7mm,domain=50:130] plot ({4*cos(\x)}, {1.5*sin(\x)});
        \draw [line width=0.7mm,domain=230:310] plot ({4*cos(\x)}, {1.5*sin(\x)});
		\foreach \X in {0,1}
		{
		\tkzDefPoint(4*cos(pi/6*\X +pi/2),1.5*sin(pi/6*\X + pi/2)){\X};
		};
		\foreach \X in {2,3}
		{
		\tkzDefPoint(4*cos(pi/6*(\X-2) +3*pi/2),1.5*sin(pi/6*(\X-2) + 3*pi/2)){\X};
		};
		
		\draw[line width=0.7mm ,bend right=20,blue, loosely dotted](-3,-0.5) edge (0);
		
		\draw [line width=0.7mm, mypurple,dash pattern={on 10pt off 2pt on 5pt off 2pt}, bend right=20] (1) edge (3);
		
		\draw[line width=0.7mm ,bend right=10,orange,dotted](0) edge (3);
		
		\foreach \X in {0,1,3}
		{
		\tkzDrawPoints[fill =red,size=4,color=red](\X);
		};
		
		%\foreach \X in {0}
		%{
		%\tkzDrawPoints[rectangle,fill =red,size=6,color=red](\X);
		%};
		
		%\foreach \X in {2}
		%{
		%\tkzDrawPoints[rectangle,size=6,color=dark-green,thick,fill=white](\X);
		%};
		%\foreach \X in {0,2}
		%{
		%\tkzDrawPoints[size=6,orange,osq](\X);
		%};
		%\foreach \X in {23}
		%{
		%\tkzDrawPoints[size=6,darkpink,line width=0.5mm,cross out, draw](\X);
		%};

		\tkzDefPoint(-2.6,0.4){gammaM};
		\tkzLabelPoint[blue](gammaM){\Large $\eta$}
		\tkzDefPoint(0.9,0.7){gammaP};
		\tkzLabelPoint[orange](gammaP){\Large $\kappa_p$}
		%\tkzDefPoint(-0.4,-1.6){v};
		%\tkzLabelPoint[dark-green](v){\Large $v$}
		%\tkzDefPoint(0,1.9){w};
		%\tkzLabelPoint[red](w){\Large $w$}
		\tkzDefPoint(-0.1,0.1){deltav};
		\tkzLabelPoint[mypurple](deltav){\Large $\delta_p$}
		%\tkzDefPoint(0.5,0.7){deltaw};
		%\tkzLabelPoint[mypurple](deltaw){\Large $\eta_{i+1}$}
		\end{scope}
    \end{tikzpicture}
\caption{\label{fig:kerepi} The two types of accordions $\kappa_i$ representing the indecomposable summands of $\Ker(f)$.}
\end{figure}
\end{prop}

\begin{prop}[\cite{BDMTY19,CPS21}]\label{prop:geom_ext}
Let $(\pmb{\Sigma}, \mathcal{M},\Delta^{\gpoint})$ be a  $\gpoint$-dissected marked surface with $\mathcal{M}_{\gpoint} \subset \partial \pmb{\Sigma}$. Let $\delta, \eta \in \Accord$. We distinguish two types of extensions from $\MM(\delta)$ to $\MM(\eta)$ which are:

\begin{enumerate}[label = $\bullet$, itemsep =0.5em]
    \item \new{Overlap extensions}: whenever we have a non-split short exact sequence \[\begin{tikzcd}
	 \MM(\eta) & {E_1 \oplus E_2} & \MM(\delta) 
	\arrow[tail, from=1-1, to=1-2]
	\arrow[two heads, from=1-2, to=1-3]
\end{tikzcd},\] where $E_1,E_2 \in \ind_\mathbb{K}(Q,R)$, then $\delta$ and $\eta$ are intersecting each other, and $\gamma_{(E_1)}$ and $\gamma_{(E_2)}$ can be obtained from $\delta$ and $\eta$ as pictured in \cref{fig:overlapextaccord}. By abuses of notations with strings of $(Q,R)$, we write $\OvExt(\delta,\eta)$ for the union of all sets $\{\gamma_{(E_1)}, \gamma_{(E_2)}\}$ over all the isomorphism classes of short exact sequences of the above shape.
    \begin{figure}[!ht]
\centering 
    \begin{tikzpicture}[mydot/.style={
					circle,
					thick,
					fill=white,
					draw,
					outer sep=0.5pt,
					inner sep=1pt
				}, scale = 1]
		\tikzset{
		osq/.style={
        rectangle,
        thick,
        fill=white,
        append after command={
            node [
                fit=(\tikzlastnode),
                orange,
                line width=0.3mm,
                inner sep=-\pgflinewidth,
                cross out,
                draw
            ] {}}}}
		\draw [line width=0.7mm,domain=50:130] plot ({4*cos(\x)}, {1.5*sin(\x)});
        \draw [line width=0.7mm,domain=230:310] plot ({4*cos(\x)}, {1.5*sin(\x)});
		\foreach \X in {0,1}
		{
		\tkzDefPoint(4*cos(pi/6*\X +pi/2),1.5*sin(pi/6*\X + pi/2)){\X};
		};
		\foreach \X in {2,3}
		{
		\tkzDefPoint(4*cos(pi/6*(\X-2.1) +3*pi/2),1.5*sin(pi/6*(\X-2.1) + 3*pi/2)){\X};
		};
		
		\draw[line width=0.7mm ,bend right=10,orange, dotted](0) edge (2);
		
		\draw[line width=0.7mm ,bend right=15,blue, loosely dotted](1) edge (3);
		
		\draw [line width=0.7mm, mypurple,dash pattern={on 10pt off 2pt on 5pt off 2pt}, bend right=20] (1) edge (0);
		\draw [line width=0.7mm, mypurple,dash pattern={on 10pt off 2pt on 5pt off 2pt}, bend right=20] (3) edge (2);
		\foreach \X in {0,...,3}
		{
		\tkzDrawPoints[circle,fill =red,size=4,color=red](\X);
		};
		
		%\foreach \X in {2}
		%{
		%\tkzDrawPoints[rectangle,size=6,color=dark-green,thick,fill=whit%e](\X);
		%};
		%\foreach \X in {0,2}
		%{
		%\tkzDrawPoints[size=6,orange,osq](\X);
		%};
		%\foreach \X in {23}
		%{
		%\tkzDrawPoints[size=6,darkpink,line width=0.5mm,cross out, draw](\X);
		%};

		\tkzDefPoint(0.1,0.5){gammaM};
		\tkzLabelPoint[orange](gammaM){\Large $\delta$}
		\tkzDefPoint(-1.2,0.2){gammaP};
		\tkzLabelPoint[blue](gammaP){\Large $\eta$}
		\tkzDefPoint(-1.2,1.2){deltav};
		\tkzLabelPoint[mypurple](deltav){\Large $\gamma_{(E_1)}$}
		\tkzDefPoint(0.2,-0.75){deltaw};
		\tkzLabelPoint[mypurple](deltaw){\Large $\gamma_{(E_2)}$}
    \end{tikzpicture}
\caption{\label{fig:overlapextaccord} Illustration of an overlap extension.}
\end{figure}
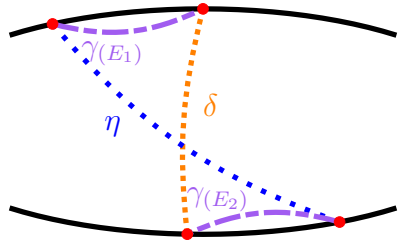
    \item \new{Arrow extension} : 
    whenever we have a non-split short exact sequence  \[\begin{tikzcd}
	 \MM(\eta) & {E} & \MM(\delta) 
	\arrow[tail, from=1-1, to=1-2]
	\arrow[two heads, from=1-2, to=1-3]
\end{tikzcd},\] where $E \in \ind(Q,R)$, then we can construct $\gamma_{(E)}$ from $\delta$ and $\eta$ as described in \cref{fig:arrowextaccord}. Similarly to overlap extensions, by abuses of notations with strings of $(Q,R)$, we write $\ArExt(\delta,\eta)$ for the set of all the accordions $\gamma_{(E)}$ over all the isomorphism classes of short exact sequences of the above shape.
    \begin{figure}[!ht]
\centering 
    \begin{tikzpicture}[mydot/.style={
					circle,
					thick,
					fill=white,
					draw,
					outer sep=0.5pt,
					inner sep=1pt
				}, scale = 1]
		\tikzset{
		osq/.style={
        rectangle,
        thick,
        fill=white,
        append after command={
            node [
                fit=(\tikzlastnode),
                orange,
                line width=0.3mm,
                inner sep=-\pgflinewidth,
                cross out,
                draw
            ] {}}}}
            
        \draw [line width=0.7mm,domain=-10:10] plot ({4*cos(\x)}, {1.5*sin(\x)});
		\draw [line width=0.7mm,domain=20:50] plot ({4*cos(\x)}, {1.5*sin(\x)});
		\draw [line width=0.7mm,domain=80:100] plot ({4*cos(\x)}, {1.5*sin(\x)});
		\draw [line width=0.7mm,domain=110:130] plot ({4*cos(\x)}, {1.5*sin(\x)});
		\draw [line width=0.7mm,domain=170:190] plot ({4*cos(\x)}, {1.5*sin(\x)});
        \draw [line width=0.7mm,domain=250:290] plot ({4*cos(\x)}, {1.5*sin(\x)});
		\foreach \X in {0,1}
		{
		\tkzDefPoint(4*cos(pi/2*(\X-0.3) +pi/3),1.5*sin(pi/2*(\X-0.3) + pi/3)){\X};
		};
		\foreach \X in {2,3}
		{
		\tkzDefPoint(4*cos(pi/6*(\X-3) +3*pi/2),1.5*sin(pi/6*(\X-3) + 3*pi/2)){\X};
		};
		
		\tkzDefPoint(4,0){4};
		\tkzDefPoint(0,1.5){5};
		\tkzDefPoint(-4,0){6};
		\tkzDefPoint(4*cos(3*pi/2-pi/11),1.5*sin(3*pi/2-pi/11)){7};
		\tkzDefPoint(4*cos(3*pi/2+pi/11),1.5*sin(3*pi/2+pi/11)){8};
		
		\filldraw[gray,opacity=0.1] (4) to [bend left=20] (5) to [bend left=20] (6) to [bend left=10] (7) to [bend right=10] (8) to [bend left=10] cycle;
		
		\draw[line width=0.5mm,bend left = 20,dark-green] (4) to (5) to (6);
		\draw[line width=0.5mm,bend left = 10,dark-green,dashed] (6) to (7);
		\draw[line width=0.5mm,bend left = 10,dark-green,dashed] (8) to (4);
		
		\draw[line width=0.7mm ,bend right=10,orange, dotted](0) edge (3);
		
		\draw[line width=0.7mm ,bend left=10,blue, loosely dotted](1) edge (3);
		
		\draw [line width=0.7mm, mypurple,dash pattern={on 10pt off 2pt on 5pt off 2pt}, bend left=20] (0) edge (1);

		\foreach \X in {0,1,3}
		{
		\tkzDrawPoints[circle,fill =red,size=4,color=red](\X);
		};
		
		\foreach \X in {4,...,8}
		{
		\tkzDrawPoints[mydot,size=6,color=dark-green,thick,fill=white](\X);
		};
		%\foreach \X in {0,2}
		%{
		%\tkzDrawPoints[size=6,orange,osq](\X);
		%};
		%\foreach \X in {23}
		%{
		%\tkzDrawPoints[size=6,darkpink,line width=0.5mm,cross out, draw](\X);
		%};

		\tkzDefPoint(0.9,0.1){gammaM};
		\tkzLabelPoint[orange](gammaM){\Large $\delta$}
		\tkzDefPoint(-1.5,0.3){gammaP};
		\tkzLabelPoint[blue](gammaP){\Large $\eta$}
		\tkzDefPoint(-0.1,1.2){deltav};
		\tkzLabelPoint[mypurple](deltav){\Large $\gamma_{(E)}$}
    \end{tikzpicture}
\caption{\label{fig:arrowextaccord} Illustration of an arrow extension. The shaded part corresponds to one cell of $\pmb{\Gamma}(\Delta^{\gpoint})$.}
\end{figure}
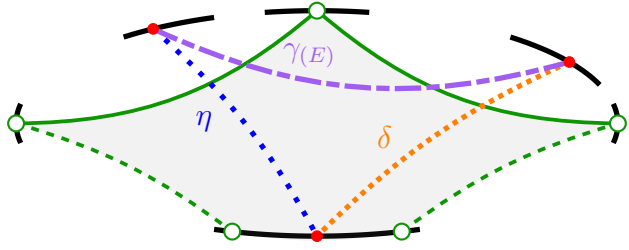
\end{enumerate}
\end{prop}
\begin{prop}[\cite{CPS21},\cite{BDMTY19}]
Let $(Q,R)$ be a gentle tree. All the extensions between indecomposable representations are either overlap extensions or arrow extensions.
\end{prop}
As we will use the geometric model to compute resolving subcategories and substantially use minimal projective resolutions, we naturally assume that $\mathbb{K}Q/\langle R \rangle$ has a finite global dimension. It implies that $\mathcal{M} \subset \partial \pmb{\Sigma}$. The $\rpoint$-dissection $\Prj(\Delta^{\gpoint})$ given by the projective accordions is homeomorphic to $\Delta^{\gpoint}$. To handle in a easier way the computations of resolving subcategories, we decide to work on $(\pmb{\Sigma}, \mathcal{M}, \Prj(\Delta^{\gpoint}))$ rather than on $\Surf(Q,R)$.

We recall a result that allows one to read an accordion of $(\pmb{\Sigma}, \mathcal{M}, \Delta^{\gpoint})$ directly in the projective dissection $\Prj(\Delta^{\gpoint})$. 

Before that, we recall a notation. Given $C \in \pmb{\Gamma}(\Prj(\Delta^{\gpoint}))$, we can order $\rpoint$-points in $\partial C$ by traveling in the counter-clockwise direction on the boundary of $C$ from the unique $\gpoint \in \partial C$. We denote by $m_C$ the last $\rpoint$-point in $ \partial C$ we go through. 

\begin{prop} \label{prop:arcsaccordions} Let $(\pmb{\Sigma}, \mathcal{M}, \Delta^{\gpoint})$ be a $\gpoint$-dissected marked surface such that $M \subset \partial \pmb{\Sigma}$. A $\rpoint$-arc $\delta$ is an accordion if and only if either $\delta \in \Prj(\Delta^{\gpoint})$ or the following assertions hold: 
\begin{enumerate}[label=$(\alph*)$, itemsep=1mm]
    \item whenever $\delta$ enters a cell $C$ of $\Prj(\Delta^{\gpoint})$ by crossing an $\rpoint$-arc $\mu$,
    \begin{enumerate}[label=$(a \arabic*)$, itemsep=1mm]
        \item if it leaves $C$, it leaves it by crossing an $\rpoint$-arc $\nu$ adjacent to $\mu$;
        \item otherwise, its endpoint in $\partial C$ is either $m_C$ or the non-common endpoint of an $\rpoint$-arc adjacent to $\mu$, or 
    \end{enumerate}
    \item if $\delta$ is contained in a cell $C$, then:
    \begin{enumerate}[label=$(b \arabic*)$, itemsep=1mm]
        \item its endpoints are the distinct endpoints of a pair of adjacent projective accordions; or,
        \item one of its endpoints must be $m_C$. 
    \end{enumerate}
\end{enumerate}
\end{prop}

\label{order:clock} As $\mathcal{M}_{\rpoint} \subset \partial \Sigma$, for any $v \in \mathcal{M}_{\rpoint}$, we can consider a total order $\preccurlyeq_v$ on the subset of accordions of $(\pmb{\Sigma}, \mathcal{M},\Prj(\Delta^{\gpoint}))$ as follows. For $(\delta,\eta) \in \Accord^2$ such that $v$ is their common endpoint, we write $\eta \preccurlyeq_v \delta$ whenever $\eta$ follows $\delta$ with respect to the counter-clockwise orientation around $v$.

For any accordion, we also introduced the notion of neighboring projective accordions of any non-projective accordion $\delta$ to be the projective accordions $\eta$ such that either $\MM(\eta)$ appears as a summand of a projective representation in the minimal projective resolution of $\MM(\delta)$, or $\MM(\delta)$ and $\MM(\eta)$ are such that there exists $i>0$ such that $\Ext^i(\MM(\delta),\MM(\eta))\neq0$. Denote by $\NP(\delta)$ the set of neighboring projective accordions of $\delta$. We can characterize them combinatorially.

\begin{prop}[\cite{BCS21}\cite{OPS18},\cite{LGH24}]\label{prop:CombidescripNproj}
Let $(Q,R)$ be a representation-finite gentle quiver. Let $\delta$ be an accordion of $\Surf(Q,R)$. Then $\NP(\delta)$ is the subset of $\Prj(\Delta^{\gpoint})$ made of all accordions $\eta$ satisfying at least one of the following conditions:
\begin{enumerate}[label=$(\roman*)$,itemsep=0.2em]
    \item the curve $\eta$ crosses $\delta$: they correspond to projective accordions such that $\OvExt(\delta,eta) \neq \varnothing$;
    \item the curve $\eta$ is part of the border of a cell crossed by $\delta$ containing one of its endpoints and is a part of the path of projective curves that links the last projective crossed to the endpoint of $\delta$: they correspond to projective accordions such that $\MM(\eta)$ is a summand of a projective representation in the minimal projective resolution of $\MM(\delta)$; or,
    \item the curve $\eta$ and $\delta$ have a common endpoint $v$, and, at the vertex $v$, all the projective accordions are smaller than $\delta$ with respect to $\preccurlyeq_v$: they are projective accordions such that either $\ArExt(\delta,\eta) \neq \varnothing$, or $\Ext^i(\MM(\delta), \MM(\eta)) \neq 0$, for some $i > 1$.
\end{enumerate}
\end{prop}

\begin{ex} \label{ex:Nproj1} Let $(Q,R)$ be the gentle tree of \cref{fig:treeex}.
\begin{figure}[!ht]
    \centering
     \scalebox{0.6}{\begin{tikzpicture}[->]
		\node (a) at (0,0) {$1$};
		\node (b) at (1,0) {$2$};
		\node (c) at (2,0) {$3$};
		\node (d) at (3,1) {$4$};
		\node (e) at (4,1) {$5$};
		\node (f) at (5,1) {$6$};
		\node (g) at (6,0) {$7$};
		\node (h) at (5,-1) {$8$};
		\node (i) at (7,0) {$9$};
		\node (j) at (8,0) {$10$};
		\node (k) at (9,0) {$11$};
		\node (l) at (10,0) {$12$};
		\node (m) at (11,-1) {$13$};
		\node (n) at (12,-1) {$14$};
		\node (o) at (13,0){$15$};
		\node (p) at (14,-1){$16$};
		\node (q) at (14,1){$17$};
		\node (r) at (15,1){$18$};
		
		\draw (a) -- (b);
		\draw (b) -- (c);
		\draw ([yshift=-1mm]d.west)--([yshift=1mm]c.east);
		\draw (f) -- (e);
		\draw (e) -- (d);
		\draw ([yshift=1mm]g.west)--([yshift=-1mm]f.east);
		\draw ([yshift=1mm]h.east) -- ([yshift=-1mm]g.west);
		\draw (i) to (g) ;
		\draw (i) -- (j) ;
		\draw (j) -- (k);
		\draw (l) to (k);
		\draw[<-] ([yshift=1mm]m.west) -- ([yshift=-1mm]l.east);
		\draw[<-] (n) to (m);
		\draw[<-]([yshift=-1mm]o.west) -- ([yshift=1mm]n.east);
		\draw ([yshift=-1mm]q.west)--([yshift=1mm]o.east);
		\draw[<-] ([yshift=1mm]p.west)--([yshift=-1mm]o.east);
		\draw (q)--(r);

		\draw[dashed,-] ([xshift=-.3cm,yshift=-.15cm]d.south) arc[start angle = -110, end angle = -15, x radius=.6cm, y radius =.6cm];
		\draw[dashed,-] ([xshift=-.3cm,yshift=.1cm]e.south) arc[start angle = -165, end angle = -15, x radius=.3cm, y radius =.3cm];
		\draw[dashed,-] ([xshift=-.5cm,yshift=0.2cm]f.south) arc[start angle = -165, end angle = -60, x radius=.6cm, y radius =.6cm];
		\draw[dashed,-] ([xshift=-.2cm,yshift=-.3cm]g.west) arc[start angle = 225, end angle = 135, x radius=.5cm, y radius =.5cm];
		\draw[dashed,-] ([xshift=.3cm,yshift=.1cm]n.north) arc[start angle = 45, end angle = 190, x radius=.4cm, y radius =.4cm];
		\draw[dashed,-] ([xshift=.3cm,yshift=-.15cm]m.north) arc[start angle = 10, end angle = 110, x radius=.5cm, y radius =.5cm];
		\draw[dashed,-] ([xshift=.2cm,yshift=-.3cm]o.east) arc[start angle = -45, end angle = 45, x radius=.5cm, y radius =.5cm];
		\end {tikzpicture}}
        \caption{ \label{fig:treeex} An example of a gentle quiver, which is a gentle tree (see \cref{def:gentle trees})}
\end{figure}
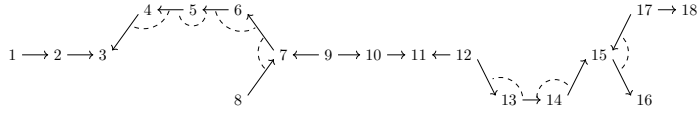
In \cref{fig:Nprojex1}, we represent the projective dissection as a $\rpoint$-dissection of a marked disc (drawn as thick red curves), corresponding to $\Surf(Q,R)$. We consider an accordion $\delta \in \Accord'$ (drawn as a dotted blue curve). The densely dotted thick red lines correspond to the accordions forming the set of its neighboring projective accordions $\NP(\delta)$. \qedhere
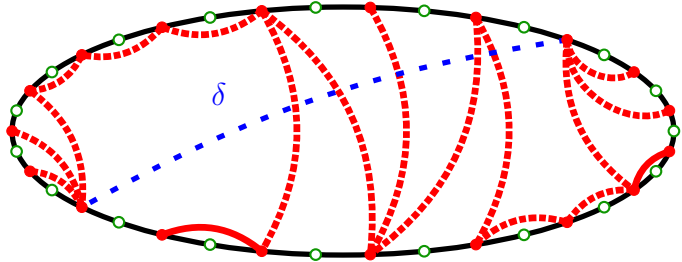
\begin{figure}[!ht]
\centering 
    \begin{tikzpicture}[mydot/.style={
					circle,
					thick,
					fill=white,
					draw,
					outer sep=0.5pt,
					inner sep=1pt
				}, scale = 1.1]
		\draw[line width=0.7mm,black] (0,0) ellipse (4cm and 1.5cm);
		\foreach \X in {0,1,...,37}
		{
		\tkzDefPoint(4*cos(pi/19*\X),1.5*sin(pi/19*\X)){\X};
		};
		%\draw[line width=0.5mm ,bend left =60,dark-green](0) edge (2);
		%\draw[line width=0.5mm ,bend left =60,dark-green](2) edge (4);
		%\draw[line width=0.5mm ,bend right =40,dark-green](4) edge (22);
		%\draw[line width=0.5mm ,bend left =40,dark-green](4) edge (6);
		%\draw[line width=0.5mm ,bend left =40,dark-green](6) edge (18);
		%\draw[line width=0.5mm ,bend left =40,dark-green](18) edge (20);
		%\draw[line width=0.5mm ,bend left =40,dark-green](8) edge (18);
		%\draw[line width=0.5mm ,bend left =40,dark-green](8) edge (16);
		%\draw[line width=0.5mm ,bend left =40,dark-green](8) edge (14);
		%\draw[line width=0.5mm ,bend left =40,dark-green](8) edge (10);
		%\draw[line width=0.5mm ,bend left=40,dark-green](12) edge (14);
		
		\draw[line width=0.9mm ,bend left =60,red, densely dashdotted](1) edge (5);
		\draw[line width=0.9mm ,bend left =60,red, densely dashdotted](3) edge (5);
		\draw[line width=0.9mm ,bend right =30,red, densely dashdotted](5) edge (35);
		\draw[line width=0.9mm ,bend left =30,red,densely dashdotted](33) edge (35);
		\draw[line width=0.9mm ,bend left =30,red](35) edge (37);
		\draw[line width=0.9mm ,bend left =30,red, densely dashdotted](31) edge (33);
		\draw[line width=0.9mm ,bend left =30,red, densely dashdotted](7) edge (31);
		\draw[line width=0.9mm ,bend left =30,red, densely dashdotted](7) edge (29);
		\draw[line width=0.9mm ,bend left =30,red, densely dashdotted](9) edge (29);
		\draw[line width=0.9mm ,bend left =30,red, densely dashdotted](11) edge (29);
		\draw[line width=0.9mm ,bend left =30,red, densely dashdotted](11) edge (27);
		\draw[line width=0.9mm ,bend right =30,red, densely dashdotted](13) edge (11);
		\draw[line width=0.9mm ,bend left =30,red, densely dashdotted](13) edge (15);
		\draw[line width=0.9mm ,bend left =30,red, densely dashdotted](17) edge (23);
		\draw[line width=0.9mm ,bend left =30,red, densely dashdotted](15) edge (17);
		\draw[line width=0.9mm ,bend left =30,red, densely dashdotted](19) edge (23);
		\draw[line width=0.9mm ,bend left =30,red,densely dashdotted](21) edge (23);
		\draw[line width=0.9mm ,bend left =30,red](25) edge (27);

		\draw[line width=0.7mm ,bend right=10,blue, loosely dashed](5) edge (23);
		\foreach \X in {0,2,...,36}
		{
		\tkzDrawPoint[size=4,color=dark-green,mydot](\X);
		};
		\foreach \X in {1,3,...,37}
		{
		\tkzDrawPoints[fill =red,size=4,color=red](\X);
		};
		\tkzDefPoint(-1.5,0.7){gamma};
		\tkzLabelPoint[blue](gamma){\Large $\delta$}
    \end{tikzpicture}
\caption{\label{fig:Nprojex1} An example of $\NP(\delta)$.}
\end{figure}
\end{ex}

\subsection{Gentle trees and resolving poset}
\label{ss:recalltrees}
We focus on a specific subfamily of gentle quivers.

\begin{definition} \label{def:gentle trees}
A \new{gentle tree} is a gentle quiver $(Q,R)$ such that $Q$ is a tree.
\end{definition}

If $(Q,R)$ is a gentle tree, then, by setting $\Surf(Q,R) = (\pmb{\Sigma}, \mathcal{M}, \Delta^{\gpoint})$, we have that $\mathcal{M} \subset \partial \pmb{\Sigma}$, and $\pmb{\Sigma}$ is homeomorphic to a disc. Moreover, indecomposable representations are characterized by their vertex support, which is equivalent to saying that accordions on $\Surf(Q,R)$ are characterized by the $\gpoint$-arcs of $\Delta^{\gpoint}$ they cross. In addition, homomorphism spaces between indecomposable representations are at most one-dimensional, which means that two arbitrary accordions can cross at most once up to homotopy.

\begin{conv}
In a gentle tree, we identify the indecomposable representations with their vertex support, and, for calculating examples, we represent them using conventions for strings as in \cite{PPP18}. We choose not to represent the arrows and only keep the sequence of vertices that the string passes through, as there is no ambiguity about the arrows that we have to keep track of.
\end{conv}

We can state that extension spaces between indecomposable objects are also of dimension at most one. Therefore, for all the non-split exact sequences \[\begin{tikzcd}
	X & E & Y
	\arrow["k",tail, from=1-1, to=1-2]
	\arrow["f",two heads, from=1-2, to=1-3]
\end{tikzcd}\] where $X,Y \in \ind(Q,R)$, the representation $E$ is precisely given either by an arrow extension or by an overlap extension. 

In this setting, we also prove a simpler criterion that allows us to check if a subcategory is resolving by focusing only on its indecomposable objects.

\begin{theorem}
\label{thm:equivres}
Let $(Q,R)$ be a gentle tree. Let $\mathscr{C}$ be an additive subcategory of $\rep(Q,R)$. Then $\mathscr{D}$ is resolving if and only if $\mathscr{D}$ satisfies the following assertions:
\begin{enumerate}[label=$(\mathsf{Res} \arabic*)$, itemsep=1mm]
    \item \label{Res1} $\proj(Q,R) \subset \mathscr{D}$;
    \item \label{Res2} for any short exact sequence \[\begin{tikzcd}
	X & E & Y
	\arrow["k",tail, from=1-1, to=1-2]
	\arrow["f",two heads, from=1-2, to=1-3]
\end{tikzcd},\] if $X,Y \in \ind(\mathscr{D})$, then $E \in \mathscr{D}$; and,
    \item \label{Res3} $\mathscr{D}$ is closed under syzygies of indecomposable representations.
\end{enumerate}
\end{theorem}

 \label{order:resindec} We can also note that its Auslander--Reiten quiver $\AR(Q,R)$ is acyclic. This allows us to endow $\pmb{\ind \setminus \proj}(Q, R)$ with an order relation defined by \[\forall X,Y \in \pmb{\ind \setminus \proj}(Q,R),\ X \Resleq Y \Longleftrightarrow \Res(X) \subseteq \Res(Y).\] By comparing the finite complete lattice $(\ResOrd(Q,R), \subseteq)$ of resolving subcategories of $(Q,R)$ with the ideals of $(\pmb{\ind \setminus \proj}(Q,R), \Resleq)$, we get an injective map from the former to the latter. This fact allows us to show the result about resolving subcategories that can be realized as the resolving closure of a non-projective indecomposable representation of $(Q,R)$. We call these resolving subcategories \new{monogeneous} resolving subcategories.

\begin{theorem}[\cite{DS251}] \label{thm:Monoareallthejoinirred} Let $(Q,R)$ be a gentle tree. Then, the join-irreducible elements in $(\ResOrd(Q,R), \subseteq)$ are exactly given by the monogeneous resolving subcategories.
\end{theorem}

These results motivated us to describe all the monogeneous resolving subcategories. 

\subsection{Monogeneous resolving subcategories}
\label{ss:RecallMono}

Any set $\mathfrak{B} \subset \Accord$ is said to be \new{resolving} whenever the category additively generated by $\{\MM(\delta) \mid \delta \in \mathfrak{B}\}$ is a resolving subcategory of $\rep(Q,R)$. Such a resolving set is said to be \new{monogeneous} if the associated resolving subcategory is also monogenenous. In this section, we recall the geometric description of all the monogeneous resolving sets of $\Accord$.

To do so, we use our variant of the geometric model, and we first define a suitable coloration of $\NP(\delta)_0$ relative to a given $\delta \in \Accord'$.

\begin{definition}
\label{def:colourendpoints}
Let $(Q,R)$ be a gentle tree and $\Surf(Q,R) = (\pmb{\Sigma}, \mathcal{M}, \Delta^{\gpoint})$. Let $\delta \in \Accord$. We define a \new{coloration} of $\NP(\delta)_0$ to be a partition of $\NP(\delta)_0$ as  follows :
\begin{enumerate}[label=\arabic*),itemsep=1mm]
    \item  As $\delta$ cuts the disc into two connected parts,  on one side of the curve, points of $\NP(\delta)_0$ are colored red $\rsquare$; on the other side, they are green $\gsquare$. In the following, we arbitrarily consider the red dots above $\delta$ and the green ones below.
    \item The left endpoint of the curve is called the \emph{source} and is colored red $\rsquare$ if all the accordions in $\NP(\delta)$ sharing the same endpoint admit a green coloration $\gsquare$ at the other end. They are colored orange $\osquare$ otherwise.
    \item The right endpoint of the curve is called the \emph{target}, and is colored dually.
    \item In cells larger than a triangle, the coloration is changed: all the intermediate points except the second to last are colored orange $\osquare$.
    \item The last intermediate point on the boundary of a large cell is finally colored in pink $\psquare$.
\end{enumerate}
We denote by:
\begin{enumerate}[label=$\bullet$,itemsep=1mm]
    \item $\NP(\delta)_{0}^{{\rsquare}}$ the set of points in $\NP(\delta)_0$ colored in red $\rsquare$;
    \item $\NP(\delta)_{0}^{{\osquare}}$ the set of points in $\NP(\delta)_0$ colored in orange $\osquare$;
    \item $\NP(\delta)_{0}^{{\gsquare}}$ the set of points in $\NP(\delta)_0$ colored in green $\gsquare$;
    \item $\NP(\delta)_{0}^{{\psquare}}$ the set of points in $\NP(\delta)_0$ colored in pink $\psquare$.
\end{enumerate}
We can encode the choice of such a coloration as a map $\col_\delta : \NP(\delta)_0 \longrightarrow \{{\rsquare},{\gsquare},{\osquare},{\psquare}\}$.
\end{definition}

\begin{lemma}[\cite{DS251}] \label{lem:uniqcolandlaterality}
Let $(Q,R)$ be a gentle tree and $\Surf(Q,R) = (\pmb{\Sigma}, \mathcal{M}, \Delta^{\gpoint})$. Let $\delta \in \Accord$. Then there exists a unique coloration of $\NP(\delta)_0$ up to exchanging the ${\rsquare}$ and ${\gsquare}$ points. 
\end{lemma}

In the following, if $\col_\delta : \NP(\delta)_0 \longrightarrow \{{\rsquare},{\gsquare},{\osquare},{\psquare}\}$ is a coloration of $\NP(\delta)_0$, we denote by $\overline{\col_\delta}$ the obtained coloration by exchanging the ${\rsquare}$ and ${\gsquare}$ points, and call it its \new{conjugate coloration}.

\begin{remark} \label{rem:laterality}
A coloration of $\NP(\delta)_0$ induces a laterality of the surface. From now on, for any accordion $\varsigma \in \Accord$, we consider that its source $s(\varsigma)$ is always on the left of its target $t(\varsigma)$, and thus ${\rsquare}$ above and ${\gsquare}$ below.
\end{remark}

\begin{ex} We consider the \cref{ex:Nproj1}.
We give in \cref{fig:Colorex} a coloration of $\NP(\delta)_0$ . Note that the points $\rpoint$ are points in $\mathcal{M}_{\rpoint} \setminus \NP(\delta)_0$.
    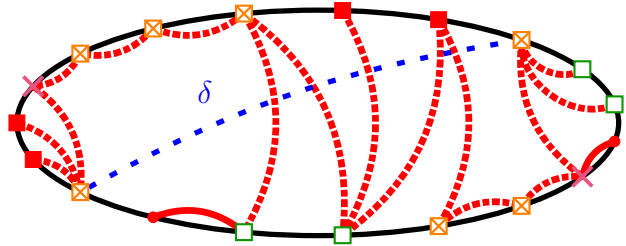
\begin{figure}[!ht]
\centering 
    \begin{tikzpicture}[mydot/.style={
					circle,
					thick,
					fill=white,
					draw,
					outer sep=0.5pt,
					inner sep=1pt
				}, scale = 1]
		\tikzset{
		osq/.style={
        rectangle,
        thick,
        fill=white,
        append after command={
            node [
                fit=(\tikzlastnode),
                orange,
                line width=0.3mm,
                inner sep=-\pgflinewidth,
                cross out,
                draw
            ] {}}}}
		\draw[line width=0.7mm,black] (0,0) ellipse (4cm and 1.5cm);
			\foreach \X in {0,1,...,37}
		{
		\tkzDefPoint(4*cos(pi/19*\X),1.5*sin(pi/19*\X)){\X};
		};
		%\draw[line width=0.5mm ,bend left =60,dark-green](0) edge (2);
		%\draw[line width=0.5mm ,bend left =60,dark-green](2) edge (4);
		%\draw[line width=0.5mm ,bend right =40,dark-green](4) edge (22);
		%\draw[line width=0.5mm ,bend left =40,dark-green](4) edge (6);
		%\draw[line width=0.5mm ,bend left =40,dark-green](6) edge (18);
		%\draw[line width=0.5mm ,bend left =40,dark-green](18) edge (20);
		%\draw[line width=0.5mm ,bend left =40,dark-green](8) edge (18);
		%\draw[line width=0.5mm ,bend left =40,dark-green](8) edge (16);
		%\draw[line width=0.5mm ,bend left =40,dark-green](8) edge (14);
		%\draw[line width=0.5mm ,bend left =40,dark-green](8) edge (10);
		%\draw[line width=0.5mm ,bend left=40,dark-green](12) edge (14);
		
		\draw[line width=0.9mm ,bend left =60,red, densely dashdotted](1) edge (5);
		\draw[line width=0.9mm ,bend left =60,red, densely dashdotted](3) edge (5);
		\draw[line width=0.9mm ,bend right =30,red, densely dashdotted](5) edge (35);
		\draw[line width=0.9mm ,bend left =30,red,densely dashdotted](33) edge (35);
		\draw[line width=0.9mm ,bend left =30,red](35) edge (37);
		\draw[line width=0.9mm ,bend left =30,red, densely dashdotted](31) edge (33);
		\draw[line width=0.9mm ,bend left =30,red, densely dashdotted](7) edge (31);
		\draw[line width=0.9mm ,bend left =30,red, densely dashdotted](7) edge (29);
		\draw[line width=0.9mm ,bend left =30,red, densely dashdotted](9) edge (29);
		\draw[line width=0.9mm ,bend left =30,red, densely dashdotted](11) edge (29);
		\draw[line width=0.9mm ,bend left =30,red, densely dashdotted](11) edge (27);
		\draw[line width=0.9mm ,bend right =30,red, densely dashdotted](13) edge (11);
		\draw[line width=0.9mm ,bend left =30,red, densely dashdotted](13) edge (15);
		\draw[line width=0.9mm ,bend left =30,red, densely dashdotted](17) edge (23);
		\draw[line width=0.9mm ,bend left =30,red, densely dashdotted](15) edge (17);
		\draw[line width=0.9mm ,bend left =30,red, densely dashdotted](19) edge (23);
		\draw[line width=0.9mm ,bend left =30,red, densely dashdotted](21) edge (23);
		\draw[line width=0.9mm ,bend left =30,red](25) edge (27);

		\draw[line width=0.7mm ,bend right=10,blue, loosely dashed](5) edge (23);
		%\foreach \X in {0,2,...,22}
		%{
		%\tkzDrawPoint[size=4,color=dark-green,mydot](\X);
		%};
		\foreach \X in {25,37}
		{
		\tkzDrawPoints[fill =red,size=4,color=red](\X);
		};
		\foreach \X in {7,9,19,21}
		{
		\tkzDrawPoints[rectangle,fill =red,size=6,color=red](\X);
		};
		
		\foreach \X in {1,3,27,29}
		{
		\tkzDrawPoints[rectangle,size=6,color=dark-green,thick,fill=white](\X);
		};
		\foreach \X in {5,11,13,15,23,31,33}
		{
		\tkzDrawPoints[size=6,orange,osq](\X);
		};
		\foreach \X in {17,35}
		{
		\tkzDrawPoints[size=6,darkpink,line width=0.5mm,cross out, draw](\X);
		};
		
		\tkzDefPoint(-1.5,0.7){gamma};
		\tkzLabelPoint[blue](gamma){\Large $\delta$}
    \end{tikzpicture}
\caption{\label{fig:Colorex} An example of a coloration of $\NP(\delta)_0$.}
\end{figure}
\end{ex}

Given $\delta \in \Accord'$, this coloration allows us to define a precise collection of accordions, which is the one corresponding to the indecomposable representations additively generating $\Res(\MM(\delta))$.

\begin{definition} \label{def:geometricres}
Let $(Q,R)$ be a gentle tree, and $\Surf(Q,R) = (\pmb{\Sigma}, \mathcal{M}, \Delta^{\gpoint})$. For any $\delta \in \Accord$, we define the \new{monogeneous geometric resolving set} of $\delta$ as the set $\opResAc(\delta) = \opResAc'(\delta) \cup \Prj(\Delta^{\gpoint})$ where: \[\opResAc'(\delta) = \{ \eta \in \Accord \mid s(\eta) \in \NP(\delta)_0^{{\rsquare}} \cup \NP(\delta)_0^{{\osquare}}, t(\eta) \in \NP(\delta)_0^{{\gsquare}} \cup \NP(\delta)_0^{{\osquare}} \}.\]
The \new{monogeneous geometric resolving subcategory} of $\delta$, denoted by \new{$\mathscr{U}_\delta$}, is the additive subcategory of $\rep(Q,R)$ generated by $\{\MM(\varsigma) \mid \varsigma \in \opResAc(\delta)\}$.
\end{definition}

\begin{ex} \label{ex:ResM}  Let $(\pmb{\Sigma}, \mathcal{M}, \Delta^{\gpoint})$ be the $\gpoint$-dissection of the marked surface as seen in \cref{ex:Nproj1} and in \cref{fig:Colorex}.  Consider $\delta \in \Accord$ as previously. In \cref{fig:exResM}, we represent the accordions in $\opResAc(\delta)$ with dotted and dashed lines. More precisely,
\begin{enumerate}[label=$\bullet$,itemsep=1mm]
    \item the blue loosely-dotted line is the accordion $\delta$; 
    \item the red densely dotted lines are the neighboring projective accordions of $\delta$ as seen in \cref{ex:Nproj1} and in \cref{fig:Colorex}; and,
    \item the purple dash-dotted lines are accordions in $\opResAc'(\delta)$. \qedhere
\end{enumerate}  
     \begin{figure}[!ht]
\centering 
    \begin{tikzpicture}[mydot/.style={
					circle,
					thick,
					fill=white,
					draw,
					outer sep=0.5pt,
					inner sep=1pt
				}, scale = 1.2]
		\tikzset{
		osq/.style={
        rectangle,
        thick,
        fill=white,
        append after command={
            node [
                fit=(\tikzlastnode),
                orange,
                line width=0.3mm,
                inner sep=-\pgflinewidth,
                cross out,
                draw
            ] {}}}}
		\draw[line width=0.7mm,black] (0,0) ellipse (4cm and 1.5cm);
			\foreach \X in {0,1,...,37}
		{
		\tkzDefPoint(4*cos(pi/19*\X),1.5*sin(pi/19*\X)){\X};
		};
		%\draw[line width=0.5mm ,bend left =60,dark-green](0) edge (2);
		%\draw[line width=0.5mm ,bend left =60,dark-green](2) edge (4);
		%\draw[line width=0.5mm ,bend right =40,dark-green](4) edge (22);
		%\draw[line width=0.5mm ,bend left =40,dark-green](4) edge (6);
		%\draw[line width=0.5mm ,bend left =40,dark-green](6) edge (18);
		%\draw[line width=0.5mm ,bend left =40,dark-green](18) edge (20);
		%\draw[line width=0.5mm ,bend left =40,dark-green](8) edge (18);
		%\draw[line width=0.5mm ,bend left =40,dark-green](8) edge (16);
		%\draw[line width=0.5mm ,bend left =40,dark-green](8) edge (14);
		%\draw[line width=0.5mm ,bend left =40,dark-green](8) edge (10);
		%\draw[line width=0.5mm ,bend left=40,dark-green](12) edge (14);
		
		\draw[line width=0.9mm ,bend left =60,red, densely dashdotted](1) edge (5);
		\draw[line width=0.9mm ,bend left =60,red,densely dashdotted](3) edge (5);
		\draw[line width=0.9mm ,bend right =30,red, densely dashdotted](5) edge (35);
		\draw[line width=0.9mm ,bend left =30,red,densely dashdotted](33) edge (35);
		\draw[line width=0.9mm ,bend left =30,red](35) edge (37);
		\draw[line width=0.9mm ,bend left =30,red, densely dashdotted](31) edge (33);
		\draw[line width=0.9mm ,bend left =30,red, densely dashdotted](7) edge (31);
		\draw[line width=0.9mm ,bend left =30,red, densely dashdotted](7) edge (29);
		\draw[line width=0.9mm ,bend left =30,red, densely dashdotted](9) edge (29);
		\draw[line width=0.9mm ,bend left =30,red, densely dashdotted](11) edge (29);
		\draw[line width=0.9mm ,bend left =30,red, densely dashdotted](11) edge (27);
		\draw[line width=0.9mm ,bend right =30,red, densely dashdotted](13) edge (11);
		\draw[line width=0.9mm ,bend left =30,red, densely dashdotted](13) edge (15);
		\draw[line width=0.9mm ,bend left =30,red, densely dashdotted](17) edge (23);
		\draw[line width=0.9mm ,bend left =30,red, densely dashdotted](15) edge (17);
		\draw[line width=0.9mm ,bend left =30,red, densely dashdotted](19) edge (23);
		\draw[line width=0.9mm ,bend left =30,red,densely dashdotted](21) edge (23);
		\draw[line width=0.9mm ,bend left =30,red](25) edge (27);
		
		\draw[line width=0.7mm,blue, loosely dashed](5) edge (23);
		
		\draw [line width=0.7mm, mypurple,dash pattern={on 10pt off 2pt on 5pt off 2pt}, bend right=10] (5) edge (33);
		\draw [line width=0.7mm, mypurple,dash pattern={on 10pt off 2pt on 5pt off 2pt}, bend right=-10] (5) edge (31);
		\draw [line width=0.7mm, mypurple,dash pattern={on 10pt off 2pt on 5pt off 2pt}, bend right=10] (1) edge (33);
		\draw [line width=0.7mm, mypurple,dash pattern={on 10pt off 2pt on 5pt off 2pt}, bend right=10] (3) edge (33);
		\draw [line width=0.7mm, mypurple,dash pattern={on 10pt off 2pt on 5pt off 2pt}, bend left=10] (5) edge (7);
		\draw [line width=0.7mm, mypurple,dash pattern={on 10pt off 2pt on 5pt off 2pt}, bend left=10] (5) edge (9);
		\draw [line width=0.7mm, mypurple,dash pattern={on 10pt off 2pt on 5pt off 2pt}, bend left=10] (5) edge (11);
		\draw [line width=0.7mm, mypurple,dash pattern={on 10pt off 2pt on 5pt off 2pt}, bend right=10] (23) edge (11);
		\draw [line width=0.7mm, mypurple,dash pattern={on 10pt off 2pt on 5pt off 2pt}, bend right=10] (23) edge (13);
		\draw [line width=0.7mm, mypurple,dash pattern={on 10pt off 2pt on 5pt off 2pt}, bend right=10] (23) edge (15);
		%\draw [line width=0.7mm, mypurple,dash pattern={on 10pt off 2pt on 5pt off 2pt}, bend right=10] (23) edge (9);
		%\draw [line width=0.7mm, mypurple,dash pattern={on 10pt off 2pt on 5pt off 2pt}, bend right=5] (23) edge (7);
		\draw [line width=0.7mm, mypurple,dash pattern={on 10pt off 2pt on 5pt off 2pt}, bend right=20] (19) edge (15);
		\draw [line width=0.7mm, mypurple,dash pattern={on 10pt off 2pt on 5pt off 2pt}, bend right=20] (21) edge (15);
		%\draw [line width=0.7mm, mypurple,dash pattern={on 10pt off 2pt on 5pt off 2pt}, bend right=10] (27) edge (9);
		%\draw [line width=0.7mm, mypurple,dash pattern={on 10pt off 2pt on 5pt off 2pt}, bend right=10] (27) edge (7);
		\draw [line width=0.7mm, mypurple,dash pattern={on 10pt off 2pt on 5pt off 2pt}, bend left=20] (23) edge (27);
		\draw [line width=0.7mm, mypurple,dash pattern={on 10pt off 2pt on 5pt off 2pt}, bend left=15] (23) edge (29);
		\draw [line width=0.7mm, mypurple,dash pattern={on 10pt off 2pt on 5pt off 2pt}, bend left=15] (23) edge (31);
		\draw [line width=0.7mm, mypurple,dash pattern={on 10pt off 2pt on 5pt off 2pt}, bend left=10] (11) edge (31);
		\draw [line width=0.7mm, mypurple,dash pattern={on 10pt off 2pt on 5pt off 2pt}, bend left=10] (9) edge (31);
		%\draw [line width=0.7mm, mypurple,dash pattern={on 10pt off 2pt on 5pt off 2pt}, bend right=10] (11) edge (9);
		%\draw [line width=0.7mm, mypurple,dash pattern={on 10pt off 2pt on 5pt off 2pt}, bend right=10] (11) edge (7);
		%\draw [line width=0.7mm, mypurple,dash pattern={on 10pt off 2pt on 5pt off 2pt}, bend right=10] (31) edge (11);
		%\draw [line width=0.7mm, mypurple,dash pattern={on 10pt off 2pt on 5pt off 2pt}, bend right=10] (31) edge (9);

		\foreach \X in {25,37}
		{
		\tkzDrawPoints[fill =red,size=4,color=red](\X);
		};
		\foreach \X in {7,9,19,21}
		{
		\tkzDrawPoints[rectangle,fill =red,size=6,color=red](\X);
		};
		
		\foreach \X in {1,3,27,29}
		{
		\tkzDrawPoints[rectangle,size=6,color=dark-green,thick,fill=white](\X);
		};
		\foreach \X in {5,11,13,15,23,31,33}
		{
		\tkzDrawPoints[size=6,orange,osq](\X);
		};
		\foreach \X in {17,35}
		{
		\tkzDrawPoints[size=6,darkpink,line width=0.5mm,cross out, draw](\X);
		};
		
		\tkzDefPoint(-1.2,-0.2){gamma};
		\tkzLabelPoint[blue](gamma){\Large $\delta$}
    \end{tikzpicture}
\caption{\label{fig:exResM} An explicit calculation of  $\opResAc(\delta)$, done on \cref{ex:Nproj1}. In addition of $\Prj(\Delta^{\gpoint})$, $\opResAc(\delta)$ is given by all the dashed lines drawn above.}
\end{figure}
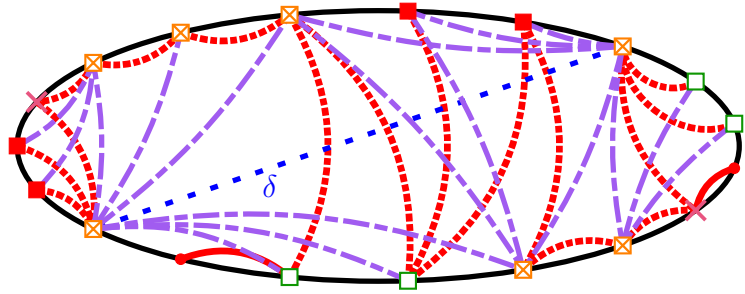
\end{ex}

Let $\delta \in \Accord'$. We recall a technical lemma characterizing $\rpoint$-arcs $\varsigma$ such that $s(\varsigma) \in \NP(\delta)_0^{{\rsquare}} \cup \NP(\delta)_0^{{\osquare}}$ and $t(\varsigma) \in \NP(\delta)_0^{{\osquare}} \cup \NP(\delta)_0^{{\gsquare}}$ that are accordions of $(\pmb{\Sigma}, \mathcal{M}, \Delta^{\gpoint})$.

For any $\delta \in \Accord'$ which is not contained in a cell of $\pmb{\Gamma}(\Prj(\Delta^{\gpoint}))$, define $\tc (\delta)$ as the pair $(\eta_R, C_R)$ where $\eta_R\in \Prj(\Delta^{\gpoint})$, and $C_R \in \pmb{\Gamma}(\Prj(\Delta^{\gpoint}))$ admitting $\eta_R$ as one of its edges, such that $\delta$ crosses $\eta_R$ and $t(\delta)$ is a vertex of $C_R$. If $\delta \subset C$ for some $\pmb{\Gamma}(\Prj(\Delta^{\gpoint}))$, then we set $\tc(\delta) = (\varnothing,C)$. We define $\sc(\delta)$ dually.

Given a $C \in \pmb{\Gamma}(\Prj(\Delta^{\gpoint}))$, we denote by $C(\delta)_0^{{\osquare}}$ the set of vertices of $C$ in $\NP(\delta)_0^{{\osquare}}$. We define similarly $C(\delta)_0^{{\gsquare}}$, $C(\delta)_0^{{\rpoint}}$ and $C(\delta)_0^{{\psquare}}$.

Define $\vartheta_1^R, \vartheta_2^R \in \Prj(\Delta^{\gpoint})$ such that:
\begin{enumerate}[label=$\bullet$,itemsep=1mm]
    \item $t(\vartheta_1^R) = t(\delta)$ and $\vartheta_1^R$ covers $\delta$ with respect to the counterclockwise order around $t(\delta)$ on $\Prj(\Delta^{\gpoint}) \cup \{\delta\}$; and,
    \item $t(\vartheta_2^R) = s(\vartheta_1^R)$ and $\vartheta_2^R$ covers $\vartheta_1^R$ with respect to the counterclockwise order around $s(\vartheta_1^R)$ on $\Prj(\Delta^{\gpoint}) \cup \{\delta\}$.
\end{enumerate}
We define $w_R(\delta) = s(\vartheta_2^R)$. We define $w_L(\delta)$ dually. See \cref{fig:exsctc} to visualize the labelings introduced.

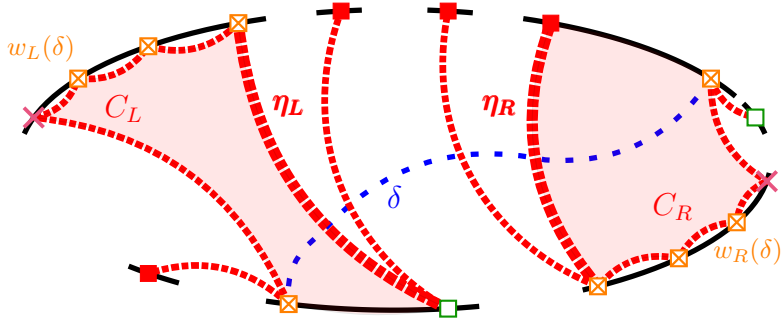
\begin{figure}[!ht]
    \centering
    \begin{tikzpicture}[mydot/.style={
					circle,
					thick,
					fill=white,
					draw,
					outer sep=0.5pt,
					inner sep=1pt
				}, scale = 1]
		\tikzset{
		osq/.style={
        rectangle,
        thick,
        fill=white,
        append after command={
            node [
                fit=(\tikzlastnode),
                orange,
                line width=0.3mm,
                inner sep=-\pgflinewidth,
                cross out,
                draw
            ] {}}}}
        \draw [line width=0.7mm,domain=10:22] plot ({5*cos(\x)}, {2*sin(\x)});
		\draw [line width=0.7mm,domain=25:70] plot ({5*cos(\x)}, {2*sin(\x)});
		\draw [line width=0.7mm,domain=78:85] plot ({5*cos(\x)}, {2*sin(\x)});
		\draw [line width=0.7mm,domain=95:102] plot ({5*cos(\x)}, {2*sin(\x)});
		\draw [line width=0.7mm,domain=110:170] plot ({5*cos(\x)}, {2*sin(\x)});
		\draw [line width=0.7mm,domain=225:235] plot ({5*cos(\x)}, {2*sin(\x)});
		\draw [line width=0.7mm,domain=250:283] plot ({5*cos(\x)}, {2*sin(\x)});
		\draw [line width=0.7mm,domain=300:355] plot ({5*cos(\x)}, {2*sin(\x)});
        %\draw [line width=0.7mm,domain=270:360] plot ({5*cos(\x)}, {2*sin(\x)});
		\foreach \X in {0,...,43}
		{
		\tkzDefPoint(5*cos(pi/22*\X),2*sin(pi/22*\X)){\X};
		};

		\draw[line width=0.7mm,blue, loosely dashed](4) to[bend left=30] (2,0) to[bend right=30] (-1,-1) to [bend right=30] (31);
		
		\draw[line width=0.9mm ,bend right=30,red, densely dashdotted](4) edge (43);
		\draw[line width=0.9mm ,bend right=30,red, densely dashdotted](4) edge (2);
		\draw[line width=0.9mm ,bend right=30,red, densely dashdotted](43) edge (41);
		\draw[line width=0.9mm ,bend right=30,red, densely dashdotted](41) edge (39);
		\draw[line width=0.9mm ,bend right=30,red, densely dashdotted](39) edge (37);
		\draw[line width=1.5mm ,bend left=30,red, densely dashdotted](37) edge (8);
		\draw[line width=0.9mm ,bend left=40,red, densely dashdotted](37) edge (10);
		\draw[line width=0.9mm ,bend left=40,red, densely dashdotted](34) edge (12);
		\draw[line width=1.5mm ,bend left=30,red, densely dashdotted](34) edge (14);
		\draw[line width=0.9mm ,bend right=30,red, densely dashdotted](16) edge (14);
		\draw[line width=0.9mm ,bend right=30,red, densely dashdotted](18) edge (16);
		\draw[line width=0.9mm ,bend right=30,red, densely dashdotted](20) edge (18);
		\draw[line width=0.9mm ,bend right=30,red, densely dashdotted](31) edge (20);
		\draw[line width=0.9mm ,bend right=30,red, densely dashdotted](31) edge (28);

		\filldraw [fill=red,opacity=0.1] (4) to [bend right=30] (43) to [bend right=30] (41) to [bend right=30] (39) to [bend right=30] (37) to [bend left=30] (8) to [bend left=10] cycle ;
		
		\filldraw [fill=red,opacity=0.1] (34) to [bend left=30] (14) to [bend left=30] (16) to [bend left=30] (18) to [bend left=30] (20) to [bend left=30] (31) to [bend right=10] cycle ;
		
		\foreach \X in {8,10,12,28}
		{
		\tkzDrawPoints[rectangle,fill =red,size=6,color=red](\X);
		};
		
		\foreach \X in {2,34}
		{
		\tkzDrawPoints[rectangle,size=6,color=dark-green,thick,fill=white](\X);
		};
		\foreach \X in {4,14,16,18,31,37,39,41}
		{
		\tkzDrawPoints[size=6,orange,osq](\X);
		};
		\foreach \X in {20,43}
		{
		\tkzDrawPoints[size=6,darkpink,line width=0.5mm,cross out, draw](\X);
		};

		\tkzDefPoint(-3.6,1){gammaM};
		\tkzLabelPoint[red](gammaM){\Large $C_L$}
		\tkzDefPoint(-1.4,1){gammaP};
		\tkzLabelPoint[red](gammaP){\Large $\pmb{\eta_L}$}
		\tkzDefPoint(1.4,1){v};
		\tkzLabelPoint[red](v){\Large $\pmb{\eta_R}$}
		\tkzDefPoint(3.7,-0.3){w};
		\tkzLabelPoint[red](w){\Large $C_R$}
		\tkzDefPoint(0,-0.2){deltav};
		\tkzLabelPoint[blue](deltav){\Large $\delta$}
		\tkzDefPoint(-4.7,1.8){deltaw};
		\tkzLabelPoint[orange](deltaw){$w_L(\delta)$}
		\tkzDefPoint(4.7,-0.9){w};
		\tkzLabelPoint[orange](w){$w_R(\delta)$}
    \end{tikzpicture}
    \caption{\label{fig:exsctc} Illustration of  $\sc(\delta) = (\eta_L, C_L)$ and $\tc(\delta) = (\eta_R, C_R)$ given $\delta \in \Accord$.}
\end{figure}

\begin{lemma}[\cite{DS251}] \label{lem:conditionsResAc}
Let $(\pmb{\Sigma}, \mathcal{M}, \Delta^{\gpoint})$ a $\gpoint$-dissected marked disc with $\mathcal{M} \subset \partial \pmb{\Sigma}$. Let $\delta \in \Accord'$. If $\delta \nsubseteq C$ for some $C \in \pmb{\Gamma}(\Prj(\Delta^{\gpoint}))$, then we set $\sc(\delta) = (\eta_L,C_L)$ and $\tc(\delta) = (\eta_R,C_R)$; otherwise, we set $C=C_L=C_R$ the cell containing $\delta$. The curve $\varsigma \in \opResAc'(\delta)$ must satisfy all of the following assertions: 
\begin{enumerate}[label=$(\roman*)$, itemsep=1mm]
        \item \label{1Accord} if $\varsigma$ crosses $\eta_L$ and $w_L(\delta) \neq t(\eta_L)$, then $s(\varsigma) = s(\delta)$, and, dually, if $\varsigma$ crosses $\eta_R$ and $w_R(\delta) \neq s(\eta_R)$, then $t(\varsigma) = t(\delta)$;
        \item \label{2Accord} if $\varsigma \subset C_L$, then $s(\varsigma)=s(\delta)$, and, dually, if $\varsigma \subset C_R$, then $t(\varsigma) = t(\delta)$;
        \item \label{3Accord} if $s(\varsigma)\in \NP(\delta)_0^{{\rsquare}}$ and $t(\varsigma) \in C_R(\delta)_0$, then $t(\varsigma)=w_L(\delta)$, and, dually, if $t(\varsigma)\in \NP(\delta)_0^{{\gsquare}}$ and $s(\varsigma) \in C_R(\delta)_0$, then $s(\varsigma)=w_R(\delta)$;
        \item \label{4Accord} $s(\varsigma) \notin C_L(\delta)_0^{{\osquare}}\setminus \{ s(\delta),t(\eta_L)\}$, and $t(\varsigma) \notin C_R(\delta)_0^{{\osquare}}\setminus \{ t(\delta),s(\eta_R)\}$.
    \end{enumerate}
\end{lemma}

Following more technical calculations, and the use of \cref{thm:equivres}, we get a geometric description of the monogeneous resolving subcategories of $\rep(Q,R)$ by using the monogeneous geometric resolving sets.

\begin{theorem}[\cite{DS251}]
\label{thm:res_clo_1}
Let $(Q,R)$ be a gentle tree, and $\Surf(Q,R) = (\pmb{\Sigma}, \mathcal{M}, \Delta^{\gpoint})$. For any $\delta \in \Accord$, we have that \[\Res(\MM(\delta)) = \add\left( \MM(\eta) \mid \eta \in \opResAc(\delta) \right) = \mathscr{U}_\delta.\]
\end{theorem}

In the following, we will provide a procedure for obtaining all the resolving subcategories as unions of monogeneous ones. As it will be underlined in our proofs, we fix the following convention.

\begin{conv} \label{conv:usethisthm}
    Whenever we must show that some subcategory of $\rep(Q,R)$, or some collection of accordions in $\Accord(Q,R)$, is resolving, we always use \cref{thm:equivres}.
\end{conv}
	
	\section{Closure operators and poset of resolving subcategories}
	\label{sec:ResPoset}
	\pagestyle{plain}

\subsection{Closure operators}
\label{ss:closop}
We fix a set $\mathfrak{A}$ in the following. We denote by $\mathcal{P}(\mathfrak{A})$ the poset of the subsets of $\mathfrak{A}$.
\begin{definition}
\label{def:closop} A \new{closure operator} on $\mathfrak{A}$ is a map $\cl : \mathcal{P}(\mathfrak{A}) \rightarrow \mathcal{P}(\mathfrak{A})$ which satisfies the following properties:
\begin{enumerate}[label=$\bullet$, itemsep=1mm]
    \item for all $\mathfrak{B} \in  \mathcal{P}(\mathfrak{A})$, $\mathfrak{B} \subseteq \cl(\mathfrak{B})$,
    \item for all $\mathfrak{B} \in  \mathcal{P}(\mathfrak{A})$, $\cl(\mathfrak{B}) = \cl(\cl(\mathfrak{B}))$, and,
    \item for all $\mathfrak{B}, \mathfrak{C} \in \mathcal{P}(\mathfrak{A})$, if $\mathfrak{B} \subseteq \mathfrak{C}$, then $\cl(\mathfrak{B}) \subseteq \cl(\mathfrak{C})$.
\end{enumerate}
A subset $\mathfrak{B} \in \mathcal{P}(\mathfrak{A})$ is \new{$\cl$-closed} if $\mathfrak{B} = \cl(\mathfrak{B})$.
We denote by $\mathcal{P}_{\cl}(\mathfrak{A})$ the subset of $\mathcal{P}(\mathfrak{A})$ made of $\cl$-closed sets.
\end{definition}

\begin{ex} \label{ex:closop} $ $
\begin{enumerate}[label=$\bullet$, itemsep=1mm]
    \item For any $n \in \mathbb{N}^*$, the convex hull operator is a closure operator on the Euclidean space $\mathbb{R}^n$.
    \item Given a gentle quiver $(Q,R)$, the resolving closure operator is a closure operator on $\rep(Q,R)$. Then $(-)^{\Res}$ is a closure operation on $\pmb{\ind \setminus \proj}(Q,R)$. \qedhere
\end{enumerate}
\end{ex}

\begin{lemma}
\label{lem:elementaryonclosop} Consider a closure operator $\cl$ on $\mathfrak{A}$. The following assertions hold:
\begin{enumerate}[label=$\bullet$, itemsep=1mm]
    \item we have $\cl(\mathfrak{A}) = \mathfrak{A}$;
    \item for any $\mathfrak{B} \in \mathcal{P}(\mathfrak{A})$, we have \[\cl(\mathfrak{B}) = \bigcap_{\mathfrak{C} \in \mathcal{P}_{\cl}(\mathfrak{A}),\ \mathfrak{B} \subseteq \mathfrak{C}} \mathfrak{C}\ ;\]
    \item the poset $(\mathcal{P}_{\cl}(\mathfrak{A}), \subseteq)$ is a complete lattice where, for any collection $(\mathfrak{B}_i)_{i \in I}$ of subsets of $\mathcal{P}_{\cl}(\mathfrak{A})$, the join $\underline{\vee}$ and the meet $\overline{\wedge}$ are defined as follows: \[ \overline{\bigwedge_{i \in I}} \mathfrak{B}_i = \bigcap_{i \in I} \mathfrak{B}_i \text{ and } \underset{i \in I}{\underline{\bigvee}}\mathfrak{B}_i = \cl \left( \bigcup_{i \in I} \mathfrak{B}_i \right).  \]
\end{enumerate}
\end{lemma}

Now we assume that $\mathfrak{A}$ is equipped with a partial order $\leqslant$. We denote by $\mathscr{J}(\mathfrak{A}, \leqslant)$ the set of order ideals in $(\mathfrak{A}, \leqslant)$.

\begin{definition}
\label{def:closopresp} A closure operator $\cl$ on $\mathfrak{A}$ is said to be
\new{compatible with $\leqslant$} whenever, for any $\mathfrak{B} \in \mathcal{P}(\mathfrak{A})$, we have that $\cl(\mathfrak{B}) \in \mathscr{J}(\mathfrak{A}, \leqslant)$.
\end{definition}

\begin{ex} \label{ex:closopcompatible} $ $
\begin{enumerate}[label=$\bullet$, itemsep=1mm]
    \item For any $n \in \mathbb{N}^*$, we endow $\mathbb{R}^n$ with a partial order $\preccurlyeq$: for $(x,y) \in (\mathbb{R}^n)^2$, we say that $x \preccurlyeq y$ if, and only if, there exists $\lambda \in [0;1]$ such that $x = \lambda y$. Consider the closure operator $\cl$ on $\mathbb{R}^n$ defined such that, for any $\mathfrak{B} \in \mathcal{P}(\mathbb{R}^2)$, $\cl(\mathfrak{B})$ is the convex hull of $\mathfrak{B} \cup \{0\}$. Then $\cl$ is compatible with $\preccurlyeq$.
    \item Given a gentle tree $(Q,R)$, the closure operation $(-)^{\Res}$ on $\pmb{\ind \setminus \proj}(Q,R)$ is compatible with the Res-relation $\Resleq$. \qedhere
\end{enumerate}
\end{ex}

\begin{prop} \label{prop:closeopcompatible} Consider $\cl$ a closure operator on $\mathfrak{A}$ compatible with $\leqslant$. Then $(\mathcal{P}_{\cl}(\mathfrak{A}), \subseteq)$ is a subposet of $(\mathscr{J}(\mathfrak{A}, \leqslant), \subseteq)$.
\end{prop}

\begin{remark} \label{rem:notasublattice} In general, given a closure operator $\cl$ on $\mathfrak{A}$ compatible with $\leqslant$, the lattice of $\cl$-closed subsets of $\mathfrak{A}$ is not a sublattice of $(\mathscr{J}(\mathfrak{A}, \leqslant), \subseteq,\cap, \cup)$.
\end{remark}

\subsection{Lattices and join-irreducible elements}
\label{ss:latticeJI}

From now on, we assume that $(\mathfrak{A}, \leqslant, \wedge, \vee)$ is a finite lattice. This section highlights unique decomposition into join-irreducible elements, which we call the \emph{upper join-decomposition}.

Let us first recall some notions and relevant results from poset theory. An \new{antichain} $\mathfrak{B}$ of $(\mathfrak{A},\leqslant)$ is a subset of pairwise non-comparable elements of $\mathfrak{A}$. Denote by $\Anti(\mathfrak{A}, \leqslant)$ the set of all the antichains of $(\mathfrak{A},\leqslant)$.

\begin{prop} \label{prop:1to1antiideal} The map \[ \Theta_\mathfrak{A} : \left\{ \begin{matrix}
\Anti(\mathfrak{A}, \leqslant) & \longrightarrow & \mathscr{J}(\mathfrak{A},\leqslant) \\
\mathfrak{B} & \longmapsto & \langle \mathfrak{B} \rangle_\mathfrak{A}
\end{matrix} \right.\]
is bijective.
\end{prop}

This one-to-one correspondence allows us to endow $\Anti(\mathfrak{A}, \leqslant)$ with a lattice structure. 
\label{order:antichain} Consider $\bleq$ the order relation on $\Anti(\mathfrak{A}, \leqslant)$ defined by:  \[ \forall (\mathfrak{B}, \mathfrak{C}) \in \Anti(\mathfrak{A}, \leqslant)^2,\ \mathfrak{B} \bleq \mathfrak{C} \Longleftrightarrow \langle \mathfrak{B} \rangle_\mathfrak{A} \subseteq \langle \mathfrak{C} \rangle_\mathfrak{A}. \]
We define the meet $\curlywedge$ and the join operations $\curlyvee$ on $\Anti(\mathfrak{A}, \leqslant)$ by:
\[ \mathfrak{B} \curlywedge \mathfrak{C} =\Theta_\mathfrak{A}^{-1} \left(\langle \mathfrak{B} \rangle_{\mathfrak{A}} \cap  \langle \mathfrak{C} \rangle_\mathfrak{A} \right) \text{ and } \mathfrak{B} \curlyvee \mathfrak{C} = \Theta_\mathfrak{A}^{-1} \left(\langle \mathfrak{B} \rangle_{\mathfrak{A}} \cup  \langle \mathfrak{C} \rangle_\mathfrak{A} \right). \] By noting that $\langle \mathfrak{B} \rangle_{\mathfrak{A}} \cup  \langle \mathfrak{C} \rangle_\mathfrak{A} = \langle \mathfrak{B} \cup  \mathfrak{C} \rangle_\mathfrak{A}$, we can calculate directly, from $\mathfrak{B}$ and $\mathfrak{C}$, the antichain $\mathfrak{D}$ by deleting, for each pair of comparable elements in $\mathfrak{B} \cup \mathfrak{C}$, the lowest one.

A \new{join-decomposition} of $a \in \mathfrak{A}$ is a subset $\mathfrak{B} \subseteq \JI(\mathfrak{A})$ such that  \[a = \bigvee_{b \in \mathfrak{B}} b.\] Note that $\widehat{0}_\mathfrak{A}$ admits a join-decomposition given by $\mathfrak{B} = \varnothing$.

\begin{prop} \label{prop:decompJI}
Let $(\mathfrak{A}, \leqslant, \wedge,\vee)$ be a finite lattice. For any $a \in \mathfrak{A}$, there exists a unique antichain $\mathfrak{B}$ of $(\JI(\mathfrak{A}), \leqslant)$ maximal with respect to $\bleq$ such that $\mathfrak{B}$ is a join-decomposition of $a$.
\end{prop}

\begin{proof}
The fact that any element $a \in \mathfrak{A}$ admits a join-decomposition by an antichain $\mathfrak{B}$ of $(\JI(\mathfrak{A}), \leqslant)$ is obtained by induction as $\mathfrak{A}$ is finite. We only have to check the unicity of a join-decomposition given by a maximal antichain of $(\JI(\mathfrak{A}), \leqslant)$ for $\bleq$.

Let $a \in \mathfrak{A}$. Consider $\mathfrak{B}$ and 
$\mathfrak{C}$ two antichains of $(\JI(\mathfrak{A}), \leqslant)$ which are two join-decompositions of $a$. Then we have\[a = \bigvee_{b \in \mathfrak{B}} b = \bigvee_{c \in \mathfrak{C}} c = \bigvee_{d \in \mathfrak{B} \cup \mathfrak{C}} d. \]
Set $\mathfrak{D} = \mathfrak{B} \curlyvee \mathfrak{C}$ in the lattice $(\Anti(\JI(\mathfrak{A}), \leqslant), \bleq, \curlywedge, \curlyvee)$. By construction, we have $\mathfrak{B} \bleq \mathfrak{D}$, $\mathfrak{C}  \bleq \mathfrak{D}$, and $a = \bigvee_{d \in \mathfrak{D}} d$. 

If $\mathfrak{B}$ and $\mathfrak{C}$ are maximal antichains both giving a join-decomposition of $a$, then $\mathfrak{B} = \mathfrak{D} = \mathfrak{C}$.
\end{proof}

\begin{definition}
\label{def:upperjoindec}
Let $(\mathfrak{A},\leqslant, \wedge, \vee)$ be a finite lattice. The join-decomposition of $a \in \mathfrak{A}$ in \cref{prop:decompJI} is called the \new{upper join-decomposition} of $a$ in $\mathfrak{A}$, and, from now on, is denoted by $\D_\mathfrak{A}(a)$.
\end{definition}

\begin{remark} \label{rem:canonicaldecompforideals} The upper join-decomposition of elements in $\mathfrak{A}$ we exhibit corresponds to the usual canonical join-decomposition of ideals in the distributive lattice $(\mathscr{J}(\JI(\mathfrak{A}), \leqslant), \subseteq, \cap, \cup)$.
\end{remark}

\begin{ex}
Let $(\mathfrak{A},\leqslant, \wedge, \vee)$ be the lattice whose Hasse diagram is given in \cref{fig:Posetanddecomp1}. We can check that $f$ admits many distinct join-decompositions; among them, $\{c,d,e\}$ is the greatest for $\bleq$. Therefore we have $\D_\mathfrak{A}(f) = \{c,d,e\}$. \qedhere
 \begin{figure}[!ht]
\centering 
    \begin{tikzpicture}
		
		\tkzDefPoint(0,0){f};
		\tkzDefPoint(0,-1.5){e};
		\tkzDefPoint(-1.5,-1.5){d};
		\tkzDefPoint(1.5,-1){c};
		\tkzDefPoint(1.5,-2){b};
		\tkzDefPoint(0,-3){a};
		
		\draw[line width=0.7mm,black] (a) -- (b);
		\draw[line width=0.7mm,black] (a) -- (d);
		\draw[line width=0.7mm,black] (a) -- (e);
		\draw[line width=0.7mm,black] (b) -- (c);
		\draw[line width=0.7mm,black] (c) -- (f);
		\draw[line width=0.7mm,black] (d) -- (f);
		\draw[line width=0.7mm,black] (e) -- (f);

		\foreach \X in {b,c,d,e}
		{
		\tkzDrawPoints[fill =red,size=6,color=red](\X);
		};
		\foreach \X in {a,f}
		{
		\tkzDrawPoints[size=6,color=black](\X);
		};
		
		\tkzLabelPoint[below left](a){\Large $a$}
		\tkzLabelPoint[red,below right](b){\Large $b$}
		\tkzLabelPoint[red,above right](c){\Large $c$}
		\tkzLabelPoint[red,below left](d){\Large $d$}
		\tkzLabelPoint[red,below left](e){\Large $e$}
		\tkzLabelPoint[above left](f){\Large $f$}
    \end{tikzpicture}
\caption{\label{fig:Posetanddecomp1} Hasse diagram of a lattice. The red elements $b,c,d,$ and $e$ are the join-irreducible ones.}
\end{figure}
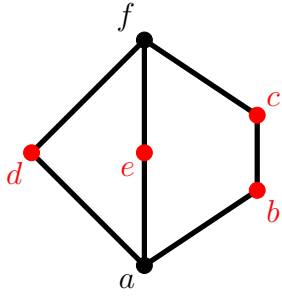
\end{ex}

We can state the following result thanks to the upper join-decomposition of any element in $\mathfrak{A}$. 

\begin{theorem} \label{thm:clandjoinirr}
Let $(\mathfrak{A}, \leqslant, \wedge, \vee)$ be a finite lattice, and $(\JI(\mathfrak{A}), \leqslant)$ be its subposet of join-irreducible elements. Consider $\cl_\vee : \mathcal{P}(\JI(\mathfrak{A})) \longrightarrow \mathcal{P}(\JI(\mathfrak{A}))$ defined by \[\cl_\vee (\mathfrak{U}) = \left\langle \bigvee_{u \in \mathfrak{U}} u \right\rangle_\mathfrak{A} \cap \JI(\mathfrak{A}).\]
Then $\cl_\vee$ is a closure operator on $\JI(\mathfrak{A})$ compatible with $\leqslant$. Moreover, the map
\[\Psi_\mathfrak{A} : \left\{ \begin{matrix} 
(\mathcal{P}_{\cl_\vee}(\JI(\mathfrak{A})), \subseteq) & \longrightarrow & (\mathfrak{A}, \leqslant) \\
\mathfrak{U} & \longmapsto & \bigvee_{u \in \mathfrak{U}} u.
\end{matrix} \right.\]
is an isomorphism whose inverse is the map
\[\Phi_\mathfrak{A} : \left\{ \begin{matrix} 
(\mathfrak{A}, \leqslant) & \longrightarrow & (\mathcal{P}_{\cl_\vee}(\JI(\mathfrak{A})), \subseteq) \\
a & \longmapsto & \langle \mathbb{D}_\mathfrak{A}(a) \rangle_\mathfrak{A} \cap \JI(\mathfrak{A})
\end{matrix} \right.\]
\end{theorem}

This is a precise restatement of a well-known property in the study of closure lattices. We refer the reader to \cite[Section 1.5]{BS81} for more details.

\subsection{On resolving subcategories}
\label{ss:rescat}
In this section, assuming that $(Q,R)$ is a gentle tree, we apply the previous results obtained on $(\ResOrd(Q,R), \subseteq)$

\begin{prop} \label{prop:folkloreres}
Let $(Q,R)$ be a gentle tree. The following assertions hold:
\begin{enumerate}[label=$(\roman*)$,itemsep=1mm]
    \item The poset $(\mathcal{P}_{\Res}(\pmb{\ind \setminus \proj}(Q,R)), \subseteq)$ is a subposet of\\ $(\mathscr{J}(\pmb{\ind \setminus \proj}(Q,R), \Resleq), \subseteq)$;
    \item The lattice $(\mathcal{P}_{\Res}(\pmb{\ind \setminus \proj}(Q,R)), \subseteq, \cap, \veebar)$  is isomorphic to \\ $(\ResOrd(Q,R), \subseteq, \cap, \underline{\cup})$;
    \item The join-irreducible elements of $(\mathcal{P}_{\Res}(\pmb{\ind \setminus \proj}(Q,R)), \subseteq, \cap, \underline{\cup})$ are the principal ideals of $(\pmb{\ind \setminus \proj}(Q,R), \Resleq)$.
\end{enumerate}
\end{prop}

\begin{proof}
The assertion $(i)$ is obvious from \cref{prop:closeopcompatible}. The assertion $(ii)$ is by definition. The assertion $(iii)$ is a direct consequence of \cref{thm:Monoareallthejoinirred}, as a principal ideal is the set of the non-projective indecomposable objects in a monogeneous resolving subcategory.
\end{proof}

\begin{ex}
In \cref{fig:Resrelex}, we gave an example of $(\pmb{\ind \setminus \proj}(Q,R), \Resleq)$ and the lattice $(\ResOrd(Q,R), \subseteq, \cap, \underline{\cup})$. We identified the resolving subcategories with the unique antichain of $(\pmb{\ind \setminus \proj}(Q,R), \Resleq)$ corresponding to their upper join-decomposition . 
 \begin{figure}[!ht]
\centering 
    \begin{tikzpicture}
    \begin{scope}[->,scale=1.4]
		\node (a) at (0,0) {$1$};
		\node (b) at (1,0) {$2$};
		\node (c) at (2,0) {$3$};
		\node (d) at (1,1) {$4$};
		
		\draw (a) -- (b);
		\draw (b) -- (c);
		\draw (b) -- (d);

		\draw[dashed,-] ([xshift=-.05cm,yshift=.35cm]b.north) arc[start angle = 100, end angle = 170, x radius=.6cm, y radius =.6cm];
		
		\node at (1,-.5) {$(Q,R)$};
    \end{scope}
    
    \begin{scope}[xshift=5cm,yshift=-1.25cm]
	
	    \node[red] (b) at (-0.5,1) {\scalebox{0.6}{$
	    \begin{tikzpicture}[baseline={(0,-.2)},scale=0.2]
	    \node at (0,0){$2$};
	    \node at (1,-1){$4$};
	    \end{tikzpicture}$}};
	    \node[red] (c) at (0.5,1)  {\scalebox{0.6}{$
	    \begin{tikzpicture}[baseline={(0,-.2)},scale=0.2]
	    \node at (0,0){$2$};
	    \node at (-1,-1){$3$};
	    \end{tikzpicture}$}};
	    \node[red] (e) at (0,3) {\scalebox{0.6}{$
	    \begin{tikzpicture}[baseline={(0,-.1)},scale=0.2]
	    \node at (0,0){$2$};
	    \end{tikzpicture}$}};
	    \node[red] (f) at (-1,3) {\scalebox{0.6}{$
	    \begin{tikzpicture}[baseline={(0,-.2)},scale=0.2]
	    \node at (0,0){$1$};
	    \node at (-1,-1){$2$};
	    \end{tikzpicture}$}};
	    \node[red] (h) at (1,3) {\scalebox{0.6}{$\begin{tikzpicture}[baseline={(0,-.1)},scale=0.2]
	    \node at (0,0){$1$};
	    \end{tikzpicture}$}};

		\draw[line width=0.7mm,black] (b) -- (e);
		\draw[line width=0.7mm,black] (b) -- (f);
		\draw[line width=0.7mm,black] (c) -- (e);
		\draw[line width=0.7mm,black] (c) -- (h);
		
		\node at (0,.5) {$(\pmb{\ind \setminus \proj}(Q,R), \Resleq)$};
	\end{scope}
    
	\begin{scope}[xshift=9cm,yshift=-2cm]
	
	    \node (a) at (0,0) {$\varnothing$};
	    \node[red] (b) at (-1,1) {\scalebox{0.6}{$\left\{
	    \begin{tikzpicture}[baseline={(0,-.2)},scale=0.2]
	    \node at (0,0){$2$};
	    \node at (1,-1){$4$};
	    \end{tikzpicture}\right\}$}};
	    \node[red] (c) at (1,1)  {\scalebox{0.6}{$\left\{
	    \begin{tikzpicture}[baseline={(0,-.2)},scale=0.2]
	    \node at (0,0){$2$};
	    \node at (-1,-1){$3$};
	    \end{tikzpicture}\right\}$}};
	    \node (d) at (0,2)  {\scalebox{0.6}{$\left\{
	    \begin{tikzpicture}[baseline={(0,-.2)},scale=0.2]
	    \node at (0,0){$2$};
	    \node at (1,-1){$4$};
	    \end{tikzpicture}, \begin{tikzpicture}[baseline={(0,-.2)},scale=0.2]
	    \node at (0,0){$2$};
	    \node at (-1,-1){$3$};
	    \end{tikzpicture}\right\}$}};
	    \node[red] (e) at (-.4,3) {\scalebox{0.6}{$\left\{
	    \begin{tikzpicture}[baseline={(0,-.1)},scale=0.2]
	    \node at (0,0){$2$};
	    \end{tikzpicture}\right\}$}};
	    \node[red] (f) at (-1.5,2.5) {\scalebox{0.6}{$\left\{
	    \begin{tikzpicture}[baseline={(0,-.2)},scale=0.2]
	    \node at (0,0){$1$};
	    \node at (-1,-1){$2$};
	    \end{tikzpicture}\right\}$}};
	    \node (g) at (-1,4) {\scalebox{0.6}{$\left\{
	    \begin{tikzpicture}[baseline={(0,-.2)},scale=0.2]
	    \node at (0,0){$1$};
	    \node at (-1,-1){$2$};
	    \end{tikzpicture}, \begin{tikzpicture}[baseline={(0,-.1)},scale=0.2]
	    \node at (0,0){$2$};
	    \end{tikzpicture}\right\}$}};
	    \node[red] (h) at (1.5,2.5) {\scalebox{0.6}{$\left\{\begin{tikzpicture}[baseline={(0,-.1)},scale=0.2]
	    \node at (0,0){$1$};
	    \end{tikzpicture}\right\}$}};
	    \node (h2) at (1,4)  {\scalebox{0.6}{$\left\{
	    \begin{tikzpicture}[baseline={(0,-.2)},scale=0.2]
	    \node at (0,0){$2$};
	    \node at (1,-1){$4$};
	    \end{tikzpicture}, \begin{tikzpicture}[baseline={(0,-.1)},scale=0.2]
	    \node at (0,0){$1$};
	    \end{tikzpicture}\right\}$}};
		\node (i) at (0,5) {\scalebox{0.6}{$\left\{
	    \begin{tikzpicture}[baseline={(0,-.2)},scale=0.2]
	    \node at (0,0){$1$};
	    \node at (-1,-1){$2$};
	    \end{tikzpicture}, \begin{tikzpicture}[baseline={(0,-.1)},scale=0.2]
	    \node at (0,0){$2$};
	    \end{tikzpicture},\begin{tikzpicture}[baseline={(0,-.1)},scale=0.2]
	    \node at (0,0){$1$};
	    \end{tikzpicture}\right\}$}};

		\draw[line width=0.7mm,black] (a) -- (b);
		\draw[line width=0.7mm,black] (a) -- (c);
		\draw[line width=0.7mm,black] (b) -- (d);
		\draw[line width=0.7mm,black] (b) -- (f);
		\draw[line width=0.7mm,black] (c) -- (d);
		\draw[line width=0.7mm,black] (c) -- (h);
		\draw[line width=0.7mm,black] (d) -- (e);
		\draw[line width=0.7mm,black] (d) -- (h2);
		\draw[line width=0.7mm,black] (e) -- (g);
		\draw[line width=0.7mm,black] (f) -- (g);
		\draw[line width=0.7mm,black] (g) -- (i);
		\draw[line width=0.7mm,black] (h) -- (h2);
		\draw[line width=0.7mm,black] (h2) -- (i);
		
		\node at (0,-.75) {$(\ResOrd(Q,R), \subseteq)$};
	\end{scope}
    \end{tikzpicture}
\caption{\label{fig:Resrelex} An example of the Res-relation poset and the lattice of resolving subcategories calculated for a gentle tree. We identify the resolving subcategories with their upper-join decomposition, and, more precisely, we keep the collection of non-projective indecomposable representations that generate the corresponding monogeneous resolving subcategories appearing in their upper join decomposition decomposition.}
\end{figure}
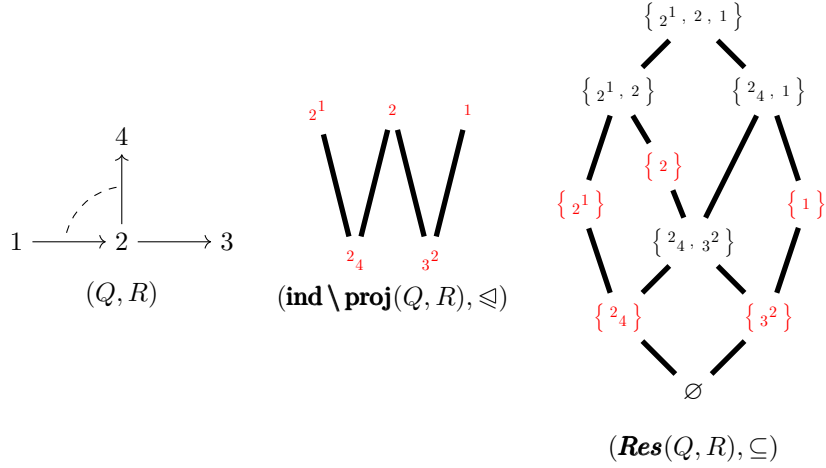
\end{ex}

\begin{prop} \label{prop:candecompres}
Let $(Q,R)$ be a gentle tree. For any resolving subcategory $\mathscr{C} \subseteq \rep(Q,R)$, there exists a unique $\mathfrak{E}_\mathscr{C} \in \Anti(\pmb{\ind \setminus \proj}(Q,R), \Resleq)$ such that \[\mathscr{C} = \add \left( \Res(M) \mid M \in \mathfrak{E}_\mathscr{C} \right). \]
\end{prop}

\begin{proof}
The result is a consequence of \cref{thm:clandjoinirr} and \cref{prop:folkloreres}.
\end{proof}

Following the previous result, we can construct an algorithm that allows one to construct $\mathfrak{E}_\mathscr{C}$ in the case $\mathscr{C} = \Res(\mathfrak{B})$ where $\mathfrak{B} \in \Anti(\pmb{\ind \setminus \proj}(Q,R), \Resleq)$. A first step is , given any pair $(X,Y) \in \mathfrak{B}^2$, to determine the antichain $\mathfrak{E}_{\Res(X,Y)}$

\begin{algo} \label{algo:Rescatmotivation} Let $(Q,R)$ be a gentle tree. Set $\mathfrak{A} = \ind \setminus \proj(Q,R)$. We input $\mathfrak{X} \subset \mathfrak{A}$.
\begin{enumerate}[label=$(\arabic*)$,itemsep=1mm]
\item Set $\mathfrak{B}^0 = \Theta_{\mathfrak{A}}^{-1} \left(\langle \mathfrak{X} \rangle_\mathfrak{A} \right)$.
\item \label{2algores1motiv} At the $i$th iteration of the algorithm, we define $\mathfrak{B}^{i+1}$ from $\mathfrak{B}^i$ as follows:
\begin{enumerate}[label=$(2 \alph*)$,itemsep=1mm]
    \item if there exists a pair $(X,Y) \in (\mathfrak{B}^i)^2$ such that $\mathfrak{E}_{\Res(X,Y)} \nsubseteq  \langle \mathfrak{B}^i \rangle_\mathfrak{A}$, then we set $\mathfrak{B}^{i+1} = \Theta_\mathfrak{A}^{-1}\left(\langle \mathfrak{B}^i \rangle_\mathfrak{A} \cup \langle \mathfrak{E}_{\Res(X,Y)} \rangle_\mathfrak{A} \right)$; 
    \item Otherwise, we set $\mathfrak{B}^{i+1} = \mathfrak{B}^i$.
\end{enumerate}
\item If $\mathfrak{B}^{i+1} \neq \mathfrak{B}^i$, then go back to Step $(2)$;

\item Otherwise, return $\mathfrak{B}^{i+1}$.
\end{enumerate}
\end{algo}

\begin{theorem} \label{thm:EndAlgo} Let $(Q,R)$ be a gentle tree, and $\mathfrak{X} \subseteq \pmb{\ind \setminus \proj}(Q,R)$. Then \cref{algo:Rescatmotivation} ends, and it returns the antichain $\mathfrak{E}_{\Res(\mathfrak{X})}$.
\end{theorem}

Let us first state the following lemma, which is a direct consequence of the definition of the Res-relation.

\begin{lemma} \label{lem:IdealsclosedSyzygies}
Let $(Q,R)$ be a gentle tree. All the ideals of $(\pmb{\ind \setminus \proj}(Q,R), \Resleq)$ are closed under non-projective syzygies, and by taking extensions with projective representations.
\end{lemma}

\begin{proof}[Proof of \cref{thm:EndAlgo}]
The sequences $\left( \mathfrak{B}^i \right)_{i \in \mathbb{N}}$ is increasing with respect to $\bleq$. Moreover, the set $\Anti(\pmb{\ind \setminus \proj}(Q,R))$ is finite. So \cref{algo:Rescatmotivation} ends.

Let $j \in \mathbb{N}$ be such that $\mathfrak{B}^j = \mathfrak{B}^{j+1}$. We prove that $\mathfrak{B}^j$ is such that \[\Res(\mathfrak{X}) = \add(\Res(M) \mid M \in \mathfrak{B}^j).\] In the following, we set $\mathscr{S} =\add(\Res(M) \mid M \in \mathfrak{B}^j)$

First, we show that $\mathscr{S}$ is a resolving subcategory. Note that \[\mathscr{S} = \add(M \mid M \in \langle \mathfrak{B}^j \rangle \cup \proj(Q,R)).\] By \cref{lem:IdealsclosedSyzygies}, we know that $\add(\Res(M) \mid M \in \mathfrak{B}^j)$ is closed under syzygies. Let $X,Y \in \mathscr{S}$. Let $M,N \in \mathfrak{B}^j$ be such that $X \Resleq M$ and $Y \Resleq N$. We have that $\Res(X,Y) \subseteq \Res(M,N)$. By assumption, we have that $\mathfrak{E}_{\Res(M,N)} \subseteq \langle \mathfrak{B}^j \rangle$, which implies that $\Res(M,N) \subseteq \mathscr{S}$. Therefore, any extension between $X$ and $Y$ is contained in $\mathscr{S}$. So $\mathscr{S}$ is closed under sums, summands, extensions, and syzygies and contains $\proj(Q,R)$. By \cref{thm:equivres}, $\mathscr{S}$ is a resolving subcategory.

By construction, $\mathfrak{X} \subseteq \langle \mathfrak{B}^j \rangle$, so $\Res(\mathfrak{X}) \subseteq \mathscr{S}$. Now we will show that $\mathscr{S} \subseteq \Res(\mathfrak{X})$. To do so, we show by induction that for all $i \in \mathbb{N}$, $\add(\Res(M) \mid M \in \mathfrak{B}^i) \subseteq \Res(\mathfrak{X})$. It is obvious that $\add(\Res(M) \mid M \in \mathfrak{B}^0) \subseteq \Res(\mathfrak{X})$. 

Assume that for some $i \in \mathbb{N}$, we have that $\add(\Res(M) \mid M \in \mathfrak{B}^i) \subseteq \Res(\mathfrak{X})$. If $\mathfrak{B}^i = \mathfrak{B}^{i+1}$, then we are done. Assume that there exists $(X,Y) \in (\mathfrak{B}^i)^2$ such that $\mathfrak{E}_{\Res(X,Y)} \nsubseteq \mathfrak{B}^i$. As $\Res(\mathfrak{X})$ is closed under extension, and by induction, we have that $\Res(X,Y) \subset \Res(\mathfrak{X})$, which means that $\mathfrak{B}^{i+1} \subseteq \Res(\mathfrak{X})$. 

As a result, we have that $\mathscr{S} = \Res(\mathfrak{X})$, and by \cref{prop:candecompres}, we have that $\mathfrak{B}^j = \mathfrak{E}_{\mathscr{C}}$.
\end{proof}

Thanks to \cref{thm:EndAlgo}, we show that to get the resolving subcategories generated by a collection of non-projective indecomposable representations of $(Q,R)$, we only need to express explicitly the resolving subcategories generated by a pair of non-comparable non-projective indecomposable representations. This will be our goal in the following section.

\subsection{On geometric resolving sets}
\label{ss:geometric}

Fix a gentle tree $(Q,R)$. In the following, we set $\Surf(Q,R) = (\pmb{\Sigma}, \mathcal{M}, \Delta^{\gpoint})$. Set $\Accord = \mathscr{A}(\mathcal{S}(Q,R))$, and $\Accord' = \Accord \setminus \Prj(\Delta^{\gpoint})$. Let us first introduce the following definition.

\begin{definition} \label{def:resclosaccord}
  Given a collection of accordions $\mathfrak{U}$, we consider its \new{resolving closure} $\opResAc(\mathfrak{U})$ to be the smallest resolving set of $\Accord$ containing $\mathfrak{U}$. We also set $\opResAc'(\mathfrak{U}) = \opResAc(\mathfrak{U}) \setminus \Prj(\Delta^{\gpoint})$.
\end{definition}

For simplicity, for $\delta, \eta \in \Accord'$, we write $\opResAc(\delta)$ for $\opResAc(\{\delta\})$, and $\opResAc(\delta,\eta)$ for $\opResAc(\{\delta,\eta\})$. Similar conventions apply to $\opResAc'$.

\begin{remark}
    Note that, for any $\delta \in \Accord'$, the monogeneous geometric resolving set $\opResAc(\delta)$ is the resolving closure of $\{\delta\}$ by \cref{thm:res_clo_1}. We chose to use the same notation.
\end{remark}

Now we translate the $\Res$-relation on $\pmb{\ind \setminus \proj}(Q,R)$

\begin{definition} \label{def:GeomResrelation}
The \new{geometric Res-relation}, denoted by $\Accordleq$, is a binary relation on $\Accord'$ defined as follows:
\[\forall (\delta, \eta) \in (\Accord')^2,\  \delta \Accordleq \eta \Longleftrightarrow \opResAc(\delta) \subseteq \opResAc(\eta). \]
\end{definition}

\begin{lemma} \label{lem:ResleqandAccordleq}
The relation $\Accordleq$ is an order relation on $\Accord'$. Moreover, the posets $(\Accord', \Accordleq)$ and $(\pmb{\ind \setminus \proj}(Q,R), \Resleq)$ are isomorphic.
\end{lemma}

\begin{proof}
This is a direct consequence of \cref{thm:GeomandRep,thm:res_clo_1}.
\end{proof}

We can then translate the different results in \cref{ss:rescat} and get the following results.

\begin{prop} \label{prop:folkloreresaccord}
The following assertions hold:
\begin{enumerate}[label=$(\roman*)$,itemsep=1mm]
    \item The operator $\opResAc'$ is a closure operator on $\Accord'$;
    \item The poset $(\mathcal{P}_{\opResAc'}(\Accord'), \subseteq)$ is a subposet of $(\mathscr{J}(\Accord', \Accordleq), \subseteq)$;
    \item The lattice $(\mathcal{P}_{\opResAc'}(\Accord'), \subseteq, \cap, \veebar)$, where \[\forall (\mathfrak{B}, \mathfrak{C}) \in \mathcal{P}_{\opResAc'}(\Accord')^2,\ \mathfrak{B} \veebar \mathfrak{C} = \opResAc'(\mathfrak{B} \cup \mathfrak{C}),\] is isomorphic to $(\ResOrd(Q,R), \subseteq, \cap, \underline{\cup})$;
    \item The join-irreducible elements of $(\mathcal{P}_{\opResAc'}(\Accord'), \subseteq, \cap, \underline{\cup})$ are the principal ideals of $(\Accord', \Accordleq)$;
    \item All the ideals of $(\Accord', \Accordleq)$ are closed under non-projective syzygies, and both arrow and overlap extensions with projective accordions.
\end{enumerate}
\end{prop}

\begin{cor} \label{cor:maingeommotivation}
For any resolving set $\mathfrak{U} \subseteq \mathcal{P}_{\opResAc'}(\Accord')$, there exists a unique antichain $\mathfrak{E}_\mathfrak{U} \in \Anti(\Accord', \Accordleq)$ such that \[\mathfrak{U} = \bigcup_{\delta \in \mathfrak{E}_{\mathfrak{U}}} \opResAc'(\delta). \]
Moreover $\mathfrak{E}_\mathfrak{U}$ can be obtained from any $\mathfrak{D} \subseteq \Accord'$ such that $\opResAc(\mathfrak{D}) = \mathfrak{U} $ by applying \cref{algo:Rescatmotivation} adapted to the geometric model language.
\end{cor}

Therefore, this motivates us to describe $\mathfrak{E}_{\opResAc'(\delta,\eta)}$  for any pair $(\delta, \eta) \in \Accord'$ in the next sections.
	
	\section{Colorations}
	\label{sec:Matchable}
        The main goal is to describe all the resolving subcategories and their upper join decomposition in $\rep(Q,R)$ for any given gentle tree $(Q,R)$. Our idea is to use \cref{thm:EndAlgo}, and, via the geometric model, to compute $\mathfrak{E}_{\opResAc'(\delta,\eta)}$ explicitly for any pair $(\delta,\eta) \in \Accord'^2$. 

In this section, given such a pair $(\delta,\eta)$, we highlight combinatorial behaviors about the neighboring projective accordions and their respective colorations, called \emph{matchability}. It will capture whether there are direct or indirect interactions between $\delta$ and $\eta$. More precisely, we enumerate three main ones:
\begin{enumerate}[label=$\bullet$, itemsep=1mm]
    \item the case where $\NP(\delta) \cap \NP(\eta) = \varnothing$: it implies for any $\xi \Accordleq \delta$ and for any $\varsigma \Accordleq \eta$, there is neither arrow nor overlap extension between $\xi$ and $\varsigma$;
    \item the case where $\NP(\delta) \cap \NP(\eta) \neq \varnothing$, and $\col_\delta$ and $\col_\eta$ are not matchable: it implies that there is neither epimorphism nor extension ($\Ext^1$) between $\MM(\delta)$ and $\MM(\eta)$; and,
    \item the case where $\NP(\delta) \cap \NP(\eta) \neq \varnothing$, and $\col_\delta$ and $\col_\eta$ are  matchable.
\end{enumerate}
This section aims at showing that to computing $\opResAc(\delta,\eta)$ can be reduced to computing $\opResAc(\xi,\varsigma)$ where $(\xi,\varsigma) \in \Accord'^2$ is such that:
\begin{enumerate}[label=$\bullet$, itemsep=1mm]
    \item $\xi \Accordleq \delta$, and $\varsigma \Accordleq \eta$; and,
    \item $\col_\xi$ and $\col_\eta$ are matchable.
\end{enumerate}
In \cref{ss:weakmatchabilityconfig}, we focus on a specific case of matchability, called \emph{weak matchability}, which will be crucial to reach our goal.

\subsection{Matchable colorations}
\label{ss:Matchable}
Let us begin with the following definition.
\begin{definition} \label{def:matchablematch}
    Let $(\delta, \eta) \in \Accord'^2$. Consider a coloration $\col_\delta$ of $\NP(\delta)_0$, and  a coloration $\col_\eta$ of $\NP(\eta)_0$. Let $\rho \in \Prj(\Delta^{\gpoint})$ and denote by $v_1, v_2 \in \mathcal{M}_{\rpoint}$ its endpoints. We say that $\rho$ \new{makes the colorations matchable} whenever we are in one of the following configurations:
    \begin{enumerate}[label=$(\mathsf{M}\arabic*)$, itemsep=1mm]
        \item \label{M1} all the assertions hold:
    \begin{enumerate}[label=$\bullet$, itemsep=1mm]
        \item we have $\rho \in \NP(\delta) \cup \NP(\eta)$;
        \item there exists $v \in \{v_1,v_2\}$ such that $v \in \NP(\delta)_0$ and $\col_\delta(v) \in \{{\rsquare},{\gsquare}\}$; and,
        \item there exists $w \in \{v_1,v_2\}$ such that $w \in \NP(\eta)_0$ and $\col_\eta(w) \in \{{\rsquare},{\gsquare}\}$;
    \end{enumerate}
        \item \label{M2} up to exchanging the role of $\delta$ and $\eta$, all assertions hold:
        \begin{enumerate}[label=$\bullet$, itemsep=1mm] 
         \item the accordion $\eta$ is contained in the source or the target cell $C$ of $\delta$; 
         \item the projective accordion $\rho$ is in $\partial C$; and,
         \item $\rho$ and $\eta$ share a common endpoint $v \in \mathcal{M}_{\rpoint}$, which is not an endpoint of $\delta$, and $\eta \prec_v \rho$.
         \end{enumerate}
    \end{enumerate}
    We say that $\col_\delta$ and $\col_\eta$ are \new{matchable} whenever $\NP(\delta) \cap \NP(\eta) \neq \varnothing$ and there exists $\rho \in \NP(\delta)\cup\NP(\eta)$  that makes the colorations matchable. In such a case we say that $\col_\delta$ and $\col_\eta$ are:
    \begin{enumerate}[label=$\bullet$, itemsep=1mm]
        \item \new{strongly} matchable whenever there exists $\rho \in \NP(\delta)\cap\NP(\eta)$ that makes the colorations matchable; and,
        \item \new{weakly} matchable otherwise.
    \end{enumerate}
\end{definition}

\begin{remark}\label{rem:M2}
    If $\rho \in \NP(\delta) \cup \NP(\eta)$ satisfies \ref{M2}, then:
    \begin{enumerate}[label=$\bullet$, itemsep=1mm]
        \item $\NP(\delta) \subset \NP(\eta)$, or $\NP(\eta) \subset \NP(\delta)$; and,
        \item $\rho \in \NP(\delta) \cap \NP(\eta)$.
    \end{enumerate}
\end{remark}

\begin{definition} \label{def:match}
We say that $\col_\delta$ and $\col_\eta$ \new{match} whenever, $\NP(\delta) \cap \NP(\eta) \neq \varnothing$, and for any $v\in\NP(\delta)_0\cap\NP(\eta)_0$, one of the following assertions holds:
\begin{enumerate}[label = $\bullet$,itemsep=1mm]
 \item $\col_\delta(v) = \col_\eta(v)$; or,
 \item ${\osquare} \in \{\col_\delta(v)\} \cup \{\col_\eta(v)\}$; or,
 \item ${\psquare} \in \{\col_\delta(v)\} \cup \{\col_\eta(v)\}$.
\end{enumerate}
\end{definition}

\begin{lemma} \label{lem:coincidecoloration}
Let $(\delta,\eta) \in \Accord'^2$. Fix the coloration $\col_\delta$ of $\NP(\delta)_0$ and $\col_\eta$ of $\NP(\eta)_0$. If $\col_\delta$ and $\col_\eta$ are matchable, then either $\col_\delta$ and $\col_\eta$ match, or $\col_\delta$ and $\overline{\col_\eta}$ match.
\end{lemma}
\begin{proof}
By \cref{lem:uniqcolandlaterality}, there are only two colorations of $\NP(\delta)_0$, and only two colorations of $\NP(\eta)_0$. Let $\col_\delta$ be a coloration of $\NP(\delta)_0$, and $\col_\eta$ be one of $\NP(\eta)_0$. Assume that $\NP(\delta) \nsubseteq \NP(\eta)$ and $\NP(\delta) \nsupseteq \NP(\eta)$. There exists $\rho \in \NP(\delta) \cup \NP(\eta)$ that makes the colorations matchable. At least one of the endpoints of $\rho$ is red or green for $\col_\delta$, and one is red or green for $\col_\eta$. One can distinguish between a top and a bottom endpoint. If the laterality of this arc induced by $\col_\delta$ and $\col_\eta$ are the same, then we consider $\col_\delta$ and $\col_\eta$  otherwise, $\col_\delta$ and $\overline{\col_\eta}$. Once this choice is made, we suppose that $\col_\eta$ was the good one. Consider a point $v\in \NP(\delta)_0\cap\NP(\eta)_0$ such that $\col_\delta(v)=\gsquare$ and $\col_\eta(v)=\rsquare$. Note that as the neighbouring projectives are connected, there exists two paths of projectives starting at $v$ and reaching $\NP(\delta)\cap\NP(\eta)$ through $\NP(\delta)$ and $\NP(\eta)$ respectively. This is impossible as the surface is a disc without an inner cell and these two paths form a loop. The case where $\NP(\delta) \subset \NP(\eta)$ can be treated similarly 
\end{proof}
\begin{conv} \label{conv:matchablematching}
In the following, each time we consider a pair $(\delta,\eta)$ of non-projective accordions whose colorations $\col_\delta$ and $\col_\eta$ are matchable, we assume that $\col_\delta$ and $\col_\eta$ match. In this case, an induced compatible laterality on $(\pmb{\Sigma}, \mathcal{M}, \Prj{\Delta^{\circ}})$ allows us to distinguish a source $s(\mu)$ and a target $t(\mu)$, for any $\rpoint$-arc $\mu$ whose endpoints are in $\NP(\delta)_0 \cup \NP(\eta)_0$ by choosing that the source is on the left and the target on the right provided that the red points are above and the green below.
\end{conv}

Now we show that if $\delta$ and $\eta$ interact, meaning there are either epimorphisms, overlap, or arrow extensions between them, then $\col_\delta$ and $\col_\eta$ are matchable. This will be helpful in \cref{ss:nomatchablemoves}.

Let us first state the following lemma.

\begin{lemma} \label{lem:crossprojmatchable}
    Let $(\delta,\eta) \in \Accord'^2$ be such that there exists $\rho \in \Prj(\Delta^{\gpoint})$ which crosses both $\delta$ and $\eta$. Then $\rho$ makes the colorations matchable.
\end{lemma}

\begin{proof}
    This is an obvious outcome of \cref{def:colourendpoints}.
\end{proof}

\begin{prop} \label{prop:matchableextorepi}
Let $(\delta,\eta) \in \Accord'^2$ be such that $\NP(\delta) \cap \NP(\eta) \neq \varnothing$. Assume that one of the following assertions holds:
\begin{enumerate}[label=$(\roman*)$, itemsep=1mm]
    \item there exists an epimorphism $\begin{tikzcd}
	{f: \MM(\eta) } & \MM(\delta);
	\arrow[two heads, from=1-1, to=1-2]
\end{tikzcd}$ or,
    \item either $\OvExt(\delta,\eta) \neq \varnothing$ or $\ArExt(\delta,\eta) \neq \varnothing$.
    \end{enumerate}
    Then $\col_\delta$ and $\col_\eta$ are matchable.
\end{prop}

\begin{proof}
By \cref{prop:geo_mor,prop:geom_ext}, one of the following holds:
\begin{enumerate}[label=$\bullet$,itemsep=1mm]
    \item $\delta$ and $\eta$ cross; or,
    \item $\delta$ and $\eta$ share a common endpoint $v$, and $\eta \prec_v \delta$.
\end{enumerate}

If $\delta$ and $\eta$ cross, then there exists $\rho \in \Prj(\Delta^{\gpoint})$ which crosses both $\delta$ and $\eta$, and by \cref{lem:crossprojmatchable}, we get that $\col_\delta$ and $\col_\eta$ are matchable.

Assume that $\delta$ and $\eta$ do not cross. Consequently, $\OvExt(\delta,\eta) = \varnothing$.  If we still have an epimorphism $\begin{tikzcd}
	{f: \MM(\eta) } & \MM(\delta),
	\arrow[two heads, from=1-1, to=1-2]
\end{tikzcd}$ then there exists $\rho \in \Prj(\Delta^{\gpoint})$ which crosses both $\delta$ and $\eta$, and we get the desired result. So now assume that $\ArExt(\delta,\eta) \neq \varnothing$, and denote by $\kappa$ the unique accordion in $\ArExt(\delta,\eta)$. Then $\delta$ and $\eta$ share a common endpoint $v$, and $\eta \prec_v \delta$

Fix $\col_\delta$ and $\col_\eta$ such that the induced lateralities imply that $v = s(\delta) = s(\eta)$. By denoting $u$ the other endpoint of $\delta$ and $w$ the other endpoint of $\eta$, recall that $\kappa$ is the accordion whose endpoints are $u$ and $w$. Whenever $\delta$ is contained in a cell of $\pmb{\Gamma}(\Prj(\Delta^{\gpoint}))$, we denote this cell by $C_\delta$; otherwise, we set $(\mu,C_\delta) = \sc(\delta)$.  We define $C_\eta$ and $\nu$ relatively to $\eta$ in a similar way.

Let us first treat the case where $C_\delta = C_\eta = C$. See \cref{fig:Samecellcasematch} to visualize the general configuration. Then:
\begin{enumerate}[label=$\bullet$, itemsep=1mm]
    \item if $\delta,\eta \subset C$, as $\eta \prec_v \delta$, we have that $\NP(\eta) \subseteq \NP(\delta)$;
    \item if $\delta \not\subset C$ and $\eta \subset C$, then, as $\kappa \in \Accord$,  $w$ must be an endpoint of $\mu$; in such a case, we have $\NP(\eta) \subseteq \NP(\delta)$;
    \item if $\delta \subset C$ and $\eta \not\subset C$, then $u$ must be an endpoint of $\nu$; in such a case, $\nu$ makes $\col_\delta$ and $\col_\eta$ matchable;
    \item otherwise, $\delta \not\subset C$ and $\eta \not\subset C$; as $\kappa \in \Accord$, this means that $\mu$ and $\nu$ must have a common endpoint and, by the definition of $\col_\delta$ and $\col_\eta$, and as $\eta \prec_v \delta$, this implies that $\mu$  makes $\col_\delta$ and $\col_\eta$ matchable.
\end{enumerate}
     \begin{figure}[!ht]
\centering 
    \begin{tikzpicture}[mydot/.style={
					circle,
					thick,
					fill=white,
					draw,
					outer sep=0.5pt,
					inner sep=1pt
				}, scale = 1.2]
		\tikzset{
		osq/.style={
        rectangle,
        thick,
        fill=white,
        append after command={
            node [
                fit=(\tikzlastnode),
                orange,
                line width=0.3mm,
                inner sep=-\pgflinewidth,
                cross out,
                draw
            ] {}}}}
			\foreach \X in {0,1,...,37}
		{
		\tkzDefPoint(4*cos(pi/19*\X),1.5*sin(pi/19*\X)){\X};
		};
        \draw [line width=0.7mm,domain=30:45] plot ({4*cos(\x)}, {1.5*sin(\x)});
        \draw [line width=0.7mm,domain=80:90] plot ({4*cos(\x)}, {1.5*sin(\x)});
        \draw [line width=0.7mm,domain=110:120] plot ({4*cos(\x)}, {1.5*sin(\x)});
        \draw [line width=0.7mm,domain=145:155] plot ({4*cos(\x)}, {1.5*sin(\x)});
        \draw [line width=0.7mm,domain=180:195] plot ({4*cos(\x)}, {1.5*sin(\x)});
        \draw [line width=0.7mm,domain=210:260] plot ({4*cos(\x)}, {1.5*sin(\x)});
        \draw [line width=0.7mm,domain=280:290] plot ({4*cos(\x)}, {1.5*sin(\x)});
        \draw [line width=0.7mm,domain=335:350] plot ({4*cos(\x)}, {1.5*sin(\x)});

		\draw[line width=0.9mm ,bend right=40,red](23) edge (20);
		\draw[line width=0.9mm ,bend right =30,red](20) edge (16);
            \draw[line width=0.9mm ,bend right =30,red](16) edge (12);
            \draw[line width=0.9mm ,bend right =20,red](12) edge (4);
            \draw[line width=0.9mm ,bend right =20,red](4) edge (30);
            \draw[line width=0.9mm ,bend right =30,red](30) edge (27);
		
		\draw[line width=0.7mm,blue,bend left=20, loosely dashed](9) edge (23);
            \draw[line width=0.7mm,blue,bend right=10, loosely dashed](36) edge (23);
		
		\draw [line width=0.7mm, mypurple,dash pattern={on 10pt off 2pt on 5pt off 2pt}, bend right=10] (9) edge (36);

            \filldraw [fill=red,opacity=0.1] (23) to [bend right=40] (20) to [bend right=30] (16) to [bend right=30] (12) to [bend right=20] (4)  to [bend right=20] (30) to [bend right=30] (27) to [bend left=10] cycle ;

		\foreach \X in {4,9,12,16,20,23,27,30,36}
		{
		\tkzDrawPoints[fill =red,size=4,color=red](\X);
		};

		\tkzDefPoint(-1.2,-0.6){g};
		\tkzLabelPoint[blue](g){\Large $\delta$}
            \tkzDefPoint(-1.5,.3){h};
		\tkzLabelPoint[blue](h){\Large $\eta$}
            \tkzDefPoint(-1,1){i};
		\tkzLabelPoint[red](i){\Large $\nu$}
            \tkzDefPoint(1,-0.7){j};
		\tkzLabelPoint[red](j){\Large $\mu$}
            \tkzDefPoint(1.3,1.3){j};
		\tkzLabelPoint[mypurple](j){\Large $\kappa$}
            \tkzDefPoint(-3,0.5){k};
		\tkzLabelPoint[red](k){\huge $C$}
            \tkzDefPoint(-3.3,-.9){l};
		\tkzLabelPoint[red](l){\Large $v$}
            \tkzDefPoint(4,-.4){l};
		\tkzLabelPoint[red](l){\Large $u$}
            \tkzDefPoint(.35,1.9){l};
		\tkzLabelPoint[red](l){\Large $w$}
    \end{tikzpicture}
\caption{\label{fig:Samecellcasematch} The case where $C = C_\delta = C_\eta$ and $\delta, \eta \nsubseteq C$. As $\kappa \in \Accord$, the projective accordions $\mu$ and $\nu$ must share a common endpoint.}
\end{figure}
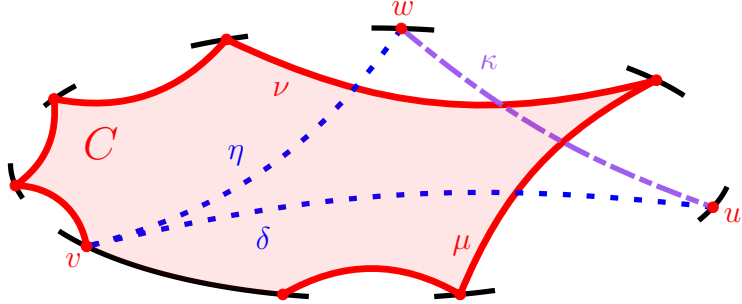
In either case, if $C_\delta = C_\eta$, then $\col_\delta$ and $\col_\eta$ are matchable.

Now, assume that $C_\delta \neq C_\eta$. There exists $\rho \in \Prj(\Delta^{\gpoint})$ such that $\eta \prec_v \rho \prec_v \delta$. Let $\rho_{\min}$ and $\rho_{\max}$ be respectively the minimal and the maximal projective accordion such that $\eta \prec_v \rho_{\min} \preccurlyeq_v \rho_{\max} \prec_v \delta$. Note that $\rho_{\min} \subset \partial C_\eta$, and $\rho_{\max} \subset \partial C_\delta$. Denote by $v_{\min}$ the other endpoint of $\rho_{\min}$, and $v_{\max}$ the other endpoint of $\rho_{\max}$ (see \cref{fig:Distinctcellcasematch})
\begin{figure}[!ht]
\centering 
    \begin{tikzpicture}[mydot/.style={
					circle,
					thick,
					fill=white,
					draw,
					outer sep=0.5pt,
					inner sep=1pt
				}, scale = 1.2]
		\tikzset{
		osq/.style={
        rectangle,
        thick,
        fill=white,
        append after command={
            node [
                fit=(\tikzlastnode),
                orange,
                line width=0.3mm,
                inner sep=-\pgflinewidth,
                cross out,
                draw
            ] {}}}}
			\foreach \X in {0,1,...,37}
		{
		\tkzDefPoint(4*cos(pi/19*\X),1.5*sin(pi/19*\X)){\X};
		};
        \draw [line width=0.7mm,domain=30:45] plot ({4*cos(\x)}, {1.5*sin(\x)});
        \draw [line width=0.7mm,domain=80:90] plot ({4*cos(\x)}, {1.5*sin(\x)});
        \draw [line width=0.7mm,domain=110:120] plot ({4*cos(\x)}, {1.5*sin(\x)});
        \draw [line width=0.7mm,domain=145:155] plot ({4*cos(\x)}, {1.5*sin(\x)});
        \draw [line width=0.7mm,domain=215:225] plot ({4*cos(\x)}, {1.5*sin(\x)});
        \draw [line width=0.7mm,domain=250:260] plot ({4*cos(\x)}, {1.5*sin(\x)});
        \draw [line width=0.7mm,domain=280:290] plot ({4*cos(\x)}, {1.5*sin(\x)});
        \draw [line width=0.7mm,domain=335:350] plot ({4*cos(\x)}, {1.5*sin(\x)});

		\draw[line width=0.9mm ,bend right=5,densely dashdotted,red](23) edge (9);
		\draw[line width=0.9mm ,bend right =0,red](23) edge (4);
            \draw[line width=0.9mm ,bend left =5, densely dashdotted, red](23) edge (36);
            \draw[line width=0.9mm ,bend right =5,dashdotted,red](36) edge (27);
            \draw[line width=0.9mm ,bend left =5, dashdotted,red](9) edge (16);
		
		\draw[line width=0.7mm,blue,bend left=20,     
            loosely dashed](12) edge (23);
            \draw[line width=0.7mm,blue,bend right=10, loosely dashed](30) edge (23);
		
		\draw [line width=0.7mm, mypurple,dash pattern={on 10pt off 2pt on 5pt off 2pt}, bend right=10] (12) edge (30);

            \filldraw [fill=red,opacity=0.1] (23) to [bend right=5] (9) to [bend left=5] (16) to [bend right=40] cycle ;

            \filldraw [pattern=dots,pattern color=red,opacity=0.3] (23) to [bend left=5] (36) to [bend right=5] (27) to [bend left=40] cycle ;

		\foreach \X in {4,9,12,16,23,27,30,36}
		{
		\tkzDrawPoints[fill =red,size=4,color=red](\X);
		};

		\tkzDefPoint(-1,-1){g};
		\tkzLabelPoint[blue](g){\Large $\delta$}
            \tkzDefPoint(-2.5,.3){h};
		\tkzLabelPoint[blue](h){\Large $\eta$}
            \tkzDefPoint(2,1){i};
		\tkzLabelPoint[red](i){\Large $\rho$}
            \tkzDefPoint(0.4,1.2){j};
		\tkzLabelPoint[red](j){\Large $\rho_{\min}$}
            \tkzDefPoint(2,0){j};
		\tkzLabelPoint[red](j){\Large $\rho_{\max}$}
            \tkzDefPoint(-0.5,0.5){j};
		\tkzLabelPoint[mypurple](j){\Large $\kappa$}
            \tkzDefPoint(-3.2,0.5){k};
		\tkzLabelPoint[red](k){\huge $C_\eta$}
            \tkzDefPoint(-1.8,-1){l};
		\tkzLabelPoint[red](l){\huge $C_\delta$}
            \tkzDefPoint(-3.3,-.9){l};
		\tkzLabelPoint[red](l){\Large $v$}
            \tkzDefPoint(1,-1.5){l};
            \tkzLabelPoint[red](l){\Large $u$}
            \tkzDefPoint(-1.7,1.8){l};
            \tkzLabelPoint[red](l){\Large $w$}
            \tkzDefPoint(0.4,2){l};
		\tkzLabelPoint[red](l){\Large $v_{\min}$}
            \tkzDefPoint(4.3,-.4){l};
		\tkzLabelPoint[red](l){\Large $v_{\max}$}
            \tkzDefPoint(-2.6,1.2){l};
		\tkzLabelPoint[red](l){\Large $\nu$}
            \tkzDefPoint(2.3,-.7){l};
		\tkzLabelPoint[red](l){\Large $\mu$}
    \end{tikzpicture}
\caption{\label{fig:Distinctcellcasematch} The case where $C_\delta \neq C_\eta$, $\delta \not\subset C_\delta$ and $\eta \not\subset C_\eta$. We note that $\rho_{\min} \subseteq \partial C_\eta$ and $\rho_{\max} \subseteq \partial C_\delta$}
\end{figure}
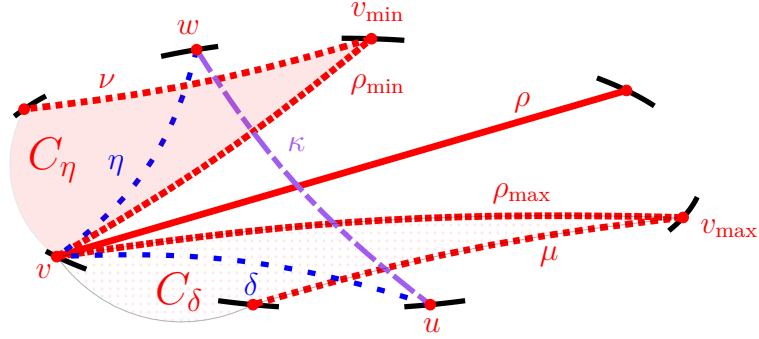

Assume that $\delta \not\subset C_\delta$ and $\eta \not\subset C_\eta$. We have $\rho_{\min} \in \NP(\delta) \cap \NP(\eta)$, and by definition of $\col_\delta$ and $\col_\eta$, the projective accordion $\rho_{\min}$ makes the coloration matchable. These arguments hold whenever $\delta \subset C_\delta$ or $\eta \subset C_\eta$. 
\end{proof}

Given $(\delta, \eta) \in \Accord'^2$ such that $\NP(\delta) \cap \NP(\eta) = \varnothing$, the accordions are too ``far away" from each other to interact when computing resolving closure. The following result accurately states this fact.

\begin{cor} \label{prop:noNPs}
Let $(\delta, \eta) \in  \Accord'^2$. Assume that $\NP(\delta) \cap \NP(\eta) = \varnothing$. Then we have $\mathfrak{E}_{\opResAc'(\delta,\eta)} = \{\delta, \eta\}$.
\end{cor}

\begin{proof}
We show that $\opResAc(\delta, \eta) = \opResAc(\delta) \cup \opResAc(\eta)$, which is equivalent to the desired result. By \cref{lem:IdealsclosedSyzygies}, we only have to check that $\opResAc(\delta) \cup \opResAc(\eta)$ is closed under extensions. Both $\opResAc(\delta)$ and $\opResAc(\eta)$ are closed under extensions by \cref{thm:res_clo_1}. 

Let $\xi \Accordleq \delta$ and $\varsigma \Accordleq \eta$. Denote by $\col_\xi$ and $\col_\varsigma$ colorations of respectively $\NP(\xi)_0$ and $\NP(\varsigma)_0$. This implies that $\NP(\xi) \subseteq \NP(\delta)$ and $\NP(\varsigma) \subseteq \NP(\eta)$. Therefore, $\NP(\xi) \cap \NP(\varsigma) = \varnothing$, and $\col_\xi$ and $\col_\varsigma$ are not matchable. By \cref{prop:matchableextorepi}, there is neither overlap nor arrow extensions between $\xi$ and $\varsigma$. By additivity of $\Ext^1$ there are no non trivial extensions between $M\in\Res (\MM(\delta))$ and $N\in\Res(\MM(\eta))$. So $\opResAc(\delta) \cup \opResAc(\eta)$ is closed under extensions.
\end{proof}

\subsection{Weakly matchable configurations}
\label{ss:weakmatchabilityconfig}

In this section, we focus on the matchability configurations. Most of the time, given $\delta, \eta \in \Accord'$ such that $\NP(\delta) \cap \NP(\eta) \neq \varnothing$, the matchability of $\col_\delta$ and $\col_\eta$ is in the scope of strong matchability. We highlight under which specific conditions the matchability, called weak matchability, is only guaranteed by some $\rho \notin \NP(\delta) \cap \NP(\eta)$. 
By \cref{lem:coincidecoloration} and \cref{conv:matchablematching}, we consider two colorations $\col_\delta$ and $\col_\eta$ that are weakly matching.

Let us first prove the following lemmas on the structure of the neighboring projective accordions of some $\delta \in \Accord'$, which are obvious consequences of both the definition of $\NP(\delta)$ and the fact that we are working in a marked disc $(\pmb{\Sigma}, \mathcal{M})$ without punctures.

\begin{definition} \label{def:connectedaccordions}
    Let $\mathfrak{D}, \mathfrak{P} \subset \Accord$ such that $\mathfrak{D} \subseteq \mathfrak{P}$. We say that $\mathfrak{D}$ is $\mathfrak{P}$-\new{fan-connected} if:
    \begin{enumerate}[label=$\bullet$, itemsep=1mm]
        \item $\mathfrak{D}$ is connected; and, 
        \item for any $(\delta_1, \delta_2) \in \mathfrak{D}$ sharing a common endpoint $v \in \mathcal{M}_{\rpoint}$, if $\delta_1 \prec_v \delta_2$, then all the accordions $\mu \in \mathfrak{P}$ with $v$ as an endpoint, and such that $\delta_1 \preccurlyeq_v \mu \preccurlyeq_v \delta_2$, are in $\mathfrak{D}$.
    \end{enumerate}
\end{definition}

\begin{lemma} \label{lem:connectedNP}
    Let $\delta \in \Accord'$. Then $\NP(\delta)$ is $\Prj(\Delta^{\gpoint})$-fan-connected. 
\end{lemma}

\begin{lemma} \label{lem:inbetweenNP}
   Let $\mathfrak{P} \subseteq \Accord$ and $\mathfrak{D}_1,\mathfrak{D}_2 \subseteq \mathfrak{P}$. Assume that:
   \begin{enumerate}[label=$\bullet$, itemsep=1mm]
       \item $\mathfrak{P}$ is a $\rpoint$-dissection without inner cell of $(\pmb{\Sigma}, \mathcal{M})$; and,
       \item $\mathfrak{D}_1$ and $\mathfrak{D}_2$ are $\mathfrak{P}$-fan-connected.
   \end{enumerate}
   Then the following assertions hold:
   \begin{enumerate}[label=$(\roman*)$, itemsep=1mm]
       \item $\mathfrak{D}_1 \cap \mathfrak{D}_2$ is $\mathfrak{P}$-fan-connected; 
       \item  if $\mathfrak{D}_1 \cap \mathfrak{D}_2 \neq \varnothing$, then $\mathfrak{D}_1 \cup \mathfrak{D}_2$ is $\mathfrak{P}$-fan-connected; and,
       \item if $\mathfrak{D}_1 \cap \mathfrak{D}_2 \neq \varnothing$, for any $\rho \in \mathfrak{D}_1$ and $\rho' \in \mathfrak{D}_2$ sharing a common endpoint $v \in \mathcal{M}_{\rpoint}$ such that $\rho \prec_v \rho'$, there exists $\mu \in \mathfrak{D}_1 \cap \mathfrak{D}_2$ having $v$ as an endpoint such that $\rho \prec_v \mu \prec_v \rho'$.
   \end{enumerate}
\end{lemma}

\begin{proof}
    The assertion $(i)$ is obvious. Assume by contradiction that there exist a marked point $v \in \mathcal{M}_{\rpoint}$ and three accordions $\delta_1, \delta_2 \in \mathfrak{D}_1 \cup \mathfrak{D}_2$ and $\mu \in \mathfrak{P} \setminus (\mathfrak{D}_1 \cup \mathfrak{D}_2)$ sharing $v$ as a common endpoint such that $\delta_1 \prec_v \mu \prec_v \delta_2$. Assume that $\delta_1 \in \mathfrak{D}_1$ and $\delta_2 \in \mathfrak{D}_2$. The accordion $\mu$ cuts the ambient disc $(\pmb{\Sigma}, \mathcal{M})$ into two distinct connected components. As $\mathfrak{D}_1$ and $\mathfrak{D}_2$ are $\mathfrak{P}$-fan-connected, all the accordions of $\mathfrak{D}_1$ belong to one of the connected components, and all the ones of $\mathfrak{D}_2$ belong to the other one. This is in contradiction with the fact that $\mathfrak{D}_1 \cap \mathfrak{D}_2 \neq \varnothing$. We get the assertion $(ii)$.

    Assume by contradiction that the assertion $(iii)$ does not hold. We can reduce the situation to the case where $\delta_1 \in \mathfrak{D}_1 \setminus \mathfrak{D}_2$ and $ \delta_2 \in \mathfrak{D}_2 \setminus \mathfrak{D}_1 $ share a common endpoint $v \in \mathcal{M}_{\rpoint}$, and $\delta_1 \precdot_v \delta_2$ in $\mathfrak{P}$. In such a case, $\delta_1$ and $\delta_2$ are in the boundary of a common cell $C \in \pmb{\Gamma}(\mathfrak{P})$. This cell contains a unique marked point $w \in \mathcal{M}_{\gpoint}$ which is also in $\partial \pmb{\Sigma}$. Consider a curve $\vartheta$ whose endpoints are $v$ and $w$. As in the proof of $(ii)$, $\vartheta$ cuts the marked disc $(\pmb{\Sigma}, \mathcal{M})$ into two connected components, and all the accordions of $\mathfrak{D}_1$ belong to one of the connected components, and all the ones of $\mathfrak{D}_2$ belong to the other one, by $\mathfrak{P}$-fan-connectivity. This is in  contradiction with the fact that $\mathfrak{D}_1 \cap \mathfrak{D}_2 \neq \varnothing$. We get the assertion $(iii)$.
\end{proof}

We now prove a first technical lemma that illustrates the colorations of the endpoints of $\rho \in \NP(\delta) \setminus \NP(\eta)$, matching the colorations.

\begin{lemma} \label{lem:angleconfigNPscolor}
    Let $\delta,\eta \in \Accord'$. Consider $\col_\delta$ and $\col_\eta$ colorations of respectively $\NP(\delta)_0$ and $\NP(\eta)_0$. Assume that $\col_\delta$ and $\col_\eta$  weakly match, and there exists $\rho \in \NP(\delta) \setminus \NP(\eta)$ that makes the colorations matchable. Then:
    \begin{enumerate}[label=$\bullet$, itemsep=1mm]
        \item  exactly one endpoint $v$ of $\rho$ belongs to in $\NP(\eta)_0$; and,
        \item  $\col_\eta(v) \in \{{\rsquare},{\gsquare}\}$, and $\col_\delta(v) = {\osquare}$.
    \end{enumerate}
\end{lemma}

\begin{proof}
    As $\rho$ makes the colorations matchable, one of the endpoints $v$ of $\rho$ must be in $\NP(\eta)_0$, with $\col_\eta(v) \in \{{\rsquare},{\gsquare}\}$. As $\rho \notin \NP(\eta)$, by \cref{lem:connectedNP}, exactly one of its endpoints is in $\NP(\eta)_0$. 

    Consider $\rho' \in \NP(\eta)$ such that $v$ is one of its endpoints. Assume, without loss of generality, that $\rho \prec_v \rho'$. By \cref{lem:inbetweenNP}, there exists $\mu \in \NP(\delta) \cap \NP(\eta)$ such that:
    \begin{enumerate}[label=$\bullet$, itemsep=1mm]
        \item $v$ is an endpoint of $\mu$; and,
        \item $\rho \prec_v \mu \prec_v \rho'$.
    \end{enumerate}
    By assumption, we have that $\mu$ does not make the colorations matchable and, therefore, $\col_\delta(v) \in \{{\osquare},{\psquare}\}$. Moreover, by letting $u$ be the other endpoint of $\rho$, we have $\col_\delta(u) \in \{{\gsquare},{\rsquare}\}$. By definition of $\col_\delta$, there is no projective accordions whose endpoints are both colored ${\psquare}$ and ${\gsquare}$, or ${\rsquare}$. Thus $\col_\delta(v) = {\osquare}$.
\end{proof}

\begin{lemma} \label{lem:angleconfigNPsplace}
    Let $\delta,\eta \in \Accord'$. Consider two colorations  $\col_\delta$ and $\col_\eta$  of  $\NP(\delta)_0$ and $\NP(\eta)_0$ respectively. Assume that $\col_\delta$ and $\col_\eta$  weakly match, and that there exists $\rho \in \NP(\delta) \setminus \NP(\eta)$ making the colorations matchable. Let $v$ be the endpoint of $\rho$ in $\NP(\eta)_0$. Then either:
    \begin{enumerate}[label=$\bullet$, itemsep=1mm]
        \item any $\rho' \in \NP(\eta)$ having $v$ as an endpoint is such that $\rho \prec_v \rho'$; or,
        \item  any $\rho' \in \NP(\eta)$ having $v$ as an endpoint is such that $\rho' \prec_v \rho$.
    \end{enumerate}
\end{lemma}

\begin{proof}
    This is direct consequence of the fact that $\NP(\eta)$ is $\Prj(\Delta^{\gpoint})$-fan-connected, and the fact that $\rho \notin \NP(\eta)$.
\end{proof}

\begin{lemma} \label{lem:angleconfigNPsnonmatchable}
    Let $\delta,\eta \in \Accord'$. Consider two colorations  $\col_\delta$ and $\col_\eta$  of $\NP(\delta)_0$ and $\NP(\eta)_0$ respectively. Assume that $\col_\delta$ and $\col_\eta$ weakly match, and that there exists $\rho \in \NP(\delta) \setminus \NP(\eta)$ that makes the colorations matchable. Let $v$ be the endpoint of $\rho$ in $\NP(\eta)_0$. Assume that, for any projective accordion $\rho' \in \NP(\eta)$ having $v$ as an endpoint, we have $\rho \prec_v \rho'$. Then any $\mu \in \NP(\delta) \setminus \NP(\eta)$ such that $\rho \prec_v \mu \prec_v \rho'$ makes the colorations matchable.
\end{lemma}

\begin{proof}
    Let $\col_\delta$ and $\col_\eta$ be colorations of  $\NP(\delta)_0$ and $\NP(\eta)_0$ respectively that match. Call $u$ the other endpoint of $\rho$. By \cref{lem:angleconfigNPscolor}, we can assume that $\col_\eta(v) = {\rsquare}$. We also have $\col_\delta(v) = {\osquare}$ and $\col_\delta(u) = {\gsquare}$. 
    
    Assume by contradiction that such a $\mu \in \NP(\delta) \setminus \NP(\eta)$ exists. Let $\rho_+$ be the smallest projective accordion for $\prec_v$ among those in $\NP(\eta)$ having $v$ as an endpoint. By \cref{lem:angleconfigNPsplace}, we have that:
    \begin{enumerate}[label=$\bullet$, itemsep=1mm]
        \item $\mu \prec_v \rho_+$; and,
        \item $\rho_+ \in \NP(\delta) \cap \NP(\eta)$.
    \end{enumerate}
    Set $x \in \mathcal{M}_{\rpoint}$ to be the other endpoint of $\mu$ and $y \in \mathcal{M}_{\rpoint}$ the one of $\rho_+$. See \cref{fig:specialmatchreduction} for an illustration.
    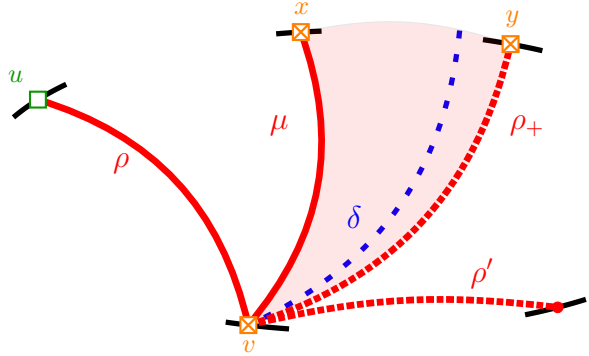
\begin{figure}[!ht]
        \centering
        \begin{tikzpicture}[mydot/.style={
					circle,
					thick,
					fill=white,
					draw,
					outer sep=0.5pt,
					inner sep=1pt
				}, scale = 1]
		\tikzset{
		osq/.style={
        rectangle,
        thick,
        fill=white,
        append after command={
            node [
                fit=(\tikzlastnode),
                orange,
                line width=0.3mm,
                inner sep=-\pgflinewidth,
                cross out,
                draw
            ] {}}}}
		\draw [line width=0.7mm,domain=60:70] plot ({5*cos(\x)}, {2*sin(\x)});
		\draw [line width=0.7mm,domain=95:102] plot ({5*cos(\x)}, {2*sin(\x)});
		\draw [line width=0.7mm,domain=140:155] plot ({5*cos(\x)}, {2*sin(\x)});
		\draw [line width=0.7mm,domain=250:260] plot ({5*cos(\x)}, {2*sin(\x)});
		\draw [line width=0.7mm,domain=297:308] plot ({5*cos(\x)}, {2*sin(\x)});
		\foreach \X in {0,...,43}
		{
		\tkzDefPoint(5*cos(pi/22*\X),2*sin(pi/22*\X)){\X};
		};

		\draw[line width=0.7mm,blue, loosely dashed] (1.4,2) to [bend left=30] (31);
		\draw[line width=0.9mm ,bend right=30,red](31) edge (18);
		\draw[line width=0.9mm ,bend right=30,red](31) edge (12);
		\draw[line width=0.9mm ,bend right=30,red, densely dashdotted](31) edge (8);
		\draw[line width=0.9mm ,bend left=10,red, densely dashdotted](31) edge (37);
		
		\filldraw [fill=red,opacity=0.1] (31) to [bend right=30] (8) to [bend right=15] (12) to [bend left=30] cycle ;
		
		\foreach \X in {37}
		{
		\tkzDrawPoints[fill =red,size=4,color=red](\X);
		};
		\foreach \X in {18}
		{
		\tkzDrawPoints[rectangle,size=6,color=dark-green,thick,fill=white](\X);
		};
		\foreach \X in {8,12,31}
		{
		\tkzDrawPoints[size=6,orange,osq](\X);
		};

		\tkzDefPoint(-3.1,0.5){gammaM};
		\tkzLabelPoint[red](gammaM){\Large $\rho$}
		\tkzDefPoint(-1,1){gammaP};
		\tkzLabelPoint[red](gammaP){\Large $\mu$}
		\tkzDefPoint(2.3,1){v};
		\tkzLabelPoint[red](v){\Large $\rho_+$}
		\tkzDefPoint(1.7,-.9){w};
		\tkzLabelPoint[red](w){\Large $\rho'$}
		\tkzDefPoint(0,-0.2){deltav};
		\tkzLabelPoint[blue](deltav){\Large $\delta$}
		\tkzDefPoint(-4.5,1.6){deltaw};
		\tkzLabelPoint[dark-green](deltaw){$u$}
		\tkzDefPoint(-1.4,-2){w};
		\tkzLabelPoint[orange](w){$v$}
            \tkzDefPoint(-.7,2.5){x};
		\tkzLabelPoint[orange](x){$x$}
            \tkzDefPoint(2.1,2.4){y};
		\tkzLabelPoint[orange](y){$y$}
    \end{tikzpicture}
        \caption{Illustration of the configuration. We assigned the colorations $\col_\delta$ to the relevant vertices. Note that $\col_\eta(v) = {\rsquare}$ and $\col_\delta(x),\col_\delta(y) \in \{{\osquare},{\psquare}\}$.}
        \label{fig:specialmatchreduction}
    \end{figure}
    By assumption on $\mu$, we have $\col_\delta(x) \in \{{\osquare}, {\psquare}\}$. As there is no projective accordion in $\NP(\delta) \cap \NP(\eta)$ that makes the colorations matchable, we have that  $\col_\delta(y) \in \{{\osquare},{\psquare}\}$. So $\mu$ and $\rho_+$ are boundaries of a source or a target cell of $\delta$. This implies that $\mu \precdot_v \rho_+$ in $\Prj(\Delta^{\gpoint})$. Moreover, as $\rho \prec_v \mu$ and $\col_\delta(u) = {\gsquare}$, then $v = t(\delta)$, and $\mu \prec_v \delta \prec_v \rho_+$ in $\Accord$. By constructing $\col_\delta$ via \cref{def:colourendpoints}, we cannot have both $\mu$ and $\rho_+$ in $\NP(\delta)$. This is a contradiction.
\end{proof}

\begin{lemma} \label{lem:impossibleorder}
    Let $\delta,\eta \in \Accord'$. Consider two colorations $\col_\delta$ and $\col_\eta$ of  $\NP(\delta)_0$ and $\NP(\eta)_0$ respectively. Assume that $\col_\delta$ and $\col_\eta$ weakly match, and that there exists $\rho \in \NP(\delta) \setminus \NP(\eta)$ that makes the colorations matchable. Let $v$ be the endpoint of $\rho$ in $\NP(\eta)_0$. Then there are no pairs $(\rho_1, \rho_2)$ of projective accordions such that:
    \begin{enumerate}[label=$\bullet$, itemsep=1mm]
        \item $\rho_2 \in \NP(\delta) \setminus \NP(\eta)$ and $\rho_1 \in \NP(\delta) \cap \NP(\eta)$;
        \item  $\rho_1$ and $\rho_2$ share $v$ as a common endpoint, and $\rho_1 \precdot_v \rho_2$; and,
        \item  $\rho_1$ makes the colorations matchable.
    \end{enumerate}
\end{lemma}

\begin{proof}
    Assume by contradiction that this configuration is possible. By \cref{lem:angleconfigNPscolor}, we have $\col_\delta(v) = {\osquare}$ and $\col_\eta(v) \in \{{\gsquare}, {\rsquare}\}$. We know that $\rho_2 \in \NP(\delta) \cap \NP(\eta)$. This can only be the case if $v$ is an endpoint of $\delta$ and, denoting by $u$ the other endpoint of $\rho_2$, we must have $\col_\delta(u) \in \{{\rsquare},{\gsquare}\}$. This is a contradiction as it means that $\col_\delta$ and $\col_\eta$ should be strongly matchable.
\end{proof}

\begin{prop} \label{prop:specialcaseofmatching} 
    Let $(\delta,\eta) \in \Accord'^2$. Consider two colorations $\col_\delta$ and $\col_\eta$ of  $\NP(\delta)_0$ and $\NP(\eta)_0$ respectively. Assume that $\col_\delta$ and $\col_\eta$ are weakly matching, and there exists $\rho \in \NP(\delta) \setminus \NP(\eta)$ that makes the colorations matchable. Then, there exist $\rho_- \in \NP(\delta) \setminus \NP(\eta)$ and $\rho_+ \in \NP(\delta) \cap \NP(\eta)$ such that all of the following assertions hold:
    \begin{enumerate}[label=$(\roman*)$, itemsep=1mm]
        \item $\rho_-$ makes the colorations matchable;
        \item $\rho_-$ and $\rho_+$ share a common endpoint $v \in \mathcal{M}_{\rpoint}$, and $\rho_- \prec_v \rho_+$;
        \item  $\rho_-$ and $\rho_+$ belong to the boundary of the same cell in $\Gamma(\Prj(\Delta^{\gpoint}))$;
        \item by denoting $u$ the other endpoint of $\rho_-$,  either $u$ or $v$ is an endpoint of $\delta$, or $\delta$ crosses $\rho_-$; and,
        \item by denoting $w$ the other endpoint of $\rho_+$, either $w$ or $v$ is an endpoint of $\eta$, or $\eta$ crosses $\rho_+$.
    \end{enumerate}
\end{prop}

\begin{figure}[!ht]
\centering 
    \begin{tikzpicture}[mydot/.style={
					circle,
					thick,
					fill=white,
					draw,
					outer sep=0.5pt,
					inner sep=1pt
				}, scale = .8]
		\tikzset{
		osq/.style={
        rectangle,
        thick,
        fill=white,
        append after command={
            node [
                fit=(\tikzlastnode),
                orange,
                line width=0.3mm,
                inner sep=-\pgflinewidth,
                cross out,
                draw
            ] {}}}}
			\foreach \X in {0,1,...,37}
		{
		\tkzDefPoint(4*cos(pi/19*\X),1.5*sin(pi/19*\X)){\X};
		};
        \draw [line width=0.7mm,domain=30:90] plot ({4*cos(\x)}, {1.5*sin(\x)});
        \draw [line width=0.7mm,domain=110:120] plot ({4*cos(\x)}, {1.5*sin(\x)});
        \draw [line width=0.7mm,domain=145:155] plot ({4*cos(\x)}, {1.5*sin(\x)});
        \draw [line width=0.7mm,domain=213:223] plot ({4*cos(\x)}, {1.5*sin(\x)});
        \draw [line width=0.7mm,domain=250:260] plot ({4*cos(\x)}, {1.5*sin(\x)});
        \draw [line width=0.7mm,domain=280:290] plot ({4*cos(\x)}, {1.5*sin(\x)});
        \draw [line width=0.7mm,domain=308:317] plot ({4*cos(\x)}, {1.5*sin(\x)});

		\draw[line width=0.9mm ,bend left=40,red](9) edge (12);
		\draw[line width=0.9mm ,bend left =20,red](12) edge (23);
            \draw[line width=0.9mm ,bend left =20,red](23) edge (30);
            \draw[line width=0.9mm ,bend left =30,red](30) edge (33);
            \draw[line width=0.9mm ,bend left =20,red](33) edge (4);
		
		\draw[line width=0.7mm,blue,bend right=20, loosely dashed](16) edge (4);
            \draw[line width=0.7mm,blue,bend left=0, loosely dashed](27) edge (4);

            \filldraw [pattern=dots, pattern color=red,opacity=0.3] (9) to [bend left=40] (12) to [bend left=20] (23) to [bend left=20] (30) to [bend left=30] (33)  to [bend left=20] (4) to [bend right=10] cycle ;

		\foreach \X in {4,9,12,16,23,27,30,33}
		{
		\tkzDrawPoints[fill =red,size=4,color=red](\X);
		};

		\tkzDefPoint(-1.2,0.8){g};
		\tkzLabelPoint[blue](g){\Large $\delta$}
            \tkzDefPoint(.6,0){h};
		\tkzLabelPoint[blue](h){\Large $\eta$}
            \tkzDefPoint(-3,0){i};
		\tkzLabelPoint[red](i){\Large $\rho_-$}
            \tkzDefPoint(-2.4,-0.8){j};
		\tkzLabelPoint[red](j){\Large $\rho_+$}
            \tkzDefPoint(-1.6,2){j};
		\tkzLabelPoint[red](j){\Large $u$}
            \tkzDefPoint(-3.7,2){k};
		\tkzLabelPoint[black](k){\Large $(a)$}
            \tkzDefPoint(-3.3,-.9){l};
		\tkzLabelPoint[red](l){\Large $v$}
            \tkzDefPoint(1,-1.5){l};
		\tkzLabelPoint[red](l){\Large $w$}

    \begin{scope}[xshift = 8cm]
        	\foreach \X in {0,1,...,37}
		{
		\tkzDefPoint(4*cos(pi/19*\X),1.5*sin(pi/19*\X)){\X};
		};
        \draw [line width=0.7mm,domain=30:90] plot ({4*cos(\x)}, {1.5*sin(\x)});
        \draw [line width=0.7mm,domain=110:120] plot ({4*cos(\x)}, {1.5*sin(\x)});
        \draw [line width=0.7mm,domain=145:155] plot ({4*cos(\x)}, {1.5*sin(\x)});
        \draw [line width=0.7mm,domain=213:223] plot ({4*cos(\x)}, {1.5*sin(\x)});
        \draw [line width=0.7mm,domain=250:260] plot ({4*cos(\x)}, {1.5*sin(\x)});
        \draw [line width=0.7mm,domain=280:290] plot ({4*cos(\x)}, {1.5*sin(\x)});
        \draw [line width=0.7mm,domain=297:307] plot ({4*cos(\x)}, {1.5*sin(\x)});
        \draw [line width=0.7mm,domain=315:328] plot ({4*cos(\x)}, {1.5*sin(\x)});

		\draw[line width=0.9mm ,bend left=40,red](9) edge (12);
		\draw[line width=0.9mm ,bend left =20,red](12) edge (23);
            \draw[line width=0.9mm ,bend left =20,red](23) edge (30);
            \draw[line width=0.9mm ,bend left =30,red](30) edge (34);
            \draw[line width=0.9mm ,bend left =20,red](34) edge (4);
		
		\draw[line width=0.7mm,blue,bend right=20, loosely dashed](16) edge (4);
            \draw[line width=0.7mm,blue,bend left=40, loosely dashed](27) edge (32);

            \filldraw [pattern=dots, pattern color=red,opacity=0.3] (9) to [bend left=40] (12) to [bend left=20] (23) to [bend left=20] (30) to [bend left=30] (34)  to [bend left=20] (4) to [bend right=10] cycle ;

		\foreach \X in {4,9,12,16,23,27,30,32,34}
		{
		\tkzDrawPoints[fill =red,size=4,color=red](\X);
		};

		\tkzDefPoint(-1.2,0.9){g};
		\tkzLabelPoint[blue](g){\Large $\delta$}
            \tkzDefPoint(.8,-.2){h};
		\tkzLabelPoint[blue](h){\Large $\eta$}
            \tkzDefPoint(-3,0){i};
		\tkzLabelPoint[red](i){\Large $\rho_-$}
            \tkzDefPoint(-2.4,-0.8){j};
		\tkzLabelPoint[red](j){\Large $\rho_+$}
            \tkzDefPoint(-1.6,2){j};
		\tkzLabelPoint[red](j){\Large $u$}
            \tkzDefPoint(-3.7,2){k};
		\tkzLabelPoint[black](k){\Large $(b)$}
            \tkzDefPoint(-3.3,-.9){l};
		\tkzLabelPoint[red](l){\Large $v$}
            \tkzDefPoint(1,-1.5){l};
		\tkzLabelPoint[red](l){\Large $w$}
    \end{scope}
    \begin{scope}[xshift = 4cm, yshift=-4.5cm]
        	\foreach \X in {0,1,...,37}
		{
		\tkzDefPoint(4*cos(pi/19*\X),1.5*sin(pi/19*\X)){\X};
		};
        \draw [line width=0.7mm,domain=30:45] plot ({4*cos(\x)}, {1.5*sin(\x)});
        \draw [line width=0.7mm,domain=80:120] plot ({4*cos(\x)}, {1.5*sin(\x)});
        \draw [line width=0.7mm,domain=145:155] plot ({4*cos(\x)}, {1.5*sin(\x)});
        \draw [line width=0.7mm,domain=195:205] plot ({4*cos(\x)}, {1.5*sin(\x)});
        \draw [line width=0.7mm,domain=220:230] plot ({4*cos(\x)}, {1.5*sin(\x)});
        \draw [line width=0.7mm,domain=250:260] plot ({4*cos(\x)}, {1.5*sin(\x)});
        \draw [line width=0.7mm,domain=280:290] plot ({4*cos(\x)}, {1.5*sin(\x)});
        \draw [line width=0.7mm,domain=305:320] plot ({4*cos(\x)}, {1.5*sin(\x)});
        \draw [line width=0.7mm,domain=335:350] plot ({4*cos(\x)}, {1.5*sin(\x)});

		\draw[line width=0.9mm ,bend left=30,red](12) edge (21);
		\draw[line width=0.9mm ,bend left =30,red](21) edge (24);
            \draw[line width=0.9mm ,bend left =20,red](24) edge (30);
            \draw[line width=0.9mm ,bend left =30,red](30) edge (36);
            \draw[line width=0.9mm ,bend left =20,red](30) edge (9);
            \draw[line width=0.9mm ,bend right =30,red](9) edge (4);
		
		\draw[line width=0.7mm,blue,bend right=20, loosely dashed](16) edge (9);
            \draw[line width=0.7mm,blue,bend left=40, loosely dashed](27) edge (33);

            \filldraw [pattern=dots, pattern color=red,opacity=0.3] (9) to [bend right=30] (4) to [bend left=50] (36) to [bend right=30] (30) to [bend left=20] cycle ;

		\foreach \X in {4,9,12,16,21,24,27,30,33,36}
		{
		\tkzDrawPoints[fill =red,size=4,color=red](\X);
		};

		\tkzDefPoint(-1.2,0.8){g};
		\tkzLabelPoint[blue](g){\Large $\delta$}
            \tkzDefPoint(0,-.2){h};
		\tkzLabelPoint[blue](h){\Large $\eta$}
            \tkzDefPoint(-.1,.7){i};
		\tkzLabelPoint[red](i){\Large $\rho_+$}
            \tkzDefPoint(1.3,1.7){j};
		\tkzLabelPoint[red](j){\Large $\rho_-$}
            \tkzDefPoint(3.5,1.5){j};
		\tkzLabelPoint[red](j){\Large $u$}
            \tkzDefPoint(-3.7,2){k};
		\tkzLabelPoint[black](k){\Large $(c)$}
            \tkzDefPoint(0.3,2.1){l};
		\tkzLabelPoint[red](l){\Large $v$}
            \tkzDefPoint(1,-1.5){l};
		\tkzLabelPoint[red](l){\Large $w$}
    \end{scope}
    \end{tikzpicture}
\caption{\label{fig:ReducedWeakConfig} Some of the weakly matchable configurations. All the configurations are obtained from them by identifying endpoints of $\delta$ and $\eta$ as stated in \cref{prop:specialcaseofmatching}. The dotted area is the cell whose boundary contains $\rho_+$ and $\rho_-$.}
\end{figure}
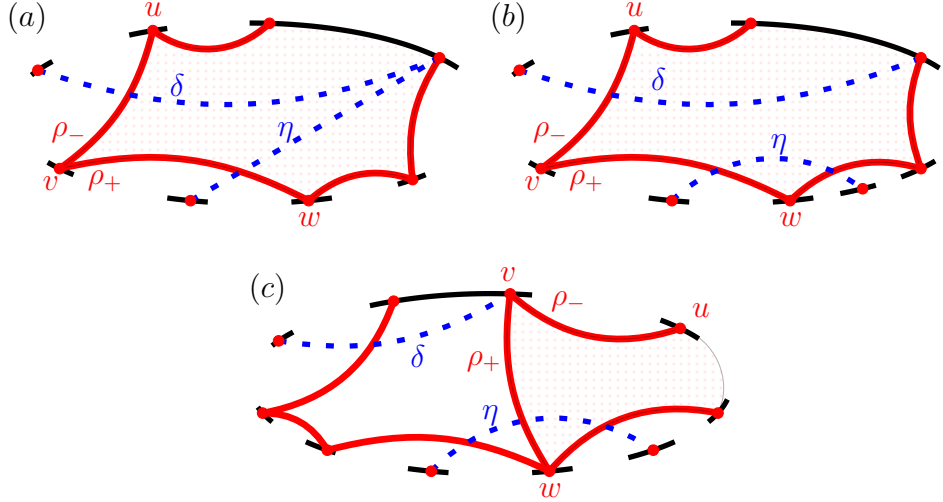

\begin{proof}
    Consider a projective accordion $\rho \in \NP(\delta) \setminus \NP(\eta)$ making the colorations matchable. By \cref{lem:angleconfigNPscolor}, exactly one of the endpoint $v$ is in $\NP(\eta)_0$, and we can assume that $\col_\eta(v) = {\rsquare}$ and $\col_{\delta}(v) = \osquare$. By \cref{lem:angleconfigNPsplace}, and by \cref{lem:impossibleorder}, we must have that all the projective accordions $\rho' \in \NP(\eta)$ having $v$ as an endpoint are greater than $\rho$ for $\prec_v$. Denote by:
    \begin{enumerate}[label=$\bullet$, itemsep=1mm]
        \item $\rho_-$ the smallest projective accordion for $\prec_v$ among those in $\NP(\delta)$ having $v$ as an endpoint that makes the colorations matchable 
        \item $\rho_+$ the greatest projective accordion for $\prec_v$ among those in $\NP(\eta)$ having $v$ as an endpoint.
    \end{enumerate}
    By \cref{lem:inbetweenNP,lem:angleconfigNPsnonmatchable}, we have that $\rho_+ \in \NP(\delta) \cap \NP(\eta)$ and $\rho_- \precdot_v \rho_+$ in $\Prj(\Delta^{\gpoint})$. By definition of $\col_\delta$, we have that $\rho_+$ and $\rho_-$ are in the boundary of some cell in $\pmb{\Gamma}(\Prj(\Delta^{\gpoint}))$.

    Denote by $u$ the other endpoint of $\rho_-$ and $w$ the other endpoint of $\rho_+$. If both $v \neq t(\delta) \neq u$, the accordion $\delta$ must cross $\rho_-$ to satisfy $\col_\delta(u) = {\gsquare}$. If both $v \neq s(\eta) \neq w$, the accordion $\eta$ must cross $\rho_+$ to satisfy $\col_\eta(v) = {\rsquare}$. Hence the desired result.
\end{proof}

Thanks to this result, we have the following consequences.

\begin{cor} \label{cor:numberofpairsrho+rho-}
Let $(\delta,\eta) \in \Accord'^2$ be such that $\col_\delta$ and $\col_\eta$ are weakly matchable. Assume that there exists $\rho \in \NP(\delta) \setminus \NP(\eta)$ that makes the colorations matchable. The following assertions hold:
    \begin{enumerate}[label=$(\roman*)$, itemsep=1mm]
        \item there are no $\mu \in \NP(\eta) \setminus \NP(\delta)$ making the colorations matchable; and,
        \item there exist at most two pairs $(\rho_+, \rho_-)$ of projective accordions satisfying the configurations in \cref{prop:specialcaseofmatching}; and,
        \item in the case where we have exactly two pairs $(\rho_+^{(1)}, \rho_-^{(2)})$ and $(\rho_+^{(2)}, \rho_-^{(2)})$ satisfying \cref{prop:specialcaseofmatching}, then, up to exchanging the pairs, $\rho_+^{(1)}$ and $\rho_+^{(2)}$ share a common endpoint $w \in \mathcal{M}_{\rpoint}$, $\rho_+^{(1)} \precdot_w \rho_+^{(2)}$ in $\Prj(\Delta^{\gpoint})$, and $\col_\delta(w) = {\psquare}$.
    \end{enumerate}
\end{cor}

\begin{figure}[!ht]
\centering 
    \begin{tikzpicture}[scale=1.4,mydot/.style={
					circle,
					thick,
					fill=white,
					draw,
					outer sep=0.5pt,
					inner sep=1pt
				}, scale = .8]
		\tikzset{
		osq/.style={
        rectangle,
        thick,
        fill=white,
        append after command={
            node [
                fit=(\tikzlastnode),
                orange,
                line width=0.3mm,
                inner sep=-\pgflinewidth,
                cross out,
                draw
            ] {}}}}
			
        	\foreach \X in {0,1,...,37}
		{
		\tkzDefPoint(4*cos(pi/19*\X),1.5*sin(pi/19*\X)){\X};
		};
        \draw [line width=0.7mm,domain=30:45] plot ({4*cos(\x)}, {1.5*sin(\x)});
        \draw [line width=0.7mm,domain=80:120] plot ({4*cos(\x)}, {1.5*sin(\x)});
        \draw [line width=0.7mm,domain=145:155] plot ({4*cos(\x)}, {1.5*sin(\x)});
        \draw [line width=0.7mm,domain=220:230] plot ({4*cos(\x)}, {1.5*sin(\x)});
        \draw [line width=0.7mm,domain=250:260] plot ({4*cos(\x)}, {1.5*sin(\x)});
        \draw [line width=0.7mm,domain=280:290] plot ({4*cos(\x)}, {1.5*sin(\x)});
        \draw [line width=0.7mm,domain=305:320] plot ({4*cos(\x)}, {1.5*sin(\x)});
        \draw [line width=0.7mm,domain=335:350] plot ({4*cos(\x)}, {1.5*sin(\x)});

		\draw[line width=0.9mm ,bend left=30,red](12) edge (24);
            \draw[line width=0.9mm ,bend left =20,red](24) edge (30);
            \draw[line width=0.9mm ,bend left =30,red](30) edge (36);
            \draw[line width=0.9mm ,bend left =20,red](30) edge (9);
            \draw[line width=0.9mm ,bend right =30,red](9) edge (4);
		
		\draw[line width=0.7mm,blue,bend right=20, loosely dashed](16) edge (9);
            \draw[line width=0.7mm,blue,bend left=40, loosely dashed](27) edge (33);

		\foreach \X in {4,9,12,16,24,27,30,33,36}
		{
		\tkzDrawPoints[fill =red,size=4,color=red](\X);
		};

		\tkzDefPoint(-1.2,0.8){g};
		\tkzLabelPoint[blue](g){\Large $\delta$}
            \tkzDefPoint(0,-.2){h};
		\tkzLabelPoint[blue](h){\Large $\eta$}
            \tkzDefPoint(-2.2,.3){i};
		\tkzLabelPoint[red](i){\Large $\rho_-^{(1)}$}
            \tkzDefPoint(.7,.7){i};
		\tkzLabelPoint[red](i){\Large $\rho_+^{(2)}$}
            \tkzDefPoint(1.4,1.7){j};
		\tkzLabelPoint[red](j){\Large $\rho_-^{(2)}$}
            \tkzDefPoint(1.1,-1.5){l};
		\tkzLabelPoint[red](l){\Large $w$}
            \tkzDefPoint(-1.2,-.2){l};
		\tkzLabelPoint[red](l){\Large $\rho_+^{(1)}$}
    \end{tikzpicture}
\caption{\label{fig:ReducedWeakConfigtworho+} Illustration of the case where we have two pairs $(\rho_+^{(1)},\rho_-^{(1)})$ and $(\rho_+^{(2)}, \rho_-^{(2)})$ that satisfy the conditions given by \cref{prop:specialcaseofmatching}.}
\end{figure}
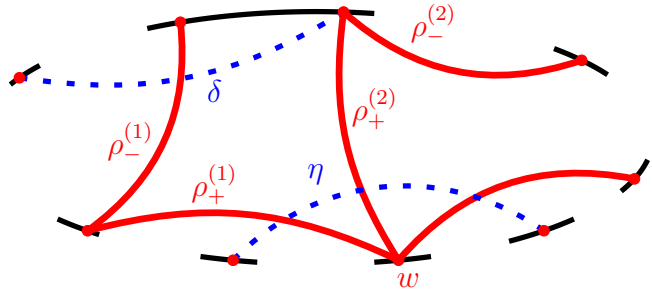

\begin{proof}
    This is a fusion of the cases $(b)$ and $(c)$ depicted in \cref{fig:specialmatchreduction}.
\end{proof}

Moreover, we can precisely identify the projective accordions that make the colorations matchable.

\begin{cor} \label{cor:ProjAccordweakmatch}
   Let $(\delta, \eta) \in \Accord'^2$. Consider two colorations $\col_\delta$ and $\col_\eta$  of $\NP(\delta)_0$ and $\NP(\eta)_0$  respectively. Assume that $\col_\delta$ and $\col_\eta$ are weakly matching, and there exists $\rho \in \NP(\delta) \setminus \NP(\eta)$ that makes the colorations matchable. For any $\rho \in \NP(\delta) \setminus \NP(\eta)$ which makes $\col_\delta$ and $\col_\eta$ matchable, there exist $\rho_- \in  \NP(\delta) \setminus \NP(\eta)$ and $\rho_+ \in \NP(\delta) \cap \NP(\eta)$ as given by \cref{prop:specialcaseofmatching} such that:
   \begin{enumerate}[label=$\bullet$,itemsep=1mm]
       \item the common endpoint $v \in \mathcal{M}_{\rpoint}$ of $\rho_-$ and $\rho_+$ is one of the endpoints of $\rho$, and if $\rho_- \prec_v \rho_+$, then $\rho \preccurlyeq_v \rho_-$, otherwise $\rho_- \preccurlyeq_v \rho$; and,
       \item either $\delta$ crosses $\rho$ or they share a common endpoint.
   \end{enumerate}
\end{cor}

\subsection{Non-matchable colorations}
\label{ss:nomatchablemoves}
In this section, we focus on pairs $(\delta,\eta) \in \Accord'^2$ such that:
\begin{enumerate}[label=$\bullet$, itemsep=1mm]
    \item $\NP(\delta) \cap \NP(\eta) \neq \varnothing$; and,
    \item $\col_\delta$ and $\col_\eta$ are not matchable.
\end{enumerate} In general, we have $\opResAc(\delta,\eta) \neq \opResAc(\delta) \cup \opResAc(\eta)$. However, by \cref{prop:matchableextorepi}, there is neither epimorphism nor extensions ($\Ext^1$) between $\MM(\delta)$ and $\MM(\eta)$. The main point of this section is to show that to calculate $\opResAc(\delta,\eta)$, we can bring ourselves back to calculate $\opResAc(\xi,\varsigma)$, with $(\xi, \varsigma) \in \Accord'^2$ such that:
\begin{enumerate}[label=$\bullet$, itemsep=1mm]
    \item $\xi \in \opResAc'(\delta)$, and $\varsigma \in \opResAc'(\eta)$; and,
    \item the colorations $\col_\xi$ and $\col_\varsigma$ are matchable.
\end{enumerate}

\begin{prop} \label{prop:nomatchablesmaller}
Let $(\delta,\eta) \in \Accord'^2$ such that $\NP(\delta) \cap \NP(\eta) \neq \varnothing$. If $\col_\delta$ and $\col_\eta$ are not matchable, then we have:
\[\opResAc(\delta,\eta) = \opResAc(\delta) \cup \opResAc(\eta) \cup \left( \bigcup_{\mu \Accordlneq \delta} \opResAc(\mu,\eta) \right) \cup \left( \bigcup_{\nu \Accordlneq \eta} \opResAc(\delta,\nu) \right). \]
\end{prop}

\begin{proof}
    We only need to show that the set : \[\mathscr{Z}(\delta,\eta) = \opResAc(\delta) \cup \opResAc(\eta) \cup \left( \bigcup_{\mu \Accordlneq \delta} \opResAc(\mu,\eta) \right) \cup \left( \bigcup_{\nu \Accordlneq \eta} \opResAc(\delta,\nu) \right)\] is a resolving set of $\Accord$ containing $\delta$ and $\eta$. Set $\mathscr{Z}'(\delta,\eta) = \mathscr{Z}(\delta,\eta) \setminus \Prj(\Delta^{\gpoint})$ So:
    \begin{enumerate}[label=$\bullet$, itemsep=1mm]
        \item $\delta, \eta \in \mathscr{Z}(\delta,\eta)$;
        \item $\mathscr{Z}'(\delta,\eta)$ is an ideal of $(\Accord',\Accordleq)$, and so, \cref{lem:IdealsclosedSyzygies} guarantees us that $\mathscr{Z}(\delta,\eta)$ is closed under syzygies;
        \item $\mathscr{Z}(\delta,\eta)$ is closed under overlap and arrow extensions between accordions $\mu$ and $\nu$ such that $\mu \Accordleq \delta$, $\nu \Accordleq \eta$ and $(\mu,\nu) \neq (\delta,\eta)$; and
        \item As $\col_\delta$ and $\col_\eta$ are not matchable, by the contraposition of \cref{prop:matchableextorepi}, there are neither overlap nor arrow extensions. 
    \end{enumerate}
    So we get the desired result.
\end{proof}
Let $\rho \in \NP(\delta)$. Consider the laterality induced by $\col_\delta$. Assume that \[\col_\delta(\{s(\rho), t(\rho)\}) \subseteq \{{\osquare},{\psquare}\}.\] Set $\sc(\delta) = (\eta_L, C_L)$ and $\tc(\delta) = (\eta_R, C_R)$. So $\rho \subset \partial C_R$, or $\rho \subset \partial C_L$, and one of the two cells contains at least three projective accordions in its boundary by \cref{def:colourendpoints}. 

In such a configuration, $\MM(\delta)$ admits indecomposable non-projective representations as summands of the $i$th syzygy for some $i \geqslant 1$. Those indecomposable non-projective summands correspond to accordions $\varsigma \Accordleq \delta$ such that:
\begin{enumerate}[label=$\bullet$,itemsep=1mm]
    \item $\varsigma \subset C_L$ and $s(\varsigma) = s(\delta)$; or,
    \item $\varsigma \subset C_R$ and $t(\varsigma) = t(\delta)$.
\end{enumerate}
We refer to \cite[Lemmas 6.12 and 6.13]{DS251} for more details.

\begin{definition} \label{def:sygaccord}
Let $\delta \in \Accord'$ and $\rho \in \NP(\delta)$ such that \[\col_\delta(\{s(\rho),t(\rho)\}) \subseteq \{{\osquare},{\psquare}\}.\]  We define $\Syg_{(\delta,\rho)}$ the \new{syzygy accordion of} $\delta$ \new{relative to} $\rho$ to be the accordion such that:
\begin{enumerate}[label=$\bullet$,itemsep=1mm]
    \item  $s(\Syg_{(\delta,\rho)}) = s(\delta)$, and:
    \begin{enumerate}[label=$\bullet$,itemsep=1mm]
    \item $t(\Syg_{(\delta,\rho)}) = t(\rho)$ if $\col_\delta (t(\rho)) \neq {\psquare}$; and,
    \item $t(\Syg_{(\delta,\rho)}) = w_L(\delta)$ otherwise,
    \end{enumerate}
   if $\rho \subset C_L$; or,
    \item $t(\Syg_{(\delta,\rho)}) = t(\delta)$ and:
    \begin{enumerate}[label=$\bullet$,itemsep=1mm]
    \item $s(\Syg_{(\delta,\rho)}) = s(\rho)$ if $\col_\delta (s(\rho)) \neq {\psquare}$; and,
    \item $s(\Syg_{(\delta,\rho)}) = w_R(\delta)$ otherwise,
    \end{enumerate} 
    if $\rho \subset C_R$.
\end{enumerate}
\end{definition}
In the case where $\delta \subset C$ for some $C \in \pmb{\Gamma}(\Delta^{\gpoint})$, the accordion $\Syg_{(\delta,\rho)}$ is still well-defined as precisely one of the constructions gives an accordion.

\begin{ex} \label{ex:sygaccordion}
In \cref{fig:nonorientable}, we show the construction of $\Syg_{(\delta,\rho)}$ and $\Syg_{(\eta, \rho)}$ in the case where:
\begin{enumerate}[label=$\bullet$, itemsep=1mm]
    \item the colorations $\col_\delta$ and $\col_\eta$ are not matchable;
    \item $\delta$ and $\eta$ are contained in a cell of $\pmb{\Gamma}(\Prj(\Delta^{\gpoint}))$; and,
    \item the image of the endpoints of $\rho \in \NP(\delta) \cap \NP(\eta)$ by $\col_\delta$ and by $\col_\eta$ are in $\{{\osquare},{\psquare}\}$. \qedhere
\end{enumerate}
\end{ex}

\begin{figure}[!ht]
        \centering
       \begin{tikzpicture}[mydot/.style={
					circle,
					thick,
					fill=white,
					draw,
					outer sep=0.5pt,
					inner sep=1pt
				}, scale = 1]
		\tikzset{
		osq/.style={
        rectangle,
        thick,
        fill=white,
        append after command={
            node [
                fit=(\tikzlastnode),
                orange,
                line width=0.3mm,
                inner sep=-\pgflinewidth,
                cross out,
                draw
            ] {}}}}
        \draw [line width=0.7mm,domain=0:15] plot ({4*cos(\x)}, {1.5*sin(\x)});
		\draw [line width=0.7mm,domain=40:55] plot ({4*cos(\x)}, {1.5*sin(\x)});
		\draw [line width=0.7mm,domain=65:75] plot ({4*cos(\x)}, {1.5*sin(\x)});
		\draw [line width=0.7mm,domain=98:105] plot ({4*cos(\x)}, {1.5*sin(\x)});
		\draw [line width=0.7mm,domain=127:140] plot ({4*cos(\x)}, {1.5*sin(\x)});
		\draw [line width=0.7mm,domain=155:170] plot ({4*cos(\x)}, {1.5*sin(\x)});
		\draw [line width=0.7mm,domain=205:217] plot ({4*cos(\x)}, {1.5*sin(\x)});
        \draw [line width=0.7mm,domain=230:240] plot ({4*cos(\x)}, {1.5*sin(\x)});
        \draw [line width=0.7mm,domain=255:262] plot ({4*cos(\x)}, {1.5*sin(\x)});
        \draw [line width=0.7mm,domain=278:285] plot ({4*cos(\x)}, {1.5*sin(\x)});
        \draw [line width=0.7mm,domain=302:309] plot ({4*cos(\x)}, {1.5*sin(\x)});
        \draw [line width=0.7mm,domain=330:345] plot ({4*cos(\x)}, {1.5*sin(\x)});
		\foreach \X in {0,...,45}
		{
		\tkzDefPoint(4*cos(pi/23*\X),1.5*sin(pi/23*\X)){\X};
		};
		
	    \draw[line width=0.9mm ,bend right =30,red] (1) edge (43);
	    \draw[line width=0.9mm ,bend right =30,red] (43) edge (39);
	    \draw[line width=0.9mm ,bend left =10,mypurple,densely dashdotted] (39) edge (13);
	    \draw[line width=0.9mm ,bend right =20,red] (13) edge (9);
	    \draw[line width=0.9mm ,bend right =20,red] (9) edge (6);

	    \draw[line width=0.9mm ,bend right=30,red] (21) edge (17);
	    \draw[line width=0.9mm ,bend right =30,red] (17) edge (13);
	    \draw[line width=0.9mm ,bend left =30,red] (36) edge (39);
	    \draw[line width=0.9mm ,bend left =40,red] (33) edge (36);
	    \draw[line width=0.9mm ,bend left =30,red] (30) edge (33);
	    \draw[line width=0.9mm ,bend left =20,red] (27) edge (30);

		\draw [line width=0.9mm,blue, dashdotted, bend right = 30, opacity=0.8]  (6) edge (1);
		\draw [line width=0.9mm,orange, loosely dotted, bend left = 30, opacity=0.8]  (21) edge (27);
        \draw [line width=0.9mm, orange,dash pattern={on 10pt off 2pt on 5pt off 2pt}, bend left=10,opacity=0.3] (21) edge (39);
         \draw [line width=0.9mm, blue,dash pattern={on 5pt off 2pt on 1pt off 2pt}, bend right=15,opacity=0.3] (13) edge (1);
         
		\foreach \X in {1,6,9,13,17,21,27,30,33,36,39,43}
		{
		\tkzDrawPoints[fill =red,size=4,color=red](\X);
		};

		\tkzDefPoint(2.5,1){gammaM};
		\tkzLabelPoint[blue](gammaM){\Large $\delta$}
		\tkzDefPoint(1.9,0.4){gammaM};
		\tkzLabelPoint[blue,opacity=0.3](gammaM){\large $\Syg_{(\delta,\rho)}$}
		\tkzDefPoint(-.8,1.2){gammaP};
		\tkzLabelPoint[mypurple](gammaP){\Large $\rho$}
		\tkzDefPoint(-3.1,0){rho};
		\tkzLabelPoint[orange](rho){\Large $\eta$}
		\tkzDefPoint(-.8,-.1){gammaM};
		\tkzLabelPoint[orange,opacity=0.3](gammaM){\large $\Syg_{(\eta,\rho)}$}
    \end{tikzpicture}
        \caption{\label{fig:nonorientable} Illustration of two accordions $\delta, \eta \in \Accord'$ whose colorations are not matchable, and $\{\rho\} = \NP(\delta) \cap \NP(\eta) \neq \varnothing$. We constructed $\Syg_{(\delta,\rho)}$ and $\Syg_{(\eta,\rho)}$.}
    \end{figure}
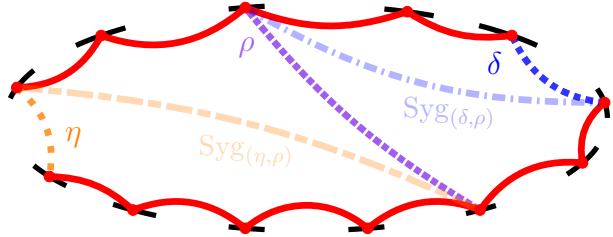

\begin{lemma} \label{lem:tomatchablecap}
Let $(\delta, \eta) \in \Accord'^2$. Assume that  $\NP(\delta)\cap\NP(\eta)\neq\emptyset$, and $\col_\delta$ and $\col_\eta$ are not matchable. For any $\rho \in \NP(\delta)\cap\NP(\eta)$, by denoting $v_1$ and $v_2$ the endpoints of $\rho$:
\begin{enumerate}[label=$(\roman*)$,itemsep=1mm]
    \item if $\col_\delta(\{v_1,v_2\}) \nsubseteq \{{\osquare}, {\psquare}\}$, then $\xi = \delta$ and $\varsigma = \Syg_{(\eta,\rho)}$ admit a pair $(\col_\delta, \col_\varsigma)$ of strongly matchable colorations;
    \item if $\col_\eta(\{v_1,v_2\}) \nsubseteq \{{\osquare}, {\psquare}\}$, then $\xi = \Syg_{(\delta,\rho)}$ and $\varsigma = \eta$ admit a pair $(\col_\xi, \col_\eta)$ of strongly matchable colorations;
    \item otherwise, $\xi = \Syg_{(\delta,\rho)}$ and $\varsigma = \Syg_{(\eta,\rho)}$ admit a pair $(\col_\xi,\col_\varsigma)$ of strongly matchable colorations.
\end{enumerate}  
\end{lemma}

\begin{proof}
Under our assumptions, if $\col_\delta(\{v_1,v_2\}) \nsubseteq \{{\osquare}, {\psquare}\}$, then it is clear that the coloration of $\NP(\varsigma)_0$, up to conjugation, sends one of the endpoints of $\rho$ to ${\rsquare}$ or ${\gsquare}$. This proves $(i)$.

The assertion $(ii)$ is the dual to $(i)$ up to exchanging the role of $\delta$ and $\eta$.

In the last case, we can consider $\xi$ and $\varsigma$ as stated, and the colorations of $\NP(\xi)_0$ and $\NP(\varsigma)_0$ are such that at least one of the endpoints of $\rho$ is ${\rsquare}$ or ${\gsquare}$ according to both colorations. It follows that the colorations are matchable. 
\end{proof}

In the following,  we consider pairs $(\xi, \varsigma) \in \Accord'^2$ such that:
    \begin{enumerate}[label=$\bullet$, itemsep=1mm]
        \item $\xi \Accordleq \delta$ and $\varsigma \Accordleq \eta$; and,
        \item  $(\col_\xi,\col_\varsigma)$ are matchable and there exists $\rho \in \NP(\xi) \setminus \NP(\varsigma)$ makes them matchable.
    \end{enumerate}
Let us first prove the following lemma.

\begin{lemma} \label{lem:positiontrymatchcupnotcap}
    Let $\eta \in \Accord'^2$. Consider a coloration $\col_\eta$ of  $\NP(\eta)_0$. Fix a projective accordion $\rho \in \Prj(\Delta^{\gpoint}) \setminus \NP(\eta)$, and denote by $v_1$ and $v_2$ the endpoints of $\rho$. Assume that there exists $w \in \{v_1,v_2\}$ such that $w \in \NP(\eta)_0$. If $w$ is an endpoint of $\eta$, then $\col_\eta(w) \in \{{\rsquare},{\gsquare}\}$.
\end{lemma}

\begin{proof} 
Let $\eta \in \Accord'^2$. Consider a coloration $\col_\eta$ of  $\NP(\eta)_0$. Fix a projective accordion $\rho \in \Prj(\Delta^{\gpoint}) \setminus \NP(\eta)$. Suppose that one of the endpoints of $\rho$ is an endpoint of $\eta$. Denote this endpoint by $w$ and note that if $\col_\eta(w)={\osquare}$ by construction of $\NP(\eta)$, all the projective accordions with endpoint $w$ are in $\NP(\eta)$.
\end{proof}

Now, we can prove the following result.

\begin{lemma} \label{lem:tomatchablecupnotcap}
Let $(\delta, \eta) \in \Accord'^2$. Assume that  $\NP(\delta)\cap\NP(\eta)\neq \varnothing$, and $\col_\delta$ and $\col_\eta$ are not matchable. For any $\rho \in \NP(\delta) \setminus \NP(\eta)$, with endpoints $v_1$ and $v_2$ :
\begin{enumerate}[label=$(\roman*)$,itemsep=1mm]
    \item if there exists $w \in \{v_1,v_2\}$ such that $w \in \NP(\eta)_0$ and $\col_\eta(w) \in \{{\rsquare}, {\gsquare}\}$, then $\xi = \Syg_{(\delta,\rho)}$ and $\varsigma = \eta$ admit a pair $(\col_\xi, \col_\eta)$ of matchable colorations;
    \item if there exists $w \in \{v_1,v_2\}$ such that $w \in \NP(\eta)_0$ and $\col_\eta(w) \in \{{\osquare},{\psquare}\}$, then, for any $(\xi, \varsigma) \in \Accord'^2$ such that:
    \begin{enumerate}[label=$\bullet$, itemsep=1mm]
        \item $\xi \Accordleq \delta$ and $\varsigma \Accordleq \eta$; and,
        \item  $\col_\xi$ and $\col_\varsigma$ are matchable and $\rho$ makes them matchable,
    \end{enumerate}
    then $\col_\xi$ and $\col_\varsigma$ are strongly matchable; 
    \item otherwise, there are no pair $(\xi, \varsigma) \in \Accord'^2$, where $\xi \Accordleq \delta$ and $\varsigma \Accordleq \eta$, such that $\rho$ makes $\col_\xi$ and $\col_\varsigma$ matchable.
\end{enumerate}  
\end{lemma}

\begin{proof}
    The proof of assertion $(i)$ is similar to that of  \cref{lem:tomatchablecap} $(ii)$.

    Assume that there exists $w \in \{v_1,v_2\}$ such that $w \in \NP(\eta)_0$ and $\col_\eta(w) = {\osquare}$.  Let us consider $(\xi, \varsigma) \in \Accord'^2$ such that:
    \begin{enumerate}[label=$\bullet$, itemsep=1mm]
        \item $\xi \Accordleq \delta$ and $\varsigma \Accordleq \eta$; and,
        \item $\col_\xi$ and $\col_\varsigma$ are matchable and $\rho$ makes them matchable.
    \end{enumerate}
    Assume by contradiction that $\col_\delta$ and $\col_\eta$ are weakly matchable. By \cref{prop:specialcaseofmatching}, there exist $\rho_- \in \NP(\xi) \setminus \NP(\varsigma)$ and $\rho_+ \in \NP(\xi) \cap \NP(\varsigma)$ such that:
    \begin{enumerate}[label=$\bullet$, itemsep=1mm]
        \item $\rho_-$ makes the colorations matchable;
        \item $\rho_-$ and $\rho_+$ share $w$ as a common endpoint;
        \item  $\rho_-$ and $\rho_+$ are together in the boundary of a cell in $\Gamma(\Prj(\Delta^{\gpoint}))$;
        \item by denoting $v$ the other endpoint of $\rho_-$,  either $v$ or $w$ is an endpoint of $\xi$, or $\xi$ crosses $\rho_-$; and,
        \item by denoting $x$ the other endpoint of $\rho_+$, either $x$ or $w$ is an endpoint of $\varsigma$, or $\varsigma$ crosses $\rho_+$.
    \end{enumerate} Without loss of generality, we assume that $\rho \preccurlyeq_w \rho_- \prec_w \rho_+$, as $\rho \in \NP(\xi) \setminus \NP(\varsigma)$ makes the colorations matchable. Thus $\col_\varsigma(w) \in \{{\gsquare},{\rsquare}\}$. This is possible only if $\varsigma \Accordleq \Syg_{(\eta, \rho_1)}$. In our configuration, we have that $\col_\xi(w) = {\osquare}$. As $\col_\eta(w) = {\osquare}$ and as $\rho \in \NP(\delta) \setminus \NP(\eta)$, by \cref{lem:positiontrymatchcupnotcap}, the marked point $w$ is not an endpoint of $\eta$. Moreover, we have that $\rho_+ \in \NP(\eta)$ as $\varsigma \Accordleq \eta$. So $x$ cannot be an endpoint of $\varsigma$. Therefore, the only possibility is that $\rho_- \in \NP(\varsigma)$ by being either:
    \begin{enumerate}[label=$\bullet$, itemsep=1mm]
        \item contained in the boundary of the source or the target cell of $\delta$; or,
        \item part of the projective accordions with extensions of higher degree with $\eta$.
    \end{enumerate}

    Assume that there exists $w \in \{v_1,v_2\}$ such that $w \in \NP(\eta)_0$ and $\col_\eta(w) = {\psquare}$. In this case, there are precisely two projective accordions $\rho_1, \rho_2 \in \NP(\delta)$ such that $w$ is one of their endpoints. Moreover $\rho_1 \precdot_v \rho_2$ in $\Prj(\Delta^{\gpoint})$. By \cref{lem:angleconfigNPsplace}, we can assume that $\rho \preccurlyeq_w \rho_1$. The case where $\rho_2 \preccurlyeq_w \rho$ can be treated similarly. By \cref{lem:inbetweenNP}, we must have $\rho_1 \in \NP(\delta) \cap \NP(\eta)$.

    Let us consider $(\xi, \varsigma) \in \Accord'^2$ such that:
    \begin{enumerate}[label=$\bullet$, itemsep=1mm]
        \item $\xi \Accordleq \delta$ and $\varsigma \Accordleq \eta$; and,
        \item  $\col_\xi$ and $\col_\varsigma$ are matchable and $\rho$ makes them matchable.
    \end{enumerate}
    We must have $\col_\varsigma(w) \in \{{\gsquare},{\rsquare}\}$. This is possible only if $\varsigma \Accordleq \Syg_{(\eta, \rho_1)}$. If $\col_\xi(w) \in \{{\gsquare},{\rsquare}\}$, then $\rho_1$ makes the colorations matchable. Otherwise we have that $\col_\xi(w) = {\osquare}$. Therfore, by \cref{lem:positiontrymatchcupnotcap}, $w$ is not an endpoint of $\xi$. As $\rho$ makes $\col_\xi$ and $\col_\varsigma$ matchable, and $\rho \prec_w \rho_1$, we have that $\rho_1$ makes the colorations matchable too. We proved assertion $(ii)$.

    We still have to treat the case where $\{v_1,v_2\} \cap \NP(\eta)_0 = \varnothing$. So for any $\varsigma \Accordleq \eta$, we also have $\NP(\varsigma) \subseteq \NP(\eta)$, and so $\{v_1,v_2\} \cap \NP(\varsigma)_0 = \varnothing$. Therefore, we have proved assertion $(iii)$.
\end{proof}

Therefore, we introduce the following notion.

\begin{definition} \label{def:Sapprox}
Let $(\delta,\eta) \in \Accord'^2$ and $\rho \in \Prj(\Delta^{\gpoint})$. Denote by $v_1, v_2 \in \mathcal{M}_{\rpoint}$ the endpoints of $\rho$. We say that $(\delta,\eta)$ \new{admits an $\Smov(\rho)$-approximation} if the following assertions hold:
\begin{enumerate}[label=$\bullet$, itemsep=1mm]
    \item $\{\delta, \eta\} \in \Anti(\Accord', \Accordleq)$; and
    \item $\rho \in \NP(\delta)\cup \NP(\eta)$ is such that one of the following assertions holds:
    \begin{enumerate}[label=$\bullet$,itemsep=1mm]
        \item $\rho \in \NP(\delta) \cap \NP(\eta)$; 
        \item there exists $w \in \{v_1, v_2\} \cap \NP(\delta)_0$ such that $\col_{\delta}(w) \in \{{\rsquare},{\gsquare}\}$; or,
        \item there exists $w \in \{v_1, v_2\} \cap \NP(\eta)_0$ such that $\col_{\eta}(w) \in \{{\rsquare},{\gsquare}\}$.
    \end{enumerate}
\end{enumerate}
In such a case, we define the \new{$\Smov(\rho)$-approximation} of $(\delta,\eta)$ to be the pair of nonprojective accordions $(\xi_\rho^{(\delta)},\varsigma_\rho^{(\eta)})$ defined as follows:
\begin{enumerate}[label=$(\roman*)$,itemsep=1mm]
    \item if there exists $w \in \{v_1, v_2\} \cap \NP(\delta)_0$ such that $\col_{\delta}(w) \in \{{\rsquare},{\gsquare}\}$, then set $\xi_\rho^{(\delta)} = \delta$; otherwise set $\xi_\rho^{(\delta)} = \Syg_{(\delta,\rho)}$; and, 
    
    \item if there exists $w \in \{v_1, v_2\} \cap \NP(\eta)_0$ such that $\col_{\eta}(w) \in \{{\rsquare},{\gsquare}\}$; $\zeta_\rho^{(\eta)}=\eta$ otherwise, set $\varsigma_\rho^{(\eta)} = \Syg_{(\eta,\rho)}$.
\end{enumerate}
\end{definition}

\begin{ex} \label{ex:Sapprox}
    In \cref{fig:nonorientable}, the pair $(\Syg_{(\delta,\rho)}, \Syg_{(\eta,\rho)})$ is the $\Smov(\rho)$-approximation of $(\delta,\eta)$ relative to $\rho \in \NP(\delta) \cap \NP(\eta)$.
\end{ex}

The following lemma is an obvious consequence of the previous definition.

\begin{lemma}\label{lem:Sapproxmatchablecase} Let $(\delta,\eta) \in \Accord'^2$ be such that $\NP(\delta) \cap \NP(\eta) \neq \varnothing$. The following assertions hold:
\begin{enumerate}[label=$(\roman*)$,itemsep=1mm]
    \item for all $\rho \in \NP(\delta) \cup \NP(\eta)$ such that $(\delta,\eta)$ admits an $\Smov(\rho)$-approximation  $(\xi_\rho^{(\delta)},\varsigma_\rho^{(\eta)})$, the colorations $\col_{\xi_\rho^{(\delta)}}$ and $\col_{\varsigma_\rho^{(\eta)}}$ are matchable;
    \item if $\col_\delta$ and $\col_\eta$ are matchable, then there exists $\rho \in \NP(\delta) \cup \NP(\delta)$ such that the $\Smov(\rho)$-approximation of the pair $(\delta,\eta)$ is $(\delta, \eta)$ itself. 
\end{enumerate}
\end{lemma} 

In the remainder of this section, we prove that if $(\delta,\eta) \in \Accord'^2$ is such that both:
\begin{enumerate}[label=$\bullet$,itemsep=1mm]
    \item $\NP(\delta) \cap \NP(\eta) \neq \varnothing$; and,
    \item $\col_\delta$ and $\col_\eta$ are not matchable,
\end{enumerate}
then, in order to compute $\opResAc'(\delta,\eta)$, it is enough to compute $\opResAc'(\delta)$, $\opResAc'(\eta)$, and $\opResAc'(\xi_\rho^{(\delta)}, \varsigma_\rho^{(\eta)})$ for each $\rho \in \NP(\delta) \cup \NP(\eta)$ such that $(\delta,\eta)$ admits an $\Smov(\rho)$-approximation. Let us first prove the following lemma.

\begin{lemma} \label{lem:upsourceuptargetexist} Let $\delta \in \Accord'$, $v \in \NP(\delta)_0^{{\gsquare}}$, and $w \in \NP(\delta)_0^{{\rsquare}}$. The following assertions hold:
\begin{enumerate}[label=$(\roman*)$, itemsep=1mm]
    \item The set $\{\varsigma \mid \varsigma \in \opResAc'(\delta),\ t(\varsigma) = v\}$ is a nonempty chain in $(\Accord', \Accordleq)$; and, 
    \item The set $\{\nu \mid \nu \in \opResAc'(\delta),\ s(\nu) = w\}$ is a nonempty chain in $(\Accord', \Accordleq)$.
\end{enumerate}
\end{lemma}

\begin{proof}
The assertions are dual to each other. Given $\varsigma_1, \varsigma_2 \in \{\varsigma \mid \varsigma \in \Accord', t(\varsigma) = v\}$, we have that either $\NP(\varsigma_1)\subset\NP(\varsigma_2)$ or $\NP(\varsigma_2)\subset\NP(\varsigma_1)$. Without loss of generality, consider that $\NP(\varsigma_1)\subset\NP(\varsigma_2)$. As $\varsigma_1$ and $\varsigma_2$ are in $\opResAc'(\delta)$, $s(\varsigma_1)\in\NP(\delta)_0^{{\rsquare}}\cup \NP(\delta)_0^{{\osquare}}$ and thus $s(\varsigma_1)\in\NP(\varsigma_2)_0^{{\rsquare}}\cup \NP(\varsigma_2)_0^{{\osquare}}$. Thus, by \cref{thm:res_clo_1} we have $\varsigma_1\Accordleq\varsigma_2$.
\end{proof}

\begin{remark} \label{rem:syzy.giestotalorder}
    In particular, the syzygy accordions of $\delta$ sharing a common endpoint forms a chain of $(\Accord', \Accordleq)$
\end{remark}

\begin{prop} \label{prop:bestchoiceformatchability}
Let $(\delta,\eta) \in \Accord'^2$ be such that $\NP(\delta) \cap \NP(\eta) \neq \varnothing$ and such that the colorations $\col_\delta$ and $\col_\eta$ are not matchable. Consider $\rho \in \NP(\delta) $, and denote by $v_1$ and $v_2$ its endpoints. Assume that $(\delta,\eta)$ admits an $\Smov(\rho)$-approximation $(\xi_\rho^{(\delta)}, \varsigma_\rho^{(\eta)})$. Then, for any pair $(\xi',\varsigma') \in \Accord'^2$ such that:
\begin{enumerate}[label=$(\alph*)$,itemsep=1mm]
    \item $\xi' \in \opResAc(\delta)$, and $\varsigma' \in \opResAc(\eta)$;
    \item $\rho \in \NP(\xi') $; and,
    \item $\rho$ makes the colorations $\col_{\xi'}$ and $\col_{\varsigma'}$ matchable,
\end{enumerate}
we have $\xi' \Accordleq \xi_\rho^{(\delta)}$ and $\varsigma' \Accordleq \varsigma_\rho^{(\eta)}$.
\end{prop}

\begin{proof}
One of the desired inequalities is obvious if $\xi_\rho^{(\delta)} = \delta$ or if $\varsigma_{\rho}^{(\eta)} = \eta$. We treat the case where $\xi_\rho^{(\delta)} = \Syg_{(\delta,\rho)}$, and we can treat similarly the case where $\varsigma_{\rho}^{(\eta)} = \Syg_{(\eta,\rho)}$. By setting $\sc(\delta) = (\eta_L,C_L)$, without loss of generality, we assume that $\rho \subset \partial C_L$.
 
By assumptions, $\xi'$ is such that $\rho \in \NP(\xi')$, as $\rho$ is part of the boundary of a large cell in $\pmb{\Gamma}(\Prj(\Delta^{\gpoint}))$ (containing at least three projective accordions in its boundary). It implies that one of the endpoints of $\xi'$ is $s(\delta)$, and $\NP(\xi_\rho^{(\delta)}) \subseteq \NP(\xi')$. Therefore $\xi_\rho^{(\delta)} \Accordleq \xi'$ by \cref{lem:upsourceuptargetexist}. If $\xi' \neq \xi_\rho^{(\delta)}$, then $\col_{\xi'}(\{v_1,v_2\}) \subseteq \{{\osquare, \psquare}\}$, which contradicts $(c)$. 
\end{proof}

The following proposition highlights the fact that to calculate $\opResAc(\delta,\eta)$ whenever $\col_\delta$ and $\col_\eta$ are not matchable, we are reduced to calculating $\opResAc(\mu,\nu)$ with a well-chosen collection of pairs $(\mu,\nu) \in \Accord'^2$ such that $\col_\mu$ and $\col_\nu$ are matchable.

\begin{prop} \label{prop:Sapproxisenough}
Let $(\delta,\eta) \in \Accord'^2$ be such that $\NP(\delta) \cap \NP(\eta) \neq \varnothing$, and $\col_\delta$ and $\col_\eta$ are not matchable. Then \[\opResAc(\delta,\eta) = \opResAc(\delta) \cup \opResAc(\eta) \cup \left( \bigcup_{\rho} \opResAc(\xi_\rho^{(\delta)},\varsigma_\rho^{(\eta)}) \right),\]
where the big union is over the projective accordions $\rho \in \NP(\delta) \cup \NP(\eta)$ such that $(\delta,\eta)$ admits an $\Smov(\rho)$-approximation $(\xi_\rho^{(\delta)},\varsigma_\rho^{(\eta)})$.
\end{prop}

\begin{proof}
Set \[\mathscr{W}(\delta,\eta) = \opResAc(\delta) \cup \opResAc(\eta) \cup \left( \bigcup_{\rho} \opResAc(\xi_\rho^{(\delta)},\varsigma_\rho^{(\eta)}) \right),\] and \[ \mathscr{W}^\circ(\delta,\eta) = \bigcup_{\rho} \opResAc(\xi_\rho^{(\delta)},\varsigma_\rho^{(\eta)}).\] It is clear that:
\begin{enumerate}[label=$\bullet$,itemsep=1mm]
    \item $\delta, \eta \in \mathscr{W}(\delta,\eta)$;
    \item $\mathscr{W}(\delta,\eta) \subseteq \opResAc(\delta,\eta)$; and,
    \item $\mathscr{W}(\delta,\eta) \cap \Accord'$ is an ideal of $(\Accord', \Accordleq)$, and so, $\mathscr{W}(\delta,\eta)$ is closed under syzygies by \cref{lem:IdealsclosedSyzygies}.
\end{enumerate}
Set $\Ups(\delta,\eta) = \{(\xi',\varsigma') \mid \xi' \in \opResAc'(\delta), \varsigma' \in \opResAc'(\eta)\}$. We prove by induction  on $\#\Ups(\delta,\eta)$ that $\opResAc(\delta,\eta) \subseteq \mathscr{W}(\delta,\eta)$. If $\#\Ups(\delta,\eta) = 1$, then $\opResAc'(\delta) = \{\delta\}$ and $\opResAc'(\eta) = \{\eta\}$. As $\col_\delta$ and $\col_\eta$ are non matchable, by \cref{prop:nomatchablesmaller}, it means that \[\opResAc(\delta,\eta) = \opResAc(\delta) \cup \opResAc(\eta).\] 

Fix an integer $n \geqslant 1$, and assume that, for any $k \leqslant n$, if $\#\Ups(\delta,\eta) \leqslant k$, then, for all $(\xi',\varsigma') \in \Ups(\delta,\eta)$, we have $\opResAc(\xi',\varsigma') \subseteq \mathscr{W}(\delta,\eta)$. Now suppose that $\#\Ups(\delta,\eta) = n+1$. By \cref{prop:nomatchablesmaller}, we get that \[
    \opResAc(\delta,\eta) = \opResAc(\delta) \cup \opResAc(\eta) \cup \left( \bigcup_{\mu \Accordlneq \delta} \opResAc(\mu,\eta) \right) \cup \left( \bigcup_{\nu \Accordlneq \eta} \opResAc(\delta,\nu) \right). \]
As $\#\Ups(\mu, \eta) < \#\Ups(\delta,\eta)$ for all $\mu \Accordlneq \delta$, and $\#\Ups(\delta, \nu) < \#\Ups(\delta,\eta)$ for all $\nu \Accordlneq \eta$, the induction hypothesis can be applied, and, after simplification, we get
\[\opResAc(\delta,\eta) = \opResAc(\delta) \cup \opResAc(\eta) \cup \left( \bigcup_{\mu \Accordlneq \delta}\mathscr{W}^\circ(\mu,\eta)\right) \cup \left( \bigcup_{\nu \Accordlneq \eta} \mathscr{W}^\circ(\delta,\nu) \right).\]
Let $\mu \Accordlneq \delta$. We show that $\mathscr{W}^\circ(\mu,\eta) \subseteq \mathscr{W}^\circ(\delta,\eta)$. If $\mathscr{W}^\circ(\mu,\eta) = \varnothing$, we are done. Otherwise, let us consider $\rho \in \NP(\mu) \cup \NP(\eta)$ such that $(\mu,\eta)$ admits an $\Smov(\rho)$-approximation $(\xi_\rho^{(\mu)},\varsigma_\rho^{(\eta)})$. Note that $(\delta,\eta)$ also admits an $\Smov(\rho)$-approximation $(\xi_\rho^{(\delta)}, \varsigma_\rho^{(\eta)})$. We thus have:
\begin{enumerate}[label=$\bullet$,itemsep=1mm]
    \item $\xi_\rho^{(\mu)} \Accordleq \mu \Accordleq \delta$ and $\varsigma_\rho^{(\eta)} \Accordleq \eta$;
    \item $\rho \in \NP(\mu) \cup \NP(\eta) \subseteq \NP(\delta) \cup \NP(\eta)$; and,
    \item $\rho$ makes the colorations $\col_{\xi_\rho^{(\mu)}}$ and $\col_{\varsigma_\rho^{(\eta)}}$ matchable.
\end{enumerate}
Using \cref{prop:bestchoiceformatchability}, we get that $\xi_\rho^{(\mu)} \Accordleq \xi_\rho^{(\delta)}$, and so $\opResAc(\xi_\rho^{(\mu)}, \varsigma_\rho^{(\eta)}) \subseteq  \opResAc(\xi_\rho^{(\delta)}, \varsigma_\rho^{(\eta)})$.

As this inclusion is true for any $\rho \in \NP(\mu) \cup \NP(\eta)$, we get that $\mathscr{W}^\circ(\mu,\eta) \subseteq \mathscr{W}^\circ(\delta,\eta)$. By duality, we also get that $\mathscr{W}^\circ(\delta,\nu) \subseteq \mathscr{W}^\circ(\delta,\eta)$ for all $\nu \Accordlneq \eta$, and thus it allows to get that $\opResAc(\delta,\eta) \subseteq \mathscr{W}(\delta,\eta)$. 
\end{proof}
        %\ibenj{Début ancienne version}
	%\input{Matchable.tex}
        %\ibenj{Fin ancienne version}
	
	\section{Moves on pairs of non-projective accordions}
	\label{sec:Moves}
	\pagestyle{plain}

In this section, we will review some technical details that outline all possible interactions between two non-projective accordions. Our understanding of those interactions motivates us to introduce \emph{``moves''} on pairs $(\delta,\eta) \in \Accord'^2$. In \cref{tab:Allmovessummary}, we summarize all of the crucial moves we must consider to compute $\mathfrak{E}_{\opResAc'(\delta,\eta)}$, depending on the matchability of the colorations, and the notion of \emph{exclusivity} on endpoints of the two accordions we are considering. We invite the reader to go directly to \cref{sec:TreePart2} to see precisely how we use them in the computation of the resolving closure of any collection of non-projective accordions and, therefore, of the resolving subcategories of $\rep(Q,R)$.
\begin{table}[!ht]
\centering
\begin{tiny}
\begin{NiceTabular}{|c|c|c|c|c|c|}
\CodeBefore
\cellcolor{gray!20}{1-1}
\cellcolor{gray!20}{2-1}
\cellcolor{gray!20}{3-1}
\cellcolor{gray!20}{4-1}
\cellcolor{gray!20}{5-1}
\cellcolor{gray!20}{6-1}
\cellcolor{gray!20}{7-1}
\cellcolor{gray!20}{8-1}
\cellcolor{gray!20}{2-2}
\cellcolor{gray!20}{3-2}
\cellcolor{gray!20}{4-2}
\cellcolor{gray!20}{5-2}
\cellcolor{gray!20}{6-2}
\cellcolor{gray!20}{7-2}
\cellcolor{gray!20}{8-2}
\cellcolor{gray!20}{3-3}
\cellcolor{gray!20}{4-3}
\cellcolor{gray!20}{5-3}
\cellcolor{gray!20}{6-3}
\cellcolor{gray!20}{7-3}
\cellcolor{gray!20}{8-3}
\cellcolor{gray!20}{4-4}
\cellcolor{gray!20}{5-4}
\cellcolor{gray!20}{6-4}
\cellcolor{gray!20}{7-4}
\cellcolor{gray!20}{8-4}
\cellcolor{gray!20}{4-5}
\cellcolor{gray!20}{4-6}
\Body
\hline
$\NP(\delta) \cap \NP(\eta) = \varnothing$   & \multicolumn{5}{c|}{$\ast$}                                                                                                                      \\ \hline
\multirow{7}{*}{$\NP(\delta) \cap \NP(\eta) \neq \varnothing$} &    non-matchable              & \multicolumn{4}{c}{$\Smov$-move}       \\                                       \cline{2-6}                                                  
                                   &  \multirow{6}{*}{  matchable} &  weakly                & \multicolumn{3}{c}{$\W$ and $\KE$-move}                            \\ \cline{3-6} 
                                     &                          & \multirow{5}{*}{ strongly} &  $\delta$ and $\eta$ are &  non-crossing &  crossing  \\ \cline{4-6} 
                                   &                          &         & satisfying \ref{0E} & $\ast$ & $\ast$  \\ \cline{4-6} 
                                    &                     &                     &  satisfying \ref{2E} & $\KE$-move & $\ast$ \\ \cline{4-6} 
                                    &                     &                     & satisfying \ref{3E} & $\KE'$ and $\V$-move & $\V$-move \\ \cline{4-6} 
                                    &                     &                     & satisfying \ref{4E} & $\CoZ$-move & $\X$-move \\ \hline
\end{NiceTabular}
\end{tiny}
\caption{Summary of the relevant moves to apply on any pair $(\delta,\eta) \in \Accord'^2$ to determine $\mathfrak{E}_{\opResAc'(\delta,\eta)}$.}
\label{tab:Allmovessummary}
\end{table}

\subsection{Exclusive vertices}
\label{ss:exclusive}
We define the \emph{exclusive vertices}, which will reveal to be crucial to compute precisely $\mathfrak{E}_{\opResAc'(\delta,\eta)}$.

\begin{definition} \label{def:exclusive}
Let $(\delta,\eta) \in \Accord^2$ such that $\NP(\delta) \cap \NP(\eta) \neq \varnothing$, with matchable colorations, and $v\in\mathcal{M}_{\rpoint}$. Let us assume that $\col_\delta$ and $\col_\eta$ are matching. We say that $v$ is:
\begin{enumerate}[label=$\bullet$, itemsep=1mm]
\item \new{exclusive to} $\delta$  \new{in the pair} $(\delta, \eta)$ if $v \in \NP(\delta)_0$, and either $v \notin \NP(\eta)_0$ or $v \in \NP(\delta)_0^{{\osquare}} \setminus \NP(\eta)_0^{{\osquare}}$ or $v \in \NP(\delta)_0^{{\psquare}}$,  
\item \new{exclusive to} $\eta$ \new{in} $(\delta, \eta)$ if $v \in \NP(\eta)_0$, and either $v \notin \NP(\delta)_0$ or $v \in \NP(\eta)_0^{{\osquare}} \setminus \NP(\delta)_0^{{\osquare}}$ or $v \in \NP(\eta)_0^{{\psquare}}$
\item \new{exclusive in} $(\delta, \eta)$ if $v$ is exclusive to $\delta$ or to $\eta$ in $(\delta, \eta)$.
\end{enumerate}
\end{definition}
Note that a common endpoint of $\delta$ and $\eta$ cannot be exclusive to both $\delta$ and $\eta$ in $(\delta, \eta)$. 

\begin{ex} In \cref{fig:ExExclcusivity}, we give an example of exclusive points in $(\delta, \eta)$.
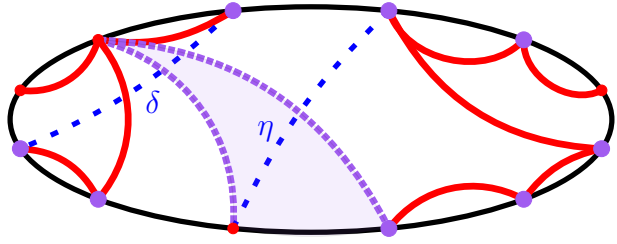
\begin{figure}[!ht]
\centering 
    \begin{tikzpicture}[mydot/.style={
					circle,
					thick,
					fill=white,
					draw,
					outer sep=0.5pt,
					inner sep=1pt
				}, scale = 1]
		\tikzset{
		osq/.style={
        rectangle,
        thick,
        fill=white,
        append after command={
            node [
                fit=(\tikzlastnode),
                orange,
                line width=0.3mm,
                inner sep=-\pgflinewidth,
                cross out,
                draw
            ] {}}}}
		\draw[line width=0.7mm,black] (0,0) ellipse (4cm and 1.5cm);
		\foreach \X in {0,1,...,23}
		{
		\tkzDefPoint(4*cos(pi/12*\X),1.5*sin(pi/12*\X)){\X};
		};
		
		\draw[line width=0.9mm ,bend left =60,red](1) edge (3);
		\draw[line width=0.9mm ,bend right =60,red](5) edge (3);
		\draw[line width=0.9mm ,bend right =30,red](5) edge (23);
		\draw[line width=0.9mm ,bend left =20,red](21) edge (23);
		\draw[line width=0.9mm ,bend left =20,red](7) edge (9);
		\draw[line width=0.9mm ,bend left =40,red](19) edge (21);
		\draw[line width=0.9mm ,bend left =30,mypurple,densely dashdotted](9) edge (19);
		\draw[line width=0.9mm ,bend left =40,mypurple,densely dashdotted](9) edge (17);
		\draw[line width=0.9mm ,bend left =40,red](9) edge (15);
		\draw[line width=0.9mm ,bend left =40,red](9) edge (11);
		\draw[line width=0.9mm ,bend left=30,red](13) edge (15);
		
		\draw[line width=0.7mm ,bend right=10,blue, loosely dashed](5) edge (17);
		\draw[line width=0.7mm ,bend left=10,blue, loosely dashed](7) edge (13);
		
		\filldraw [fill=mypurple,opacity=0.1] (9) to [bend left=30] (19) to [bend left=10] (17) to [bend right=40] cycle ;

		\foreach \X in {1,3,...,23}
		{
		\tkzDrawPoints[fill =red,size=4,color=red](\X);
		};
		
		\foreach \X in {3,5,7,13,15,19,21,23}
		{
		\tkzDrawPoints[size=6,color=mypurple](\X);
		};

		\tkzDefPoint(-2.1,0.5){gamma};
		\tkzLabelPoint[blue](gamma){\Large $\delta$}
		\tkzDefPoint(-0.6,0.1){gamma};
		\tkzLabelPoint[blue](gamma){\Large $\eta$}
    \end{tikzpicture}
\caption{\label{fig:ExExclcusivity} Given a $\rpoint$-dissected marked disc $(\pmb{\Sigma}, \mathcal{M}, \Delta^{\rpoint})$, where arcs of $\Delta^{\rpoint} = \Prj(\Delta^{\gpoint})$ for some  $\gpoint$-dissection $\Delta^{\gpoint}$, the purple marked points correspond to the exclusive ones in $(\delta, \eta) \in \Accord'^2$. The purple dashed lines are projective accordions in $\NP(\delta) \cap \NP(\eta)$.}
\end{figure}
\end{ex}

The following lemma enumerates the different hypotheses on the exclusivity which partition all the configurations that could be satisfied by a pair $(\delta, \eta) \in \Accord'^2$ such that both $\NP(\delta) \cap \NP(\eta) \neq \varnothing$, and $\col_\delta$ and $\col_\eta$ are matchable.

\begin{lemma} \label{lem:allexclusivityhyp} Let $(\delta, \eta) \in \Accord'^2$  be such that $\NP(\delta) \cap \NP(\eta) \neq \varnothing$, with matchable colorations. We assume that $\col_\delta$ and $\col_\eta$ match. Then, one of the following assertions occurs:
\begin{enumerate}[label=$(\arabic*E)$,itemsep=1mm]
\setcounter{enumi}{-1}
    \item \label{0E} There exists $\varsigma \in \{\delta,\eta\}$ such that $s(\varsigma)$ and $t(\varsigma)$ are not exclusive to $\varsigma$ in $(\delta, \eta)$;
\setcounter{enumi}{1}
    \item \label{2E} Each accordion has exactly one endpoint that is exclusive in $(\delta,\eta)$;
    \item \label{3E} Among all the vertices  $\{s(\delta), t(\delta), s(\eta),t(\eta)\}$ exactly one is not exclusive to the curve naturally associated. 
    \item \label{4E} All the vertices of $\{s(\delta), t(\delta), s(\eta),t(\eta)\}$ are exclusive in $(\delta,\eta)$;
\end{enumerate}
\end{lemma}

In the following, we treat each of those cases. Let us first treat the case \ref{0E} . For the other ones, we introduce, in \cref{sec:Moves}, moves on the pair $(\delta,\eta)$, which help us to describe $\opResAc'(\delta,\eta)$.

\begin{prop}
\label{prop:0E}
Let $(\delta, \eta) \in \Accord'^2$. Let $\col_\delta$ be a coloration of $\NP(\delta)_0$, and $\col_\eta$ be one of $\NP(\eta)_0$. Then the following assertions are equivalent:
\begin{enumerate} [label=$(\roman*)$, itemsep=1mm]
    \item We have $\NP(\delta) \cap \NP(\eta) \neq \varnothing$, $\col_\delta$ and $\col_\eta$ matchable, and $(\delta,\eta)$ satisfies \ref{0E}; and,
    \item We have $\eta \Accordleq \delta$ or $\delta \Accordleq \eta$.
\end{enumerate}
\end{prop}
\begin{proof}
The fact that $(ii)$ implies $(i)$ is obvious.

Assume, without loss of generality, that:
\begin{enumerate}[label=$\bullet$, itemsep=1mm]
    \item $\col_\delta$ and $\col_\eta$ are matching; and,
    \item $s(\eta)$ and $t(\eta)$ are not exclusive to $\eta$ in $(\delta,\eta)$.
\end{enumerate}
First $s(\eta) \in \NP(\eta)_0$. As $s(\eta)$ is not exclusive to $\eta$ in $(\delta, \eta)$, the following statements hold:
\begin{enumerate}[label=$\bullet$,itemsep=1mm]
    \item $s(\eta) \in \NP(\delta)_0$;
    \item $s(\eta) \notin \NP(\delta)_0^{{\osquare}} \setminus \NP(\eta)_0^{{\osquare}}$; and, 
    \item $s(\eta) \notin \NP(\delta)_0^{{\psquare}}$.
\end{enumerate}
Therefore, by convention on the choice of the coloration for source and target of accordions, $s(\eta)\in \NP(\delta)_0^{{\rsquare}} \cup \{s(\delta)\}$. Similarly $t(\eta)\in \NP(\delta)_0^{{\gsquare}}\cup \{ t(\delta)\}$.  So we get $\eta \in \opResAc(\delta)$ by \cref{thm:res_clo_1}.
\end{proof}

\subsection{KE-move}
\label{ss:KE}
This first move is a translation of the computation of the kernel of an epimorphism and an arrow extension between indecomposable representations. It will be defined for any pair $(\delta,\eta) \in \Accord'^2$, but it will be relevant whenever $\col_\delta$ and $\col_\eta$ match.

\begin{definition}\label{def:completeaskernelorextension}
Let $(\delta,\eta,\kappa) \in \Accord'^3$. We say that $\kappa$ is a \new{kernel completion} of $(\delta,\eta)$ if either: 
\begin{enumerate}[label=$\bullet$, itemsep=1mm]
    \item there exists an epimorphism $\MM(\delta) \twoheadrightarrow \MM(\eta)$ in $\rep(Q,R)$ and $\MM(\kappa)$ is isomorphic to its kernel; or,
    \item there exists an epimorphism $\MM(\eta) \twoheadrightarrow \MM(\delta)$ in $\rep(Q,R)$ and $\MM(\kappa)$ is isomorphic to its kernel.
\end{enumerate}
We say that $\kappa$ is a \new{$1$-extension completion} of $(\delta,\eta)$ if one of the following holds:
\begin{enumerate}[label=$\bullet$,itemsep=1mm]
    \item there is a non-split short exact sequence \[\begin{tikzcd}
	\MM(\delta) & \MM(\kappa) & \MM(\eta)
	\arrow["k",tail, from=1-1, to=1-2]
	\arrow["f",two heads, from=1-2, to=1-3]
\end{tikzcd}\]  in $\rep(Q,R)$; or,
\item there is a non-split short exact sequence \[\begin{tikzcd}
	\MM(\eta) & \MM(\kappa) & \MM(\delta)
	\arrow["k",tail, from=1-1, to=1-2]
	\arrow["f",two heads, from=1-2, to=1-3]
\end{tikzcd}\]  in $\rep(Q,R)$.
\end{enumerate}
If $\kappa$ is either a kernel or a $1$-extension completion of $(\delta,\eta)$, then we call $\kappa$ the \new{$\KE$-completion} of $(\delta,\eta)$.
\end{definition}

\begin{lemma} \label{lem:uniqKE}
For any pair $(\delta,\eta) \in \Accord'^2$, there exists at most one $\kappa \in \Accord'$ such that $\kappa$ is a $\KE$-completion of $(\delta, \eta)$. Moreover, if such a $\kappa$ exists, then $(\delta,\eta)$ has matchable colorations.
\end{lemma}

\begin{proof}
The result follows from the fact that the extension spaces between indecomposable representations of a gentle tree are at most one-dimensional, and to have such a $\KE$-completion, we need to have an arrow extension between two out of the three accordions $\delta, \eta, \kappa$. In addition, by \cref{prop:matchableextorepi}, if such a $\KE$-completion exists, we get that $(\delta,\eta)$ admits matchable colorations.
\end{proof}

\begin{definition} \label{def:KEmove}
For any $(\delta,\eta)\in \Accord'^2$, we define the \new{$\KE$-move} of $(\delta,\eta)$ as follows:
\[\KE(\delta,\eta) = \begin{cases}
\{\delta,\eta,\kappa\} & \text{if } \{\delta,\eta\} \text{ admits a } \KE\text{-completion } \kappa \in \Accord' \\
\{\delta,\eta\} & \text{otherwise.}
\end{cases}\]
\end{definition}

\begin{ex} \label{ex:KEmove} \cref{fig:Triangle} shows an example of the $\KE$-move of two accordions. This picture represents the general behavior of such a completion.
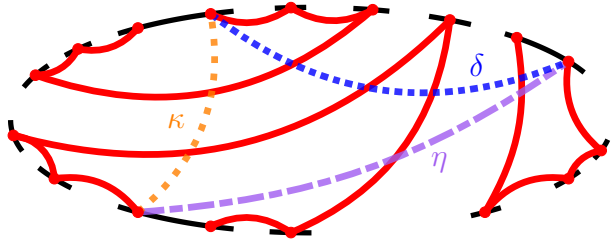
\begin{figure}[!ht]
        \centering
       \begin{tikzpicture}[mydot/.style={
					circle,
					thick,
					fill=white,
					draw,
					outer sep=0.5pt,
					inner sep=1pt
				}, scale = 1]
		\tikzset{
		osq/.style={
        rectangle,
        thick,
        fill=white,
        append after command={
            node [
                fit=(\tikzlastnode),
                orange,
                line width=0.3mm,
                inner sep=-\pgflinewidth,
                cross out,
                draw
            ] {}}}}
		\draw [line width=0.7mm,domain=25:50] plot ({4*cos(\x)}, {1.5*sin(\x)});
		\draw [line width=0.7mm,domain=58:67] plot ({4*cos(\x)}, {1.5*sin(\x)});
		\draw [line width=0.7mm,domain=75:82] plot ({4*cos(\x)}, {1.5*sin(\x)});
		\draw [line width=0.7mm,domain=91:97] plot ({4*cos(\x)}, {1.5*sin(\x)});
		\draw [line width=0.7mm,domain=105:130] plot ({4*cos(\x)}, {1.5*sin(\x)});
		\draw [line width=0.7mm,domain=137:145] plot ({4*cos(\x)}, {1.5*sin(\x)});
		\draw [line width=0.7mm,domain=151:163] plot ({4*cos(\x)}, {1.5*sin(\x)});
		\draw [line width=0.7mm,domain=180:196] plot ({4*cos(\x)}, {1.5*sin(\x)});
		\draw [line width=0.7mm,domain=205:217] plot ({4*cos(\x)}, {1.5*sin(\x)});
        \draw [line width=0.7mm,domain=230:255] plot ({4*cos(\x)}, {1.5*sin(\x)});
        \draw [line width=0.7mm,domain=262:270] plot ({4*cos(\x)}, {1.5*sin(\x)});
        \draw [line width=0.7mm,domain=302:309] plot ({4*cos(\x)}, {1.5*sin(\x)});
        \draw [line width=0.7mm,domain=324:333] plot ({4*cos(\x)}, {1.5*sin(\x)});
        \draw [line width=0.7mm,domain=340:350] plot ({4*cos(\x)}, {1.5*sin(\x)});
		\foreach \X in {0,...,45}
		{
		\tkzDefPoint(4*cos(pi/23*\X),1.5*sin(pi/23*\X)){\X};
		};
		
	    \draw[line width=0.9mm ,bend right =30,red] (4) edge (44);
	    \draw[line width=0.9mm ,bend left =30,red] (42) edge (44);
	    \draw[line width=0.9mm ,bend left =30,red] (39) edge (42);
	    \draw[line width=0.9mm ,bend right =20,red] (39) edge (6);

	    \draw[line width=0.9mm ,bend left=30,red] (32) edge (34);
	    \draw[line width=0.9mm ,bend right =30,red] (34) edge (8);
	    \draw[line width=0.9mm ,bend right =30,red] (24) edge (8);
	    \draw[line width=0.9mm ,bend left =40,red] (24) edge (27);
	    \draw[line width=0.9mm ,bend left =30,red] (27) edge (30);
	    
	    \draw[line width=0.9mm ,bend right =30,red] (14) edge (12);
	    \draw[line width=0.9mm ,bend right =30,red] (12) edge (10);
	    \draw[line width=0.9mm ,bend right =30,red] (20) edge (10);
	    \draw[line width=0.9mm ,bend right =30,red] (20) edge (18);
	    \draw[line width=0.9mm ,bend right =30,red] (18) edge (16);
	    
	    \draw [line width=0.9mm, mypurple,dash pattern={on 10pt off 2pt on 5pt off 2pt}, bend right=15,opacity=0.8] (30) edge (4);
		\draw [line width=0.9mm,blue, dashdotted, bend right = 30, opacity=0.8]  (14) edge (4);
		\draw [line width=0.9mm,orange, loosely dotted, bend right = 30, opacity=0.8]  (30) edge (14);
        
		\foreach \X in {4,6,8,10,12,14,16,18,20,24,27,30,32,34,39,42,44}
		{
		\tkzDrawPoints[fill =red,size=4,color=red](\X);
		};

		\tkzDefPoint(2.2,1){gammaM};
		\tkzLabelPoint[blue](gammaM){\Large $\delta$}
		\tkzDefPoint(1.7,-0.3){gammaP};
		\tkzLabelPoint[mypurple](gammaP){\Large $\eta$}
		\tkzDefPoint(-1.8,0.2){rho};
		\tkzLabelPoint[orange](rho){\Large $\kappa$}
    \end{tikzpicture}
        \caption{\label{fig:Triangle} An example of $\KE$-move of $\{\delta, \eta\}$. More precisely, $\kappa$ is a kernel completion of $\{\delta,\eta\}$.}
    \end{figure}
\end{ex}

\begin{remark} \label{rem:KETriangle}
Given $(\delta, \eta) \in \Accord'^2$, if $(\delta,\eta)$ admits a \textbf{KE}-completion $\kappa$, by the fact that $(\pmb{\Sigma}, \mathcal{M})$ is a marked disc with $\mathcal{M} \subset \partial \pmb{\Sigma}$, and by \cref{prop:geom_ext}, we get that $\delta$, $\eta$ and $\kappa$ form a triangle in $(\pmb{\Sigma}, \mathcal{M}, \Delta^{\gpoint})$ up to homeomorphism.
\end{remark}

\begin{prop}
\label{prop:SmallEV3}
Let $(\delta, \eta) \in \Accord'^2$ be such that:
\begin{enumerate}[label = $\bullet$, itemsep=1mm]
    \item $\col_\delta$ and $\col_\eta$ are matching;
    \item $\{\delta, \eta\} \in \Anti(\Accord', \Accordleq)$; and,
    \item $\delta$ and $\eta$ satisfy \ref{2E}.
\end{enumerate}
The following assertions hold:
\begin{enumerate}[label = $(\alph*)$, itemsep=1mm]
    \item \label{SmallEV3a} if $\delta$ and $\eta$ have a common endpoint, then $\mathfrak{E}_{\opResAc'(\delta,\eta)}= \KE(\delta,\eta)$; 
    \item \label{SmallEV3b} otherwise, we have $\mathfrak{E}_{\opResAc'(\delta, ,\eta)} = \{\delta,\eta\}$.
\end{enumerate}
\end{prop}

\begin{proof}
Let us assume that $\delta$ and $\eta$ share a common endpoint. Without loss of generality, we can assume that:
\begin{enumerate}[label=$\bullet$,itemsep=1mm]
    \item $\delta$ is above $\eta$ following the orientation of $\NP(\delta)\cup\NP(\eta)$; and,
    \item their common endpoint is their target, meaning $t(\delta) = t(\eta)$.
\end{enumerate}
Moreover, in both cases, by \cref{prop:folkloreresaccord} $(v)$, we only have to check that $\bigcup_{\nu \in \KE(\delta,\eta)} \opResAc(\nu)$ is closed under extensions.

Assume that $(\delta,\eta)$ admits a \textbf{KE}-completion $\kappa$. Otherwise, we can set $\opResAc(\kappa) = \Prj(\Delta^{\gpoint})$.
To check the extension closure of $\opResAc(\delta) \cup \opResAc(\eta) \cup \opResAc(\kappa)$, it is enough to check the extensions of any non-comparable pair $(\vartheta, \varsigma) \in \Accord'^2$ such that each curve lies in exactly one of the three $\opResAc(\delta), \opResAc(\eta), \opResAc(\kappa)$ (as each term is extension closed). Consider as a first step that  $\vartheta \in  \opResAc'(\delta) \setminus \opResAc'(\eta)$, and $\varsigma \in \opResAc'(\eta) \setminus \opResAc'(\delta)$. By assumptions, $\vartheta$ and $\varsigma$ are crossing or having a common endpoint in $\NP(\delta)\cap\NP(\eta)$  which implies that the coloration of $t(\vartheta)$ and $t(\varsigma)$ are orange or green in both of the coloration of $\NP(\delta)_0$ and $\NP(\eta)_0$. The sources are colored red or orange. Two cases naturally appear in the computations : 
\begin{itemize}
\item $\varsigma$ is not an indecomposable summand of the syzygies of $\eta$. Then, the orientation of the crossing is the one following common colorations on teh common part, and the conclusion is straightforward
\item $\varsigma$ is an indecomposable summand of the syzygies of $\eta$. In this situation, the completion may be needed. Then, both endpoints of $\varsigma$ are orange. The orientation of the crossing follows the one of $\NP(\kappa)$. Thus $s(\varsigma)\in\NP(\kappa)_0^{{\osquare}}$ and $s(\vartheta)\in\NP(\kappa)_0^{{\osquare}}\cup \NP(\kappa)_0^{{\gsquare}}$. 
\end{itemize}
Similar arguments work by exchanging the roles of $\kappa$, $\delta$, and $\eta$. We proved \ref{SmallEV3a}.

Now let us assume that $\delta$ and $\eta$ do not share a common endpoint. If $s(\delta)$ and $s(\eta)$ are exclusive in $(\delta,\eta)$, then  $t(\delta) = t(\eta)$, which is absurd. A similar contradiction would arise if $t(\delta)$ and $t(\eta)$ were exclusive in $(\delta,\eta)$. So,  without loss of generality, we can assume that $s(\delta)$ and $t(\eta)$ are exclusive in $(\delta, \eta)$. We illustrate the configuration in \cref{fig:Resuniona2config}. 
    \begin{figure}[!ht]
        \centering
       \begin{tikzpicture}[mydot/.style={
					circle,
					thick,
					fill=white,
					draw,
					outer sep=0.5pt,
					inner sep=1pt
				}, scale = 1]
		\tikzset{
		osq/.style={
        rectangle,
        thick,
        fill=white,
        append after command={
            node [
                fit=(\tikzlastnode),
                orange,
                line width=0.3mm,
                inner sep=-\pgflinewidth,
                cross out,
                draw
            ] {}}}}
		\draw [line width=0.7mm,domain=50:130] plot ({4*cos(\x)}, {1.5*sin(\x)});
        \draw [line width=0.7mm,domain=230:310] plot ({4*cos(\x)}, {1.5*sin(\x)});
		\foreach \X in {0,1}
		{
		\tkzDefPoint(4*cos(pi/5*\X +pi/3),1.5*sin(pi/5*\X + pi/3)){\X};
		};
		\foreach \X in {2,3}
		{
		\tkzDefPoint(4*cos(pi/5*(\X-2) +4*pi/3),1.5*sin(pi/5*(\X-2) + 4*pi/3)){\X};
		};
		
		\draw[line width=0.7mm ,bend right=10,orange, dashed](-2,1.3) edge (3);
		
		\draw[line width=0.7mm ,bend right=30,blue, loosely dotted](1) edge (2,-1.3);
		
		\draw[line width=0.9mm ,bend right=40,red, densely dashdotted] (0.7,1.45) edge (-0.8,-1.45);
		
		\draw [line width=0.9mm, mypurple,dash pattern={on 10pt off 2pt on 5pt off 2pt}, bend right=20,opacity=0.5] (1) edge (2);
		\draw [line width=0.9mm,mypurple,domain=60:96, dash pattern={on 10pt off 2pt on 5pt off 2pt},opacity=0.5] plot ({4*cos(\x)}, {1.5*sin(\x)});
		\draw [line width=0.9mm, mypurple,dash pattern={on 10pt off 2pt on 5pt off 2pt}, bend right=20,opacity=0.5] (0) edge (3);
		\draw [line width=0.9mm,mypurple,domain=240:276, dash pattern={on 10pt off 2pt on 5pt off 2pt},opacity=0.5] plot ({4*cos(\x)}, {1.5*sin(\x)});
		
        \filldraw [fill=mypurple,opacity=0.1] (1) to [bend right=20] (2) to [bend right=10] (3) to [bend left=20] (0) to [bend right=10] cycle ;
        
		\foreach \X in {0,...,3}
		{
		\tkzDrawPoints[fill =red,size=4,color=red](\X);
		};
		
		\tkzDrawPoint[fill =red,size=4,color=red](0.7,1.45);
		\tkzDrawPoint[fill =red,size=4,color=red](-0.8,-1.45);
		\tkzDrawPoint[fill =red,size=4,color=red](-2,1.3);
		\tkzDrawPoint[fill =red,size=4,color=red](2,-1.3);

		\tkzDefPoint(-2.1,1){gammaM};
		\tkzLabelPoint[orange](gammaM){\Large $\delta$}
		\tkzDefPoint(1.3,-0.5){gammaP};
		\tkzLabelPoint[blue](gammaP){\Large $\eta$}
		\tkzDefPoint(-.8,0.7){rho};
		\tkzLabelPoint[red](rho){\Large $\rho$}
    \end{tikzpicture}
        \caption{\label{fig:Resuniona2config} Any projective accordion $\rho$ contained in the purple area is in $\NP(\delta) \cap \NP(\eta)$.}
    \end{figure}
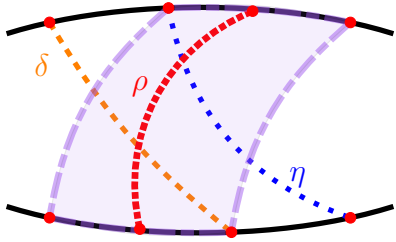 As previously, we only need to prove that $\opResAc(\delta) \cup \opResAc(\eta)$ is closed under extensions. Let $\vartheta \in \opResAc'(\delta)$ and $\varsigma \in \opResAc'(\eta)$. 
    
    Assume that $\OvExt(\vartheta, \varsigma) \neq \varnothing$. Then $\vartheta$ crosses $\varsigma$, implying that they both cross a common projective accordion  $\rho$ in $\NP(\delta) \cap \NP(\eta)$.Considering that $\vartheta$ and $\varsigma$ must cross, we deduce that both of the endpoints of $\varsigma$ or $\vartheta$ must land on vertices of $\NP(\delta)_0 \cap \NP(\eta)_0$. Suppose, without loss of generality, that it is $\varsigma$. It follows that $\varsigma\in\opResAc'(\delta)$ which offers us the conclusion
    
    Assume that $\ArExt(\vartheta, \varsigma) \neq \varnothing$. Then they admit a common endpoint which must be in $\NP(\delta)_0 \cap \NP(\eta)_0$. Up to symmetry, suppose the common endpoint is at the bottom of the common part. Thus $s(\varsigma)\in\NP(\eta)_0^{{\osquare}}$. If $s(\varsigma)\in\NP(\delta)_0^{{\osquare}}$, this point is either in a common large cell, which would imply that one of $\delta$ or $\eta$ is included in the large cell, which is in contradiction with the hypotheses. Otherwise, it is $t(\delta)$ but $t(\delta)\in \NP(\delta)_0^{{\osquare}}$ would imply that the common endpoint is in the top part, which is a contradiction. Thus $s(\varsigma)\in\NP(\eta)_0^{{\osquare}}\cap \NP(\delta)_0^{{\gsquare}}$. It follows that $s(\varsigma)=s(\eta)$. This contradicts $s(\eta)$ being non-exclusive. We proved \ref{SmallEV3b}.
\end{proof}

Let us summarize the result we have proved as follows.

\begin{cor} \label{cor:KEmove}
Let $(\delta, \eta) \in \Accord'^2$ be such that 
\begin{enumerate}[label=$\bullet$,itemsep=1mm]
    \item $\col_\delta$ and $\col_\eta$ are matching;
    \item $\{\delta, \eta\} \in \Anti(\Accord', \Accordleq)$; and,
    \item $\delta$ and $\eta$ satisfy \ref{2E}.
\end{enumerate}
Then $\mathfrak{E}_{\opResAc'(\delta,\eta)} = \KE(\delta,\eta)$.
\end{cor}

\subsection{Co-Z-move}
\label{sss:CoZ}
This second move takes into account the behavior of two consecutive (higher) extensions involving a projective accordion $\rho \in \NP(\delta) \cap \NP(\eta)$ such that, by denoting $v$ and $w$ its endpoints, we have  \[\col_\delta(\{v,w\}) \nsubseteq \{{\osquare}, {\psquare}\} \text{ and } \col_\eta(\{v,w\}) \nsubseteq \{{\osquare}, {\psquare}\}.\]

To define this move, we need to introduce vertices associated to $(\delta, \eta) \in \Accord'^2$ and some given vertex $v \in \NP(\delta)_0 \cap \NP(\eta)_0$. \cref{lem:upsourceuptargetexist} allows us to give the following definition.

\begin{definition}
Let $\delta\in\Accord'$, $v\in\NP(\delta)_0^{{\gsquare}}$, $w\in\NP(\delta)_0^{{\rsquare}}$. The \new{upper source of} $\delta$ \new{associated to} $v$ denoted by $\s_{(\delta)}(v)$ is the source of the maximal accordion in  $\{\varsigma \mid \varsigma \in \opResAc'(\delta),\ t(\varsigma) = v\}$ with respect to $\Accordleq$.
Similarly, we define the \new{upper target of} $\delta$ \new{associated to} $w$, denoted by $\t_{(\delta)}(w)$, as the target of the maximal accordion in $\{\nu \mid \nu \in \opResAc'(\delta),\ s(\nu) = w\}$.
\end{definition}

We can make $\s_{(\delta)}(v)$ and $\t_{(\delta)}(v)$ explicit.

\begin{lemma}\label{rem:useful_up_s}
Let $\delta \in \Accord'$, $v \in \NP(\delta)_0^{{\gsquare}}$ and $w \in \NP(\delta)_0^{{\rsquare}}$. The following assertions hold:
\begin{enumerate}[label=$(\roman*)$, itemsep=1mm]
    \item we have $\displaystyle \s_{(\delta)}(v)= \begin{cases}
    s(\delta) &  \begin{matrix} 
    \text{if there exists } \rho \in \Prj(\Delta^{\gpoint}) \text{ such that both} \\
    v \text{ is an endpoint of }\rho \text{ and } \rho \text{ crosses } \delta; \hfill \end{matrix} \\
    w_R(\delta) & \text{otherwise;}
    \end{cases}$
    \item we have $\displaystyle \t_{(\delta)}(w)= \begin{cases}
    t(\delta) & \begin{matrix} 
    \text{if there exists } \rho \in \Prj(\Delta^{\gpoint}) \text{ such that both} \\
    w \text{ is an endpoint of }\rho \text{ and } \rho \text{ crosses } \delta; \hfill \end{matrix} \\
    w_L(\delta) & \text{otherwise.}
    \end{cases}$
\end{enumerate}
\end{lemma}

\begin{proof}
This follows from \cref{prop:arcsaccordions}, and the exclusion conditions in the proof of \cref{thm:res_clo_1}.
\end{proof}
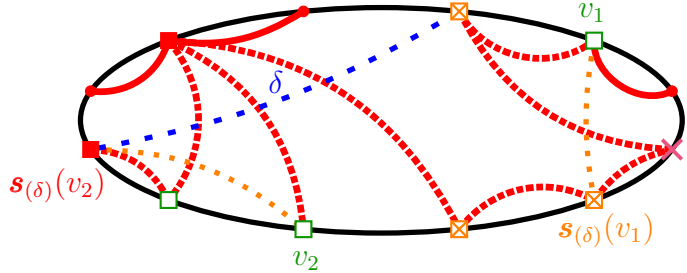
\begin{figure}[!ht]
\centering 
    \begin{tikzpicture}[mydot/.style={
					circle,
					thick,
					fill=white,
					draw,
					outer sep=0.5pt,
					inner sep=1pt
				}, scale = 1]
		\tikzset{
		osq/.style={
        rectangle,
        thick,
        fill=white,
        append after command={
            node [
                fit=(\tikzlastnode),
                orange,
                line width=0.3mm,
                inner sep=-\pgflinewidth,
                cross out,
                draw
            ] {}}}}
		\draw[line width=0.7mm,black] (0,0) ellipse (4cm and 1.5cm);
		\foreach \X in {0,1,...,23}
		{
		\tkzDefPoint(4*cos(pi/12*\X),1.5*sin(pi/12*\X)){\X};
		};

		\draw[line width=0.9mm ,bend left =60,red](1) edge (3);
		\draw[line width=0.9mm, densely dashdotted,bend right =50,red](5) edge (3);
		\draw[line width=0.9mm,  densely dashdotted,bend right =30,red](5) edge (23);
		\draw[line width=0.9mm,  densely dashdotted,bend left =20,red](21) edge (23);
		\draw[line width=0.9mm ,bend left =20,red](7) edge (9);
		\draw[line width=0.9mm, densely dashdotted, bend left =40,red](19) edge (21);
		\draw[line width=0.9mm ,bend left =30,red,densely dashdotted](9) edge (19);
		\draw[line width=0.9mm ,bend left =30,red,densely dashdotted](9) edge (17);
		\draw[line width=0.9mm,  densely dashdotted ,bend left =40,red](9) edge (15);
		\draw[line width=0.9mm ,bend left =40,red](9) edge (11);
		\draw[line width=0.9mm,  densely dashdotted,bend left=30,red](13) edge (15);
		
		\draw[line width=0.7mm ,bend right=-10,blue, loosely dashed](5) edge (13);
		\draw[line width=0.7mm ,bend left=-10,orange, loosely dotted](3) edge (21);
		\draw[line width=0.7mm ,bend left=20,orange, loosely dotted](13) edge (17);

        \foreach \X in {1,7,11}
		{
		\tkzDrawPoints[fill =red,size=4,color=red](\X);
		};
		\foreach \X in {9,13}
		{
		\tkzDrawPoints[rectangle,fill =red,size=6,color=red](\X);
		};
		
		\foreach \X in {3,15,17}
		{
		\tkzDrawPoints[rectangle,size=6,color=dark-green,thick,fill=white](\X);
		};
		\foreach \X in {5,19,21}
		{
		\tkzDrawPoints[size=6,orange,osq](\X);
		};
		\foreach \X in {23}
		{
		\tkzDrawPoints[size=6,darkpink,line width=0.5mm,cross out, draw](\X);
		};

		\tkzDefPoint(-1.4,0.8){gamma};
		\tkzLabelPoint[blue](gamma){\Large $\delta$}
		\tkzDefPoint(2.8,1.7){gamma};
		\tkzLabelPoint[dark-green](gamma){\Large $v_1$}
        \tkzDefPoint(-1,-1.6){gamma};
		\tkzLabelPoint[dark-green](gamma){\Large $v_2$}
		\tkzDefPoint(-4.3,-0.5){s};
		\tkzLabelPoint[red](s){\Large $\pmb{s}_{(\delta)}(v_2)$}
		\tkzDefPoint(3,-1.1){t};
		\tkzLabelPoint[orange](t){\Large $\pmb{s}_{(\delta)}(v_1)$}
    \end{tikzpicture}
\caption{\label{fig:exuppersources} Example of calculations of $s_{(\delta)}(v)$ for some $\delta \in \Accord'$. The two distinct cases appear in this example.}
\end{figure}

\begin{definition} \label{def:sourceproxy}
Let  $(\delta, \eta) \in \Accord'^2$ be a pair of noncrossing accordions, such that $\NP(\delta)\cap \NP(\eta)\neq \varnothing$, and $\col_\delta$ and $\col_\eta$ match. Without loss of generality, consider that $\delta$ is above $\eta$. Let $v \in \NP(\eta)_0 \cap \NP(\delta)_0^{{\gsquare}}$, and $w \in \NP(\delta)_0 \cap \NP(\eta)_0^{{\rsquare}}$. The \new{upper source of} $(\delta,\eta)$ \new{associated to} $v$, denoted by $\s_{(\delta,\eta)}(v)$, is defined as $\s_{(\delta)}(v)$. We define the \new{upper target of} $(\delta,\eta)$ \new{associated to} $w$, denoted by $\t_{(\delta,\eta)}(w)$, to be $\t_{(\eta)}(w)$.
\end{definition}

The following lemmas show that $\s_{(\delta,\eta)}(v)$ and $\t_{(\delta,\eta)}(w)$ only depend on the pair $(\delta,\eta)$.

\begin{lemma} \label{lem:placeofcommonprojectivevertices}
Let $(\delta, \eta) \in \Accord'^2$ be a pair of noncrossing accordions such that $\NP(\delta) \cap \NP(\eta) \neq \varnothing$ with strongly matching colorations. Assume that $\delta$ is above $\eta$ if there exists $\rho'$ that makes colorations matchable and either $\rho'$ crosses $\delta$ or $\eta$, or whose endpoints are endpoints of $\delta$ or $\eta$. And assume that $\delta$ is on the left of $\eta$ in any other configuration of $\rho'$. Then the following assertions hold:
\begin{enumerate}[label=$(\roman*)$, itemsep =1mm]
    \item exactly one of the following holds:
    \begin{enumerate}[label=$(\alph*)$, itemsep =1mm]
        \item For all $\rho \in \NP(\delta) \cap \NP(\eta)$ with an endpoint in  $\NP(\delta)_0^{{\gsquare}}$, we have that $\Ext^1(\MM(\delta), \MM(\rho)) \neq 0$ or the module $\MM(\rho)$ appears as summand of the minimal projective resolution of $\MM(\delta)$; or,
        \item For all $\rho \in \NP(\delta) \cap \NP(\eta)$ with an endpoint in $\NP(\delta)_0^{{\gsquare}}$, we have for some $i>1$ $\Ext^i(\MM(\delta), \MM(\rho)) \neq 0$; and,
    \end{enumerate}
    \item exactly one of the following holds:
    \begin{enumerate}[label=$(\alph*)$, itemsep =1mm]
        \item For all $\rho \in \NP(\delta) \cap \NP(\eta)$ with an endpoint $e$ such that $e \in \NP(\eta)_0^{{\rsquare}}$,  $\Ext^1(\MM(\delta), \MM(\rho)) \neq 0$ or the module $\MM(\rho)$ appears as summand of the minimal projective resolution of $\MM(\delta)$; or,
        \item For all $\rho \in \NP(\delta) \cap \NP(\eta)$ with an endpoint $e$ such that $e \in \NP(\eta)_0^{{\rsquare}}$, we have, for some $i>1$, $\Ext^i(\MM(\eta), \MM(\rho)) \neq 0$; and,
    \end{enumerate}
\end{enumerate}
\end{lemma}
\begin{proof}  Assertion $(ii)$ is dual to $(i)$.
Let $(\delta, \eta) \in \Accord'^2$ be a pair of noncrossing accordions such that $\NP(\delta) \cap \NP(\eta) \neq \varnothing$ with 
strongly matching colorations. Assume first that there exists $\rho'$ that makes the colorations matchable and either $\rho'$ crosses $\delta$ or $\eta$, or whose endpoints are endpoints of $\delta$ or $\eta$, and without loss of generality that $\delta$ is above $\eta$.
Let $\rho \in \NP(\delta) \cap \NP(\eta)$ with an endpoint $e$ such that $e \in \NP(\delta)_0^{{\gsquare}}$. Due to the construction of the set $\NP(\delta)$ the module $M(\rho)$ is either a summand of the minimal projective resolution of $M(\delta)$ or we have for some $i\geq 1$ $\Ext^i(\MM(\delta), \MM(\rho)) \neq 0$. We now prove that all projectives fall into one of the two cases. Using \cref{lem:allexclusivityhyp}, the common part of neighboring projectives is not allowed to cross a large projective cell, as they do not have a common endpoint. The two previous sets of projectives appear on each side of a large cell, indicating that the two aforementioned cases are mutually exclusive.
Now consider the last case where $\delta$ and $\eta$ are strongly matchable and share a common endpoint with $\rho$. The curve  $\rho$ is the unique element of $\NP(\delta) \cap \NP(\eta)$. It follows from the construction of $\NP(\delta)$ that $\rho$ lies in exactly one of the two cases.
\end{proof}

\begin{lemma} \label{lem:sourceproxyindep}
Let $(\delta, \eta) \in \Accord'^2$ be a pair of noncrossing accordions such that $\NP(\delta) \cap \NP(\eta) \neq \varnothing$ with strongly matching colorations. Assume that $\delta$ is above $\eta$ if there exists $\rho$ that makes the colorations matchable, and such that $\delta$ and $\rho$ do not share a common endpoint. Otherwise, consider $\delta$ on the left of $\eta$. The following assertions hold:
\begin{enumerate}[label=$(\roman*)$, itemsep =1mm]
    \item for all $v \in \NP(\eta)_0 \cap \NP(\delta)_0^{{\gsquare}}$, if $v$ is an endpoint of a projective curve that has a higher extension with $\delta$, we have $\s_{(\delta,\eta)}(v) =  w_R(\delta)$, else $\s_{(\delta,\eta)}(v) =  s(\delta)$; and,
    \item for all $w \in \NP(\delta)_0 \cap \NP(\eta)_0^{{\rsquare}}$, if $v$ is an endpoint of a projective curve that has a higher extension with $\eta$, we have $\t_{(\delta,\eta)}(w) = w_L(\eta)$, otherwise $\t_{(\delta,\eta)}(w) = t(\eta)$.
\end{enumerate}
\end{lemma}

\begin{proof}
The assertions are equivalent by duality. As $\delta$ and $\eta$ are two noncrossing accordions, for any $v \in \NP(\eta)_0\cap\NP(\delta)_0^{{\gsquare}}$, there exists an accordion $\nu \in \Prj(\Delta^{\gpoint})$  such that $v$ is an endpoint of $\nu$. If $\nu$ is such that $\Ext^1(\MM(\delta),\MM(\nu))\neq 0$,  using \cref{rem:useful_up_s}, we get that $\s_{(\delta,\eta)}(v) = s(\delta)$. If $\nu$ has higher extensions with $\delta$ then $\s_{(\delta,\eta)}(v) = w_R(\delta)$ as an admissible curve cannot go further due to the definition of an accordion.
\end{proof}

Therefore, we simply call \new{upper source of $(\delta,\eta)$} the one associated to any $v \in \NP(\eta)_0 \cap \NP(\delta)_0^{{\gsquare}}$, and we denote it by \new{$\s_{(\delta, \eta)}$}.  We also call $\t_{(\delta, \eta)}$ the \new{upper target for $(\delta, \eta)$} similarly.

\begin{definition} \label{def:CoZcompl}
A pair $(\delta, \eta) \in \Accord'^2$ \new{admits a $\CoZ$-completion} if the following assertions hold:
\begin{enumerate}[label=$\bullet$,itemsep=1mm]
\item $\delta$ and $\eta$ are not crossing; and,
\item $\col_\delta$ and $\col_\eta$ are strongly matching.
\end{enumerate}

In such a case, up to exchanging $\delta$ and $\eta$, we may assume $\delta$ is above $\eta$. We define the $\rpoint$-arc $\xi$ by $s(\xi)=\s_{(\delta,\eta)}$ and $t(\xi)=\t_{(\delta,\eta)}$. We will seet that $\xi$ is an accordion, 
and call it the \new{$\CoZ$-completion} of $(\delta,\eta)$.
\end{definition}

\begin{lemma} \label{lem:welldefinedCoZ}
Let $(\delta, \eta) \in \Accord'^2$. Assume that $(\delta, \eta)$ admits a $\CoZ$-completion $\xi$. Then $\xi$ is an accordion and $\xi \in \opResAc'(\delta,\eta)$.
\end{lemma}

\begin{proof}
The curve $\xi$ is built through successive extensions. The accordions $\delta$ and $\eta$ are strongly matchable. Thus $\rho$ admits an extension either by $\delta$ or by a non-projective summand of its highest syzygy. The summands are precisely described by \cref{rem:useful_up_s}. Thus the curve $\psi$ such that $s(\psi)=\s_{(\delta,\eta)}$ and $t(\psi)=v$ where $v$ is the green endpoint of $\rho$ is an accordion and botain through an extension of $\delta$ and  $\rho$. Considering $\psi$ and $\eta$ we obtain $\xi$ in the same way. 

Then we can check that $\xi \in \Accord'$. This follows from the fact that $\xi$ must cross the projective accordions that are crossed by $\delta$ after $\s_{(\delta,\eta)}$  until i reaches $\NP(\delta)\cap\NP(\eta)$, and then $\xi$ crosses projective accordions crossed by $\eta$.

Consider $\vartheta_1 \in \Accord$ such that $s(\vartheta_1)=\s_{(\delta,\eta)}$ and $t(\vartheta_1)=t(\delta)$. Similarly, consider $\vartheta_2 \in \Accord$ such that $s(\vartheta_2)=s(\eta)$ and $t(\vartheta_2)=\t_{(\delta,\eta)}$. By definition of  $\opResAc'(\delta)$, for any $v\in \NP(\eta)_0\cap \NP(\delta)_0^{{\gsquare}}$, there exists an accordion $\iota$ such that $s(\iota) = s(\vartheta_1)$, and $t(\iota) = v$. We also have $\iota \in \opResAc(\vartheta_1)$. In addition, $\iota$ crosses $\vartheta_2$. Therefore $\xi \in \opResAc'(\iota, \vartheta_2)$, and we get that $\xi \in \opResAc'(\delta,\eta)$. 
\end{proof}

\begin{definition} \label{def:CoZ}
For any $(\delta, \eta) \in \Accord'^2$, we define the \new{$\CoZ$-move} of $(\delta,\eta)$ as follows:
\[\CoZ(\delta,\eta) = \begin{cases}
\{\delta, \eta, \xi\} & \text{if } (\delta,\eta) \text{ admits a }\CoZ\text{-completion } \xi; \\
\{\delta,\eta\} & \text{otherwise.}
\end{cases}\]
\end{definition}

\begin{ex}
   In \cref{fig:Co-Z}, we represent the upper source and the upper target of a given pair $(\delta, \eta) \in \mathscr{A}$ in the  $\gpoint$-dissected marked surface $\Surf(Q,R)$ seen in \cref{fig:ExExclcusivity}. Then we construct the accordion $\xi$ which is the $\CoZ$-completion of $(\delta,\eta)$.
\begin{figure}[ht!]
\centering 
    \begin{tikzpicture}[mydot/.style={
					circle,
					thick,
					fill=white,
					draw,
					outer sep=0.5pt,
					inner sep=1pt
				}, scale = 1]
		\tikzset{
		osq/.style={
        rectangle,
        thick,
        fill=white,
        append after command={
            node [
                fit=(\tikzlastnode),
                orange,
                line width=0.3mm,
                inner sep=-\pgflinewidth,
                cross out,
                draw
            ] {}}}}
		\draw[line width=0.7mm,black] (0,0) ellipse (4cm and 1.5cm);
		\foreach \X in {0,1,...,23}
		{
		\tkzDefPoint(4*cos(pi/12*\X),1.5*sin(pi/12*\X)){\X};
		};
		
		\draw[line width=0.9mm ,bend left =60,red](1) edge (3);
		\draw[line width=0.9mm ,bend right =60,red](5) edge (3);
		\draw[line width=0.9mm ,bend right =30,red](5) edge (23);
		\draw[line width=0.9mm ,bend left =20,red](21) edge (23);
		\draw[line width=0.9mm ,bend left =20,red](7) edge (9);
		\draw[line width=0.9mm ,bend left =40,red](19) edge (21);
		\draw[line width=0.9mm ,bend left =30,mypurple,densely dashdotted](9) edge (19);
		\draw[line width=0.9mm ,bend left =40,mypurple,densely dashdotted](9) edge (17);
		\draw[line width=0.9mm ,bend left =40,red](9) edge (15);
		\draw[line width=0.9mm ,bend left =40,red](9) edge (11);
		\draw[line width=0.9mm ,bend left=30,red](13) edge (15);
		
		\draw[line width=0.7mm ,bend right=10,blue, loosely dashed](5) edge (17);
		\draw[line width=0.7mm ,bend left=10,blue, loosely dashed](7) edge (13);
		\draw[line width=0.7mm ,bend left=10,orange, densely dashdotted](5) edge (13);
		
		\filldraw [fill=mypurple,opacity=0.1] (9) to [bend left=30] (19) to [bend left=10] (17) to [bend right=40] cycle ;

		\foreach \X in {1,3,...,23}
		{
		\tkzDrawPoints[fill =red,size=4,color=red](\X);
		};

		\tkzDefPoint(-2.1,0.6){gamma};
		\tkzLabelPoint[blue](gamma){\Large $\delta$}
		\tkzDefPoint(-.8,0.1){gamma};
		\tkzLabelPoint[blue](gamma){\Large $\eta$}
		\tkzDefPoint(-0.2,1.5){gamma};
		\tkzLabelPoint[orange](gamma){\Large $\xi$}
		\tkzDefPoint(-4.2,-0.5){s};
		\tkzLabelPoint[red](s){\Large $\pmb{s}_{(\delta,\eta)}$}
		\tkzDefPoint(1,2.1){t};
		\tkzLabelPoint[red](t){\Large $\pmb{t}_{(\delta,\eta)}$}
    \end{tikzpicture}
\caption{\label{fig:Co-Z} Construction of $\xi \in \Accord$ the $\CoZ$-completion of $(\delta, \eta)$.}
\end{figure}
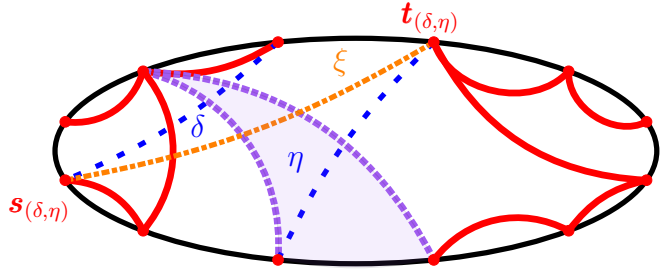
\end{ex}

\begin{prop}
\label{prop:Co-Z}
Let $(\delta, \eta) \in \Accord'^2$ be such that they admit a $\CoZ$-completion $\xi$. If $(\delta,\eta)$ satisfies \ref{4E}, then $\{\delta, \eta,\xi\} \in \Anti(\Accord', \Accordleq)$.
\end{prop}

\begin{proof}
As $(\delta,\eta)$ does not satisfy \ref{0E}, we only have to check that $\xi$ is not comparable to $\delta$ and $\eta$. Up to symmetry, proving one of the two is enough. We prove that $\delta$ and $\xi$ are not comparable using \cref{prop:0E}. We have that $t(\xi)=\t_{(\delta,\eta)}$ is exclusive in the pair $(\delta,\xi)$ as it was exclusive in the pair $(\delta,\eta)$. Then $t(\delta)$ is exclusive in the pair $(\delta,\eta)$, and thus exclusive in the pair $(\delta,\xi)$ as the neighboring projective accordions of $\xi$ are the ones of $\eta$ after the intersection. Therefore $(\delta, \xi)$ does not satisfy \ref{0E}, and we are done.
\end{proof}

\begin{remark} \label{rem:co-ZandEAccord}
Indeed, we prove that, following the \cref{algo:Rescatmotivation}, if we begin with an antichain $\mathfrak{B}$ such that $\delta, \eta \in \mathfrak{B}_0$, then $\xi$ should be in $\langle \mathfrak{B}_i \rangle$ for some $i \geqslant 1$.
\end{remark}

\subsection{KE'-move}
\label{ss:KE'}
The following move combines a $\CoZ$-move and a $\KE$-move. It could happen that $(\delta,\eta)$ admits a $\CoZ$-move $\xi$ contained in $\opResAc'(\delta)$, but $\KE(\xi,\eta)$ is not contained in $\opResAc'(\delta) \cup \opResAc'(\eta)$. This can happen when there are non-trivial extensions between syzygies of $\MM(\delta)$ and $\MM(\eta)$, for instance. This third move takes this phenomenon into account.

\begin{definition} \label{def:KE'compl}
Let $(\delta, \eta) \in \Accord'^2$ be such that $(\delta,\eta)$ admits a $\CoZ$-completion $\xi \in \Accord$. We say that $(\delta, \eta)$ \new{admits a $\KE'$-completion} whenever both:
\begin{enumerate}[label=$\bullet$,itemsep=1mm]
    \item $(\delta,\eta)$ satisfies \ref{3E}: $s(\delta)$ is the unique non exclusive vertex in $(\delta,\eta)$;
    \item $s(\xi) = s(\delta)$ and $t(\xi) = t(\eta)$; and,
    \item the $\rpoint$-arc $\kappa$ such that $s(\kappa) = t(\eta)$ and $t(\kappa) = t(\delta)$ is in $\Accord$.
\end{enumerate}
 In such a case, we define the \new{$\KE'$-completion} of $(\delta,\eta)$ to be the accordion $\kappa$. 
\end{definition}

\begin{definition} \label{def:KE'}
For any $(\delta, \eta) \in \Accord'^2$, we define the \new{$\KE'$-move} of $(\delta,\eta)$ as follows:
\[\KE'(\delta,\eta) = \begin{cases}
\{\delta,\eta,\kappa\}& \text{if } (\delta, \eta) \text{ admits a }\KE'\text{-completion } \kappa;\\
\{\delta,\eta\} & \text{otherwise.}
\end{cases}\]
\end{definition}

\begin{ex}
We represent the general configuration of a pair $(\delta,\eta) \in \Accord'^2$ such that $(\delta,\eta)$ admits a $\KE'$-completion $\kappa$ in \cref{fig:3EforKE}.
\end{ex}

\begin{remark} \label{lem:KE'andKE}
Let $(\delta,\eta) \in \Accord'^2$ be such that $(\delta, \eta)$ admits a $\CoZ$-completion $\xi$. If $(\delta, \eta)$ admits a $\KE'$-completion $\kappa$, then $\kappa$ is the $\KE$-completion of $\{\delta,\xi\}$. 
\end{remark}

\begin{prop}
\label{prop:3EforKE}
Let $(\delta, \eta) \in \Accord'^2$ be such that $(\delta,\eta)$ admits a $\CoZ$-completion $\xi$. If, moreover, $(\delta,\eta)$ admits a $\KE'$-completion, then $\mathfrak{E}_{\opResAc'(\delta,\eta)} = \KE'(\delta,\eta)$
\end{prop}
\begin{proof}
See \cref{fig:3EforKE} for an illustration of the situation. We prove that \[\opResAc(\delta, \eta) = \bigcup_{\nu \in \KE(\delta, \xi) \cup \lbrace\eta\rbrace} \opResAc(\nu).\] It will induce the desired result.
\begin{figure}[!ht]
        \centering
       \begin{tikzpicture}[mydot/.style={
					circle,
					thick,
					fill=white,
					draw,
					outer sep=0.5pt,
					inner sep=1pt
				}, scale = 1]
		\tikzset{
		osq/.style={
        rectangle,
        thick,
        fill=white,
        append after command={
            node [
                fit=(\tikzlastnode),
                orange,
                line width=0.3mm,
                inner sep=-\pgflinewidth,
                cross out,
                draw
            ] {}}}}
		\draw [line width=0.7mm,domain=50:130] plot ({4*cos(\x)}, {1.5*sin(\x)});
        \draw [line width=0.7mm,domain=230:330] plot ({4*cos(\x)}, {1.5*sin(\x)});
		\foreach \X in {0,1}
		{
		\tkzDefPoint(4*cos(pi/5*\X +pi/3),1.5*sin(pi/5*\X + pi/3)){\X};
		};
		\foreach \X in {2,3}
		{
		\tkzDefPoint(4*cos(pi/5*(\X-2) +4*pi/3),1.5*sin(pi/5*(\X-2) + 4*pi/3)){\X};
		};
		
		\draw[line width=0.7mm ,bend right=15,orange, dashed](-2,1.3) edge (1.3,-1.42);
		
		\draw[line width=0.7mm ,bend right=30,blue, loosely dotted](1) edge (3,-1);
		
		\draw[line width=0.7mm ,bend right=30,darkgreen, dash pattern={on 5pt off 2pt on 1pt off 2pt}](1) edge (1.3,-1.42);
		
		\draw[line width=0.7mm ,bend right=-10,purple, dash pattern={on 3pt off 2pt on 2pt off 2pt}](1.3,-1.42) edge (3,-1);
		
		\draw [line width=0.9mm, mypurple,dash pattern={on 10pt off 2pt on 5pt off 2pt}, bend right=20,opacity=0.5] (1) edge (2);
		\draw [line width=0.9mm,mypurple,domain=60:96, dash pattern={on 10pt off 2pt on 5pt off 2pt},opacity=0.5] plot ({4*cos(\x)}, {1.5*sin(\x)});
		\draw [line width=0.9mm, mypurple,dash pattern={on 10pt off 2pt on 5pt off 2pt}, bend right=20,opacity=0.5] (0) edge (3);
		\draw [line width=0.9mm,mypurple,domain=240:276, dash pattern={on 10pt off 2pt on 5pt off 2pt},opacity=0.5] plot ({4*cos(\x)}, {1.5*sin(\x)});
		
        \filldraw [fill=mypurple,opacity=0.1] (1) to [bend right=20] (2) to [bend right=10] (3) to [bend left=20] (0) to [bend right=10] cycle ;
        
		\foreach \X in {0,...,3}
		{
		\tkzDrawPoints[fill =red,size=4,color=red](\X);
		};
		
		\tkzDrawPoint[fill =red,size=4,color=red](-2,1.3);
		\tkzDrawPoint[fill =red,size=4,color=red](1.3,-1.42);
		\tkzDrawPoint[fill =red,size=4,color=red](3,-1);

		\tkzDefPoint(-2,1.1){gammaM};
		\tkzLabelPoint[orange](gammaM){\Large $\eta$}
		\tkzDefPoint(1.5,0){gammaP};
		\tkzLabelPoint[blue](gammaP){\Large $\delta$}
		\tkzDefPoint(-0.6,0.8){rho};
		\tkzLabelPoint[darkgreen](rho){\Large $\xi$}
		\tkzDefPoint(1.7,-0.8){ups};
		\tkzLabelPoint[purple](ups){\Large $\kappa$}
    \end{tikzpicture}
        \caption{\label{fig:3EforKE} General configuration of the $\KE'$-completion $\kappa$ of $(\delta,\eta)$. The shaded area contains the projective accordions in $\NP(\delta) \cap \NP(\eta)$. The accordion $\xi$ is the $\CoZ$-completion of $(\delta,\eta)$. }
\end{figure}
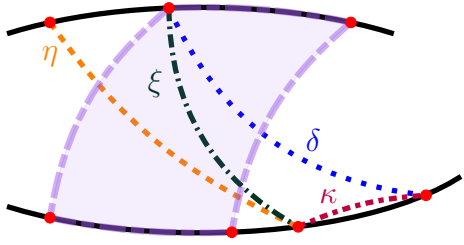
Note that $\xi \Accordleq \eta$ by \cref{prop:0E}. Then, if it exists, the $\KE$-completion of $(\xi,\eta)$ is in $\opResAc(\eta)$. By \cref{prop:folkloreresaccord} $(iv)$, we only have to check that $\bigcup_{\nu \in \KE(\delta, \xi) \cup \lbrace\eta\rbrace} \opResAc(\nu)$ is closed under taking extensions. It is sufficient to note that $\mathfrak{E}_{\opResAc'(\xi,\delta)} = \{\xi,\delta,\kappa\} = \KE(\delta,\xi)$ by \cref{cor:KEmove}.
\end{proof}

\subsection{V-move}
\label{ss:V-move}
This fourth move encodes (higher) extensions between $\MM(\delta)$ and $\MM(\eta)$. It extracts a non-projective accordion $\varsigma \in \opResAc'(\delta,\eta)$ such that $\delta \Accordlneq \varsigma$ or $\eta \Accordlneq \varsigma$.

\begin{definition} \label{def:Vcompl}
Let $(\delta, \eta) \in \Accord'^2$. We say that $(\delta, \eta)$ \new{admits a $\V$-completion} if the following hold:
\begin{enumerate}[label=$\bullet$,itemsep=1mm]
    \item $\col_\delta$ and $\col_\eta$ strongly match;
    \item $\delta$ and $\eta$ are crossing or share a common endpoint;
    \item $\{\delta, \eta\} \in \Anti(\Accord', \Accordleq)$; and,
    \item $(\delta, \eta)$ satisfies \ref{3E}: $s(\eta)$ is not exclusive to $\eta$ in $(\delta, \eta)$ . 
\end{enumerate}
In such a case, we define the \new{$\V$-completion} of $(\delta, \eta)$ to be the accordion $\varsigma$ such that $s(\varsigma)=\s_{(\delta,\eta)}$ and $t(\varsigma)=t(\eta)$
\end{definition}

\begin{definition} \label{def:V}
For any $(\delta, \eta) \in \Accord'^2$, we define the \new{$\V$-move} of $(\delta,\eta)$ as follows:
\[\V(\delta,\eta) = \begin{cases}
\{\delta, \varsigma\} & \text{if } (\delta,\eta) \text{ admits a }\V\text{-completion } \varsigma; \\
\{\delta,\eta\} & \text{otherwise.}
\end{cases}\]
\end{definition}
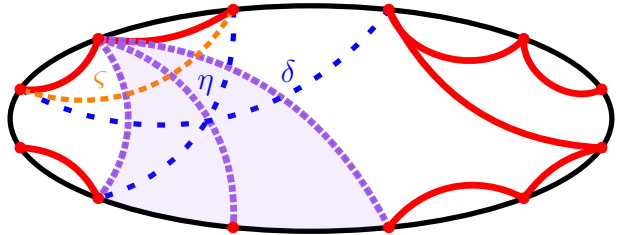
\begin{figure}[!ht]
\centering 
    \begin{tikzpicture}[mydot/.style={
					circle,
					thick,
					fill=white,
					draw,
					outer sep=0.5pt,
					inner sep=1pt
				}, scale = 1]
		\tikzset{
		osq/.style={
        rectangle,
        thick,
        fill=white,
        append after command={
            node [
                fit=(\tikzlastnode),
                orange,
                line width=0.3mm,
                inner sep=-\pgflinewidth,
                cross out,
                draw
            ] {}}}}
		\draw[line width=0.7mm,black] (0,0) ellipse (4cm and 1.5cm);
		\foreach \X in {0,1,...,23}
		{
		\tkzDefPoint(4*cos(pi/12*\X),1.5*sin(pi/12*\X)){\X};
		};
        
		\draw[line width=0.9mm ,bend left =60,red](1) edge (3);
		\draw[line width=0.9mm ,bend right =60,red](5) edge (3);
		\draw[line width=0.9mm ,bend right =30,red](5) edge (23);
		\draw[line width=0.9mm ,bend left =20,red](21) edge (23);
		\draw[line width=0.9mm ,bend left =20,red](7) edge (9);
		\draw[line width=0.9mm ,bend left =40,red](19) edge (21);
		\draw[line width=0.9mm ,bend left =30,mypurple,densely dashdotted](9) edge (19);
		\draw[line width=0.9mm ,bend left =40,mypurple,densely dashdotted](9) edge (17);
		\draw[line width=0.9mm ,bend left =40,mypurple, densely dashdotted](9) edge (15);
		\draw[line width=0.9mm ,bend left =40,red](9) edge (11);
		\draw[line width=0.9mm ,bend left=30,red](13) edge (15);
		
		\draw[line width=0.7mm ,bend right=40,blue, loosely dashed](11) edge (5);
		\draw[line width=0.7mm ,bend left=40,blue, loosely dashed](7) edge (15);
		\draw[line width=0.7mm ,bend left=40,orange, densely dashed](7) edge (11);

		\filldraw [fill=mypurple,opacity=0.1] (9) to [bend left=40] (15) to [bend right=10] (19) to [bend right=30] cycle ;

		\foreach \X in {1,3,...,23}
		{
		\tkzDrawPoints[fill =red,size=4,color=red](\X);
		};

		\tkzDefPoint(-1.4,0.7){gamma};
		\tkzLabelPoint[blue](gamma){\Large $\eta$}
		\tkzDefPoint(-0.3,0.9){gamma};
		\tkzLabelPoint[blue](gamma){\Large $\delta$}
		\tkzDefPoint(-2.8,0.75){s};
		\tkzLabelPoint[orange](s){\Large $\varsigma$}
    \end{tikzpicture}
\caption{\label{fig:V} Construction of $\varsigma \in \Accord$, the $\V$-completion of $(\delta,\eta)$, in the case where $\delta$ and $\eta$ cross. }
\end{figure}

\begin{figure}[!ht]
\centering 
    \begin{tikzpicture}[mydot/.style={
					circle,
					thick,
					fill=white,
					draw,
					outer sep=0.5pt,
					inner sep=1pt
				}, scale = 1]
		\tikzset{
		osq/.style={
        rectangle,
        thick,
        fill=white,
        append after command={
            node [
                fit=(\tikzlastnode),
                orange,
                line width=0.3mm,
                inner sep=-\pgflinewidth,
                cross out,
                draw
            ] {}}}}
		\draw[line width=0.7mm,black] (0,0) ellipse (4cm and 1.5cm);
		\foreach \X in {0,1,...,23}
		{
		\tkzDefPoint(4*cos(pi/12*\X),1.5*sin(pi/12*\X)){\X};
		};
		
		\draw[line width=0.9mm ,bend left =25,red](7) edge (9);
		\draw[line width=0.9mm ,bend left =10,red](9) edge (15);
		\draw[line width=0.9mm ,bend left =25,red](15) edge (17);
		\draw[line width=0.9mm ,bend left =20,red](17) edge (19);
		\draw[line width=0.9mm ,bend left =15,mypurple,densely dashdotted](19) edge (5);
		\draw[line width=0.9mm ,bend left =15,mypurple,densely dashdotted](21) edge (5);
		\draw[line width=0.9mm ,bend left =15,red](21) edge (3);
		\draw[line width=0.9mm ,bend left =15,red](21) edge (1);
		\draw[line width=0.9mm ,bend left =15,red](23) edge (1);
		\draw[line width=0.9mm ,bend left =15,red](9) edge (11);
		\draw[line width=0.9mm ,bend left =15,red](11) edge (13);
		
		\draw[line width=0.7mm ,bend right=20,blue, loosely dashed](13) edge (5);
		\draw[line width=0.7mm ,bend right=20,blue, loosely dashed](5) edge (1);
		\draw[line width=0.7mm ,bend left=20,orange, densely dashed](17) edge (1);

		\filldraw [fill=mypurple,opacity=0.1] (5) to [bend right=15] (19) to [bend right=10] (21) to [bend left=15] cycle ;

		\foreach \X in {1,3,...,23}
		{
		\tkzDrawPoints[fill =red,size=4,color=red](\X);
		};

		\tkzDefPoint(-1.3,0.7){gamma};
		\tkzLabelPoint[blue](gamma){\Large $\delta$}
		\tkzDefPoint(1.6,1){gamma};
		\tkzLabelPoint[blue](gamma){\Large $\eta$}
		\tkzDefPoint(0.2,-.6){s};
		\tkzLabelPoint[orange](s){\Large $\varsigma$}
    \end{tikzpicture}
\caption{\label{fig:Vspecial} Construction of $\varsigma \in \Accord$, the $\V$-completion of $(\delta,\eta)$, in the case where $\delta$ and $\eta$ do not cross.}
\end{figure}
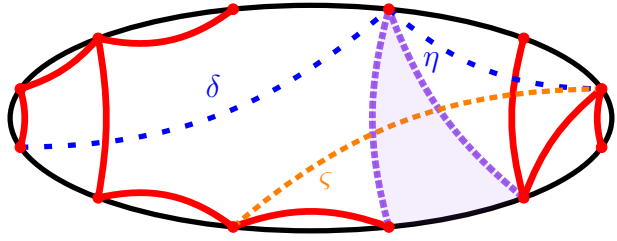

\begin{prop}
\label{V}
Let $(\delta, \eta) \in \Accord'^2$ be such that $(\delta, \eta)$ admits a $\V$-completion $\varsigma$. Then $\varsigma \in \opResAc'(\delta, \eta)$, and $\eta \Accordleq \varsigma$. 
\end{prop}

\begin{proof}
Let $\delta, \eta, \varsigma \in \Accord'$ be as assumed. If $\delta$ and $\eta$ are crossing, we note that $\varsigma \in \OvExt(\delta, \eta) \cup \OvExt(\eta,\delta)$, we hence have that $\varsigma \in \opResAc(\delta, \eta)$. If $\delta$ and $\eta$ share a common endpoint since the endpoints are in a \ref{3E} configuration it is $s(\eta)=t(\delta)$. In such a configuration $\delta$ and $\eta$ either admit an arrow extension or $\eta$ has an arrow extension with a non-projective summand of the highest syzygy of $\delta$. This curve is precisely $\varsigma$ thus $\varsigma \in \opResAc(\delta, \eta)$. Note that $\varsigma$ crosses accordions forming $\NP(\delta)\cap\NP(\eta)$ and thus it is non-projecive.

As $s(\eta)$ is not exclusive for $\delta$ in $(\delta, \eta)$, we have that $s(\eta)\in \NP(\delta)_0^{{\rsquare}}$. Therefore $s(\delta)\in \NP(\varsigma)_0^{{\rsquare}}$. By \cref{prop:0E}, we have that $\eta \Accordleq \varsigma$. 
\end{proof}

\begin{cor}
\label{cor:Vantichain}
    Let $(\delta, \eta) \in \Accord'^2$ be such that $(\delta, \eta)$ admits a $\V$-completion $\varsigma$. Then $\lbrace\delta,\varsigma\rbrace$ forms an antichain.
\end{cor}
\begin{proof}
    By \cref{V} that $\eta\Accordleq\varsigma$, and so, as $\{\delta,\eta\} \in \Anti(\Accord', \Accordleq)$, it is enough to show that $\delta\not\Accordleq\varsigma$. 
    
    Assume by absurdity that $\delta \Accordleq \varsigma$. By \cref{prop:0E}, $s(\delta) = s(\varsigma)$ and $t(\delta)$ are not exclusive to $\delta$ in $(\delta,\varsigma)$. As $\eta \Accordleq \varsigma$, this implies that $t(\delta)$ is not exclusive to $\delta$ in $(\delta,\eta)$, which is a contradiction to the fact that $(\delta,\eta)$ satisfies \ref{3E}, with $s(\eta)$ non-exclusive to $\eta$ in $(\delta,\eta)$.
\end{proof}

\begin{remark} \label{rem:Vnotgenerating}
    In the case where $(\delta,\eta)$ admits a $\V$-move $\varsigma$, it could happen that $\opResAc(\delta,\eta) \neq \opResAc(\varsigma) \cup \opResAc(\delta)$, meaning that $\mathfrak{E}_{\opResAc'(\delta,\eta)} \neq \{\delta,\eta\}$ (See \cref{fig:Vnotenough}).
\end{remark}

\begin{figure}[!ht]
        \centering
       \begin{tikzpicture}[mydot/.style={
					circle,
					thick,
					fill=white,
					draw,
					outer sep=0.5pt,
					inner sep=1pt
				}, scale = 1]
		\tikzset{
		osq/.style={
        rectangle,
        thick,
        fill=white,
        append after command={
            node [
                fit=(\tikzlastnode),
                orange,
                line width=0.3mm,
                inner sep=-\pgflinewidth,
                cross out,
                draw
            ] {}}}}
		\draw [line width=0.7mm] (0,0) ellipse (4cm and 1.5cm);
		\foreach \X in {0,...,45}
		{
		\tkzDefPoint(4*cos(pi/23*\X),1.5*sin(pi/23*\X)){\X};
		};
		
	    %\draw[line width=0.9mm ,bend right =30,red] (4) edge (44);
	    %\draw[line width=0.9mm ,bend left =30,red] (42) edge (44);
	    \draw[line width=0.9mm ,bend right =30,red] (39) edge (4);
	    \draw[line width=0.9mm ,bend right =20,red] (39) edge (8);

	    %\draw[line width=0.9mm ,bend left=30,red] (32) edge (34);
	    \draw[line width=0.9mm ,bend right =10,red] (24) edge (8);
	    \draw[line width=0.9mm ,bend left =20,red] (24) edge (30);
	    
	    \draw[line width=0.9mm ,bend left =10,red] (8) edge (14);
	    
	    \draw [line width=0.9mm, mypurple,dash pattern={on 10pt off 2pt on 5pt off 2pt}, bend right=15,opacity=0.8] (30) edge (4);
        \draw [line width=0.9mm, darkgreen,dash pattern={on 5pt off 2pt on 1pt off 2pt}, bend right=15,opacity=0.8] (30) edge (8);
		\draw [line width=0.9mm,blue, dashdotted, bend right = 30, opacity=0.8]  (14) edge (4);
		\draw [line width=0.9mm,orange, loosely dotted, bend right = 30, opacity=0.8]  (30) edge (14);
        
		\foreach \X in {4,8,14,24,30,39}
		{
		\tkzDrawPoints[fill =red,size=4,color=red](\X);
		};

		\tkzDefPoint(2.5,1){gammaM};
		\tkzLabelPoint[blue](gammaM){\Large $\delta$}
		\tkzDefPoint(1.7,-0.3){gammaP};
		\tkzLabelPoint[mypurple](gammaP){\Large $\varsigma$}
		\tkzDefPoint(-1.8,.7){rho};
		\tkzLabelPoint[orange](rho){\Large $\kappa$}
        \tkzDefPoint(-.3,.2){rho};
		\tkzLabelPoint[darkgreen](rho){\Large $\eta$}
    \end{tikzpicture}
        \caption{\label{fig:Vnotenough}In this situation, the pair $(\delta,\eta)$ admits a $\V$-completion $\varsigma$, and the pair $(\delta,\varsigma)$ admits a $\KE$-completion $\kappa$, and by \cref{V,cor:KEmove}, we have $\mathfrak{E}_{\opResAc'(\delta,\eta)} = \{\delta,\varsigma, \kappa\} \neq \V(\delta,\eta)$.}
    \end{figure}
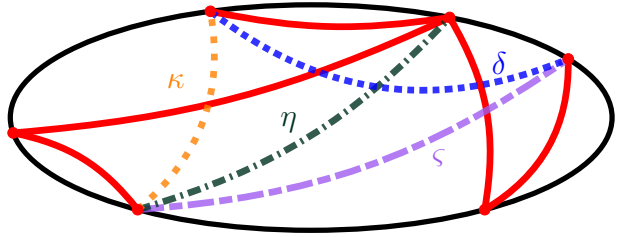

\subsection{X-move}
\label{sss:X-move}
This fifth move extracts the behavior of non-trivial overlap extensions between $\delta$ and $\eta$, where the accordions appearing as middle terms of said extensions are not in $\opResAc'(\delta) \cup \opResAc'(\eta)$.

\begin{definition} \label{def:Xcompl}
Let $(\delta, \eta) \in \Accord'^2$. We say that $(\delta,\eta)$ \new{admits an $\X$-completion} if the following assertions hold:
\begin{enumerate}[label=$\bullet$, itemsep=1mm]
    \item $\delta$ and $\eta$ are crossing; 
    \item $\{\delta, \eta\} \in \Anti(\Accord', \Accordleq)$; and,
    \item $(\delta, \eta)$ satisfies \ref{4E}.
\end{enumerate}
In such a case, we define the \new{$\X$-completion} of $(\delta, \eta)$ as the pair $(\varrho_1, \varrho_2 )\in \Accord^2$ such that $s(\varrho_1)=s(\delta)$, $t(\varrho_1)=t(\eta)$, $s(\varrho_2)=s(\eta)$ and $t(\varrho_2)=t(\delta)$.
\end{definition}

\begin{remark} \label{rem:crossingmatching}
    By \cref{lem:tomatchablecap}, if $\delta$ and $\eta$ are crossing, then $\col_\delta$ and $\col_\eta$ are strongly matchable.
\end{remark}

\begin{definition} \label{def:X}
For any $(\delta, \eta) \in \Accord'^2$, we define the \new{$\X$-move} of $\{\delta,\eta\}$ as follows:
\[\X(\delta,\eta) = \begin{cases}
\{\delta, \eta, \varrho_1, \varrho_2\} & \text{if } \{\delta,\eta\} \text{ admits a }\X\text{-completion } (\varrho_1,\varrho_2); \\
\{\delta,\eta\} & \text{otherwise.}
\end{cases}\]
\end{definition}

\begin{ex} \label{ex:X}
We give an example of a pair $(\delta, \eta) \in \Accord'^2$ which admits an $\X$-completion in \cref{fig:X}.
\end{ex}

\begin{figure}[!ht]
\centering 
    \begin{tikzpicture}[mydot/.style={
					circle,
					thick,
					fill=white,
					draw,
					outer sep=0.5pt,
					inner sep=1pt
				}, scale = 1]
		\tikzset{
		osq/.style={
        rectangle,
        thick,
        fill=white,
        append after command={
            node [
                fit=(\tikzlastnode),
                orange,
                line width=0.3mm,
                inner sep=-\pgflinewidth,
                cross out,
                draw
            ] {}}}}
		\draw[line width=0.7mm,black] (0,0) ellipse (4cm and 1.5cm);
		\foreach \X in {0,1,...,23}
		{
		\tkzDefPoint(4*cos(pi/12*\X),1.5*sin(pi/12*\X)){\X};
		};
		
		\draw[line width=0.9mm ,bend left =60,red](1) edge (3);
		\draw[line width=0.9mm ,bend right =60,red](5) edge (3);
		\draw[line width=0.9mm ,bend right =30,red](5) edge (23);
		\draw[line width=0.9mm ,bend left =20,red](21) edge (23);
		\draw[line width=0.9mm ,bend left =20,red](7) edge (9);
		\draw[line width=0.9mm ,bend left =40,red](19) edge (21);
		\draw[line width=0.9mm ,bend left =30,mypurple,densely dashdotted](9) edge (19);
		\draw[line width=0.9mm ,bend left =40,mypurple,densely dashdotted](9) edge (17);
		\draw[line width=0.9mm ,bend left =40,mypurple, densely dashdotted](9) edge (15);
		\draw[line width=0.9mm ,bend left =40,red](9) edge (11);
		\draw[line width=0.9mm ,bend left=30,red](13) edge (15);
		
		\draw[line width=0.7mm ,bend right=40,blue, loosely dashed](11) edge (5);
		\draw[line width=0.7mm ,bend left=40,blue, loosely dashed](7) edge (13);
		\draw[line width=0.7mm ,bend left=40,orange, densely dashed](7) edge (11);
		\draw[line width=0.7mm ,bend right=30,orange, densely dashed](13) edge (5);

		\filldraw [fill=mypurple,opacity=0.1] (9) to [bend left=40] (15) to [bend right=10] (19) to [bend right=30] cycle ;

		\foreach \X in {1,3,...,23}
		{
		\tkzDrawPoints[fill =red,size=4,color=red](\X);
		};

		\tkzDefPoint(-1.3,0.7){gamma};
		\tkzLabelPoint[blue](gamma){\Large $\delta$}
		\tkzDefPoint(-0.3,0.9){gamma};
		\tkzLabelPoint[blue](gamma){\Large $\eta$}
		\tkzDefPoint(0.7,1){gamma};
		\tkzLabelPoint[orange](gamma){\Large $\varrho_1$}
		\tkzDefPoint(-2.9,0.75){s};
		\tkzLabelPoint[orange](s){\Large $\varrho_2$}
    \end{tikzpicture}
\caption{\label{fig:X} Construction of  the $\X$-completion $(\varrho_1, \varrho_2) \in \Accord$ of $(\delta,\eta)$.}
\end{figure}
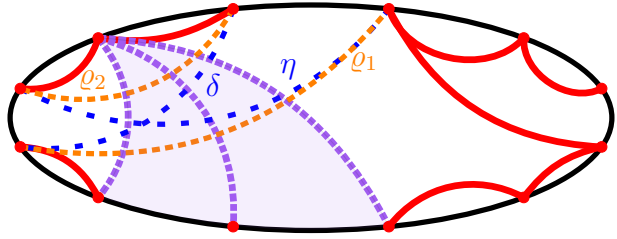

\begin{prop}
\label{prop:X}
Let $(\delta, \eta) \in \Accord'^2$ which admit an $\X$-completion $(\varrho_1,\varrho_2)$.
Then $\varrho_1, \varrho_2 \in \opResAc'(\delta, \eta)$ and $\{\delta, \eta, \varrho_1, \varrho_2\} \in \Anti(\Accord', \Accordleq)$.
\end{prop}

\begin{proof}
 By construction, we have $\varrho_1,\varrho_2 \in \OvExt(\delta,\eta) \cup \OvExt(\eta,\delta)$. Thus $\varrho_1,\varrho_2 \in \opResAc(\delta,\eta)$. As $\varrho_1$ and $\varrho_2$ contain all the projective accordions in $\NP(\delta) \cap \NP(\eta)$ within their own neighboring projectives, we have that $\varrho_1,\varrho_2$ are either crossing a projective or filling an angle of projective accordions. These accordions are therefore non-projective, as $(\pmb{\Sigma}, \mathcal {M})$ must be a disc without punctures, and projective accordions cannot cross themselves or form an inner cell. Thus $\varrho_1,\varrho_2\in\opResAc'(\delta,\eta)$.
 Moreover, none of the pairs $(\vartheta_1, \vartheta_2)$ of accordions in $\{\delta, \eta, \varrho_1,\varrho_2\}$,  satisfy \ref{0E}. Therefore $\{\delta, \eta, \varrho_1,\varrho_2\} \in \Anti(\Accord',\Accordleq)$.
\end{proof}

\subsection{W-move}
\label{ss:S'mov}
This sixth move comes from \cref{prop:specialcaseofmatching}. It will be relevant when a pair $(\delta,\eta) \in \Accord'^2$ is such that $\NP(\delta) \cap \NP(\eta) \neq \varnothing$ and $\col_\delta$ and $\col_\eta$ are weakly matchable and $\delta$ and $\eta$ are without a common endpoint. It allows us to reduce  to the strongly matchable case which is already known. Recall that all the configurations of weak matchability are listed in \cref{prop:specialcaseofmatching}.
\begin{prop}\label{prop:rho+approxwmatch}
 Let $(\delta,\eta) \in \Accord'^2$ be such that $\col_\delta$ and $\col_\eta$ are weakly matchable with no common endpoint. Keeping the notations of \cref{prop:specialcaseofmatching}. The curves $\delta$ and $\eta$ admit an $\Smov(\rho_+)$-approximation (\cref{def:Sapprox}).
\end{prop}
\begin{proof}
 Let $(\delta,\eta) \in \Accord'^2$ be such that $\col_\delta$ and $\col_\eta$ are weakly matchable with no common endpoints. By definition, we have that $\{\delta,\eta\} \in \Anti(\Accord', \Accordleq)$. All the configurations are listed in \cref{prop:specialcaseofmatching}. Moreover, by denoting $v$ the common endpoint of $\rho_+$ and $\rho_-$, we have $ \col_\eta(v) \in \{{\gsquare},{\rsquare}\}$. This implies the claimed result.
\end{proof}

\begin{lemma} \label{lem:onlyonerho+isenough}
    Let $(\delta, \eta) \in \Accord'^2$ be such that $\col_\delta$ and $\col_\eta$ are weakly matchable. Assume that there exists $\rho \in \NP(\delta) \setminus \NP(\eta)$ that makes the colorations matchable. Assume there are exactly two pairs $(\rho_+^{(1)}, \rho_-^{(1)})$ and $(\rho_+^{(2)}, \rho_-^{(2)})$ that satisfy \cref{prop:specialcaseofmatching}. Then the $\Smov(\rho_+^{(1)})$-approximation and the $\Smov(\rho_+^{(2)})$-approxiation of $(\delta,\eta)$ are equal.
\end{lemma}

\begin{proof}
    It follows from \cref{def:Sapprox} and \cref{cor:numberofpairsrho+rho-} $(iii)$. Indeed, the curve that enables the computation of the $\Smov(\rho_+^{(1)})$-approximation is a summand of the highest nonprojective syzygy of $\delta$. It is the same curve that enables the computation of the $\Smov(\rho_+^{(2)})$-approximation.
\end{proof}

\begin{definition}\label{def:S'}
Let $(\delta,\eta) \in \Accord'^2$ with no common endpoint, be such that $\col_\delta$ and $\col_\eta$ are weakly matchable. Assume there exists $\rho \in \NP(\delta) \setminus \NP(\eta)$ that makes the colorations matchable. We define the \new{$\W$-move} as follows :
\[
\W(\delta,\eta)=\{ \delta,\eta \} \curlyvee \All(\xi_{\rho_+}^{(\delta)},\eta)
\]
where:
\begin{enumerate}[label=$\bullet$, itemsep=1mm]
    \item the projective accordion $\rho_+$ is given by a pair $(\rho_+, \rho_-)$ satisfying the configuration from \cref{prop:specialcaseofmatching}; and,
    \item for any $(\xi, \varsigma) \in \Accord'^2$, $\All(\xi,\varsigma)$ is the join of $\KE(\xi,\varsigma)$, $\CoZ(\xi,\varsigma)$, $\KE'(\xi,\varsigma)$, $\V(\xi,\varsigma)$ and $\X(\xi, \varsigma)$.
\end{enumerate}
\end{definition}

\begin{prop}\label{prop:weaklytostrongly}
    Let $(\delta,\eta) \in \Accord'^2$ without a common endpoint. Assume that $\col_\delta$ and $\col_\eta$ are weakly matchable, and that there exists $\rho \in \NP(\delta) \setminus \NP(\eta)$ which makes the colorations matchable. Then \[\opResAc(\delta,\eta) = \opResAc(\delta) \cup \opResAc(\eta) \cup  \opResAc(\xi_{\rho_+}^{(\delta)},\eta) ,\]
where $\rho_+$ is the projective accordion of a pair $(\rho_+, \rho_-)$ satisfying \cref{prop:specialcaseofmatching}.
\end{prop}
\begin{proof}
    By assumption, we have that $(\xi_{\rho_+}^{(\delta)}, \eta)$ is the $\Smov(\rho_+)$-approximation of $(\delta, \eta)$. Thus, \[\opResAc(\delta) \cup \opResAc(\eta) \cup  \opResAc(\xi_{\rho^+},\eta)\subseteq \opResAc(\delta,\eta). \] To prove the converse, we will show that $\opResAc(\delta) \cup \opResAc(\eta) \cup  \opResAc(\xi_{\rho_+^{(\delta)}},\eta)$ is a resolving set. First, this set is closed under syzygies by \cref{lem:IdealsclosedSyzygies}. To check the closure under extensions, it is sufficient to show that, for any pair $(\xi, \varsigma) \in \Accord'^2$, such that $\xi \Accordleq \delta$ and $\varsigma \Accordleq \eta$,  $\xi\in\opResAc(\delta)$ admit an arrow or an overlap extension with $\varsigma \in\opResAc(\eta)$ only if $\xi\Accordleq \xi_{\rho_+}^{(\delta)}$. 
    
    Let $\xi$ and $\varsigma$ be as stated. It follows that either:
    \begin{enumerate}[label=$\bullet$, itemsep=1mm]
        \item $\xi$ crosses $\varsigma$; or,
        \item $\xi$ and $\varsigma$ share a common endpoint.
    \end{enumerate}  
    
    The first case is easily deduced from \cref{prop:specialcaseofmatching} as curves crossing $\NP(\eta)$ are below of $\xi_{\rho_+}^{(\delta)}$ using the conventions of \cref{fig:ReducedWeakConfig}. Such an $\Xi$ is in $\opResAc(\xi_{\rho_+}^{(\delta)})$ by \cref{thm:res_clo_1}. 
    The second case appears in the rightmost case of \cref{fig:specialmatchreduction} through curves ending in $v,w$ or the next vertex. However such a $\xi$ is below the syzygy of $\delta$ starting at the same vertex and thus $\xi\Accordleq \xi_{\rho_+}^{(\delta)}$. 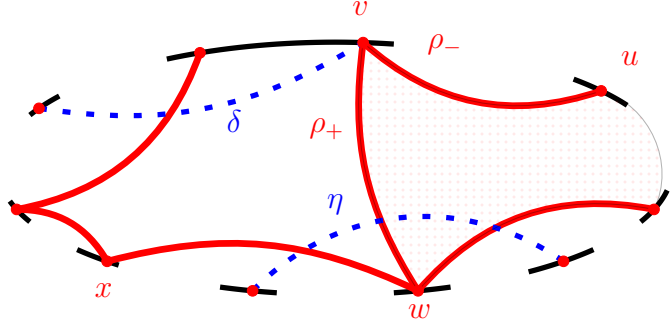
\begin{figure}[!ht]
\centering 
    \begin{tikzpicture}[scale=1.4,mydot/.style={
					circle,
					thick,
					fill=white,
					draw,
					outer sep=0.5pt,
					inner sep=1pt
				}, scale = .8]
		\tikzset{
		osq/.style={
        rectangle,
        thick,
        fill=white,
        append after command={
            node [
                fit=(\tikzlastnode),
                orange,
                line width=0.3mm,
                inner sep=-\pgflinewidth,
                cross out,
                draw
            ] {}}}}
        	\foreach \X in {0,1,...,37}
		{
		\tkzDefPoint(4*cos(pi/19*\X),1.5*sin(pi/19*\X)){\X};
		};
        \draw [line width=0.7mm,domain=30:45] plot ({4*cos(\x)}, {1.5*sin(\x)});
        \draw [line width=0.7mm,domain=80:120] plot ({4*cos(\x)}, {1.5*sin(\x)});
        \draw [line width=0.7mm,domain=145:155] plot ({4*cos(\x)}, {1.5*sin(\x)});
        \draw [line width=0.7mm,domain=195:205] plot ({4*cos(\x)}, {1.5*sin(\x)});
        \draw [line width=0.7mm,domain=220:230] plot ({4*cos(\x)}, {1.5*sin(\x)});
        \draw [line width=0.7mm,domain=250:260] plot ({4*cos(\x)}, {1.5*sin(\x)});
        \draw [line width=0.7mm,domain=280:290] plot ({4*cos(\x)}, {1.5*sin(\x)});
        \draw [line width=0.7mm,domain=305:320] plot ({4*cos(\x)}, {1.5*sin(\x)});
        \draw [line width=0.7mm,domain=335:350] plot ({4*cos(\x)}, {1.5*sin(\x)});

		\draw[line width=0.9mm ,bend left=30,red](12) edge (21);
		\draw[line width=0.9mm ,bend left =30,red](21) edge (24);
            \draw[line width=0.9mm ,bend left =20,red](24) edge (30);
            \draw[line width=0.9mm ,bend left =30,red](30) edge (36);
            \draw[line width=0.9mm ,bend left =20,red](30) edge (9);
            \draw[line width=0.9mm ,bend right =30,red](9) edge (4);
		
		\draw[line width=0.7mm,blue,bend right=20, loosely dashed](16) edge (9);
            \draw[line width=0.7mm,blue,bend left=40, loosely dashed](27) edge (33);

            \filldraw [pattern=dots, pattern color=red,opacity=0.3] (9) to [bend right=30] (4) to [bend left=50] (36) to [bend right=30] (30) to [bend left=20] cycle ;

		\foreach \X in {4,9,12,16,21,24,27,30,33,36}
		{
		\tkzDrawPoints[fill =red,size=4,color=red](\X);
		};

		\tkzDefPoint(-1.2,0.8){g};
		\tkzLabelPoint[blue](g){\Large $\delta$}
            \tkzDefPoint(0,-.2){h};
		\tkzLabelPoint[blue](h){\Large $\eta$}
            \tkzDefPoint(-.1,.7){i};
		\tkzLabelPoint[red](i){\Large $\rho_+$}
            \tkzDefPoint(1.3,1.7){j};
		\tkzLabelPoint[red](j){\Large $\rho_-$}
            \tkzDefPoint(3.5,1.5){j};
		\tkzLabelPoint[red](j){\Large $u$}
            \tkzDefPoint(0.3,2.1){l};
		\tkzLabelPoint[red](l){\Large $v$}
            \tkzDefPoint(1,-1.5){l};
		\tkzLabelPoint[red](l){\Large $w$}
            \tkzDefPoint(-2.75,-1.25){l};
		\tkzLabelPoint[red](l){\Large $x$}
    \end{tikzpicture}
\caption{\label{fig:Wmovesituation} Illustration of the situation.}
\end{figure} In \cref{fig:Wmovesituation}, the accordion $\xi$ must end at $v$, at $w$, or at the next vertex of the cell that we denote by $x$. If $\xi$ ends at $x$ it follows from \cref{thm:res_clo_1} that $\xi\Accordleq \xi_{\rho_+}^{(\delta)}$ . The only curve $\xi\in\opResAc(\delta)$ with endpoint $w$ is $\rho_+$ and thus the extension between $\rho_+$ and $\varsigma$ lies in $\opResAc(\eta)$. The last option is to consider $\xi$ admitting $v$ as an endpoint. If $v=s(\xi)$ the curve $\xi$ is such that $\xi\Accordleq \xi_{\rho_+}^{(\delta)}$ by \cref{thm:res_clo_1}. If $t(\xi)=v$, the completion of the triangle must be admissible and thus $s(\xi)=x$. It follows that $\xi=\xi_{\rho_+}^{(\delta)}$.
\end{proof}

\subsection{S-move}
\label{ss:S}

This seventh and last move comes from \cref{prop:Sapproxisenough}. It will be relevant to apply it to any pair $(\delta,\eta) \in \Accord'^2$ such that $\NP(\delta) \cap \NP(\eta) \neq \varnothing$ and $\col_\delta$ and $\col_\eta$ are not matchable.

\begin{definition} \label{def:S}
Let $(\delta,\eta) \in \Accord'^2$. Then we define the \new{$\Smov$-move} on $(\delta,\eta)$ as follows:
\[\Smov(\delta,\eta) = \{\delta,\eta\} \curlyvee \left( \bigcurlyvee_\rho \All^{+}(\xi_\rho,\varsigma_\rho) \right),\]
where:
\begin{enumerate}[label=$\bullet$,itemsep=1mm]
    \item the big join is indexed by $\rho \in \NP(\delta) \cup \NP(\eta)$ such that $(\delta,\eta)$ admits an $\Smov(\rho)$-approximation $(\xi_\rho,\varsigma_\rho)$; and,
    \item for any $(\xi,\varsigma) \in \Accord'^2$, $\All^{+}(\xi,\varsigma)$ is the join of $\KE(\xi,\varsigma)$, $\CoZ(\xi,\varsigma)$, $\KE'(\xi,\varsigma)$, $\V(\xi,\varsigma)$, $\X(\xi, \varsigma)$ and $\W (\xi,\varsigma)$.
\end{enumerate}
\end{definition}
\begin{prop} \label{prop:S}
Let $(\delta,\eta) \in \Accord'^2$ be such that $\NP(\delta) \cap \NP(\eta) \neq \varnothing$, and $\col_\delta$ and $\col_\eta$ are not matchable. If, for all $\nu \in \Smov(\delta,\eta)$, we have $\nu \Accordleq \delta$ or $\nu \Accordleq \eta$, then $\mathfrak{E}_{\opResAc'(\delta,\eta)} = \{\delta,\eta\}$.
\end{prop}

\begin{proof}
The result follows from \cref{prop:Sapproxisenough}.
\end{proof}

	\section{Resolving subcategories for gentle trees}
	\label{sec:TreePart2}
	\pagestyle{plain}

In this section, we state and prove our main result. We show that the subsets in $\Accord'$ that are closed under the moves introduced in \cref{sec:Moves} are in one-to-one correspondence with the resolving sets of \cref{def:resolv}. We then give a combinatorial algorithm, mixing \cite[Algorithm 2.14]{DS251} and \cref{algo:Rescatmotivation}, which allows to construct any resolving closure geometrically. 

Recall that we fix a gentle tree $(Q,R)$, and we set $(\pmb{\Sigma}, \mathcal{M},\Delta^{\gpoint}) = \Surf(Q,R)$, $\Accord = \mathscr{A}(\Surf(Q,R))$, and $\Accord' = \Accord \setminus \Prj(\Delta^{\gpoint})$.

\subsection{Move-closed sets and resolving sets}
\label{ss:Movecomplete}

\begin{definition} \label{def:moveclos}
Let $\mathfrak{B} \subseteq \Accord'$. We say that $\mathfrak{B}$ is \new{move-closed} if $\mathfrak{B}$ is closed under $\KE$-moves, $\W$-moves, $\CoZ$-moves, $\V$-moves and $\X$-moves. We say that $\mathfrak{B}$ is \new{strongly move-closed} if $\mathfrak{B}$ is move-closed, and is closed under $\KE'$-moves and $\Smov$-moves.
\end{definition}

\begin{prop} \label{prop:moveclosedandRes}
Let $\mathfrak{B} \subseteq \Accord'$. Assume that $\mathfrak{B}$ is an  ideal of $(\Accord',\Accordleq)$. Then $\mathfrak{B}$ is move-closed if and only if $\mathfrak{B} \cup \Prj(\Delta^{\gpoint})$ is a resolving set of $\mathfrak{A}$.
\end{prop}  

\begin{proof}
If $\mathfrak{B} \subset \Accord'$ is such that $\mathfrak{B}\cup \Prj(\Delta^{\gpoint})$ is a resolving set, then $\mathfrak{B}$ is an ideal by \cref{prop:folkloreresaccord}. Note that the $\KE$-move adds kernels or extensions that are non-projective and thus accordions in $\mathfrak{B}$. Using \cref{lem:welldefinedCoZ}, we obtain the closure under $\CoZ$-move. By \cref{V}, we have the closure under $\V$-move.  \cref{prop:X} ensures the closure under $\X$-move. The $\W$-move is a combination of these previous moves applied to syzygies. Thus,  $\mathfrak{B}$ is move-closed.

Let us assume that $\mathfrak{B}$ is a move-closed ideal. By \cref{lem:IdealsclosedSyzygies}, $\mathfrak{B} \cup \Prj(\Delta^{\gpoint})$ is closed under taking syzygies. Consider $\delta, \eta\in\mathfrak{B} \cup \Prj(\Delta^{\gpoint})$. If, for instance, $\eta \in \Prj(\Delta^{\gpoint})$, then the accordions in $\ArExt(\delta,\eta)$ or in $\OvExt(\delta,\eta)$ are smaller than $\delta$ for $\Accordleq$. So they are in $\mathfrak{B}$. Assume that $\delta$ and $\eta$ are both non-projective accordions. Then accordions arising from overlap extensions are also obtained by $\X$-moves and $\V$-moves, and those arising from arrow extensions by $\KE$-moves and $\V$-moves. So $\mathfrak{B}$ is closed under extensions. By \cref{thm:equivres}, $\mathfrak{B} \cup \Prj(\Delta^{\gpoint})$ is a resolving set. 
\end{proof}

%\ibenj{T'es sûr que les $\V$-moves gèrent routes les arrow ext ? Je pensais que ça gérait surtout des extensions supérieurs, ou certain overlap extension ou l'un des deux indecomposables est projectif. C'est pas plutôt les $\KE$-moves ? Il y a potentiellement des cas de $\V$-moves qui s'apparentent à du $\KE$, et donc, je ne suis pas sûr mais, je pense qu'il te faut les deux pour gérer routes les extensions...}

According to the proof, and combined with \cref{lem:IdealsclosedSyzygies}, we obtain the following result.

\begin{cor} \label{cor:KEVXonly}
    Let $\mathfrak{B} \in \mathscr{J}(\Accord', \Accordleq)$. Then $\mathfrak{B}$ is move-closed if and only if $\mathfrak{B}$ is closed under $\KE$-moves, $\V$-moves and $\X$-moves.
\end{cor}
    
Based on the same ideas, we can state a stronger result.

\begin{cor} \label{cor:moveandstronglymove}
Let $\mathfrak{B} \in \mathscr{J}(\Accord', \Accordleq)$. Then $\mathfrak{B}$ is  strongly move-closed if and only if $\mathfrak{B}$ is move-closed if and only if $\mathfrak{B}$ is closed under $\KE$-moves,  $\V$-move and $\X$-move
\end{cor}
\begin{proof}
    It is sufficient to note that:
    \begin{enumerate}[label=$\bullet$,itemsep=1mm]
        \item Assume that $(\delta,\eta)\in\Accord'^2$  admits both a $\CoZ$-completion $\xi$, and a $\KE'$-completion $\kappa$. Then, by \cref{lem:KE'andKE}, $\kappa$ is the $\KE$-completion of $(\xi,\eta)$; and
        \item Applying a $\Smov$-move on $(\delta,\eta) \in (\Accord')^2$, which is relevant only if $\col_\delta$ and $\col_\eta$ are not matchable, consists in applying all the five other moves on pairs of accordions $(\xi,\varsigma)$ such that $\xi \Accordleq \delta$ and $\varsigma \Accordleq \eta$.
    \end{enumerate}
    So, any move-closed ideal is a strongly move-closed ideal.
    Therefore, the desired equivalence follows from \cref{prop:moveclosedandRes}, and \cref{cor:KEVXonly}.
\end{proof}

Recall from \cref{prop:1to1antiideal} that $\Theta_{\Accord'}$ associates the ideal generated by this antichain.
\begin{prop} \label{prop:moveclosedandantichain}
Let $\mathfrak{B} \in \mathscr{J}(\Accord',\Accordleq)$, and $\mathfrak{D} = \Theta_{\Accord'}^{-1}(\mathfrak{B}) \in\Anti(\Accord',\Accordleq)$. The following assertions are equivalent:
\begin{enumerate}[label=$(\roman*)$,itemsep=1mm]
    \item $\mathfrak{B}$ is move-closed;
    \item For any $(\delta,\eta) \in \mathfrak{D}^2 $, the sets $\KE(\delta,\eta)$, $\CoZ(\delta,\eta)$, $\KE'(\delta,\eta)$, $\V(\delta,\eta)$, $\X(\delta,\eta)$, $\W(\delta, \eta)$ and $\Smov(\delta,\eta)$ are subsets of $\mathfrak{B}$.
\end{enumerate}
\end{prop}

\begin{proof}
It is obvious that $(i)$ implies $(ii)$, by \cref{cor:moveandstronglymove}.

Let $\mathfrak{B} \in \mathscr{J}(\Accord',\Accordleq)$ satisfy $(ii)$. Set $\mathfrak{D} = \Theta_{\Accord'}^{-1}(\mathfrak{B})$. The result is clear if $\mathfrak{B}$ is either empty or a principal ideal. Moreover, $\mathfrak{B}$ is an ideal, so it is closed under non-projective syzygies. To show that $\mathfrak{B}$ is move-closed, we will show that $\mathfrak{B}$ is a resolving set by using \cref{cor:maingeommotivation}. More precisely, it is enough to show that for any $(\delta, \eta) \in \Theta_{\Accord'}^{-1}(\mathfrak{B})^2$, we have $\opResAc'(\delta,\eta) \subseteq \mathfrak{B}$.

Assume that $\#\mathfrak{D} \geqslant 2$. Let $(\delta, \eta) \in \mathfrak{D}^2$ be such that $\delta \neq \eta$. First, if we have $\NP(\delta) \cap \NP(\eta) = \varnothing$, then, by \cref{prop:noNPs}, $\mathfrak{E}_{\opResAc'(\delta,\eta)} = \{\delta, \eta\}$. Assume that $\NP(\delta) \cap \NP(\eta) \neq \varnothing$.

Assume that $\col_\delta$ and $\col_\eta$ match. 
As $\delta$ and $\eta$ are not comparable, $\{\delta, \eta\}$ does not satisfy \ref{0E} by \cref{prop:0E}. Then, we can distinguish the cases below by \cref{lem:allexclusivityhyp}.
\begin{enumerate}[label=$\bullet$, itemsep=1mm]
    \item If $(\delta,\eta)$ satisfies \ref{2E}, by \cref{cor:KEmove}, we have that either $\mathfrak{E}_{\opResAc'(\delta,\eta)} = \KE(\delta,\eta)$ or $\mathfrak{E}_{\opResAc'(\delta,\eta)} = \KE(\eta,\delta)$, and in both cases $\mathfrak{E}_{\opResAc'(\delta,\eta)} 
    \subseteq \mathfrak{B}$ by assumption $(ii)$.
    \item If $\delta$ and $\eta$ are weakly matchable and $(\delta,\eta)$  does not satisfy \ref{2E} we are reduced to the strongly matchable case by using a $\W$-move . Indeed, up to exchanging $\delta$ and $\eta$ and in the notations of \cref{prop:weaklytostrongly}, we replace $\lbrace\delta,\eta\rbrace$ by  the antichain $\lbrace\eta,\xi_{\rho_+}^{(\delta)}\rbrace$, which is strongly matchable. The antichain $\lbrace\eta,\xi_{\rho_+}^{(\delta)}\rbrace$ then falls into one of the cases treated below.
    \item If $\{\delta,\eta\}$ are strongly matchable, satisfy \ref{3E}, and:
    \begin{enumerate}[label=$\bullet$, itemsep=1mm]
        \item if $\delta$ and $\eta$ are not crossing, by assuming without loss of generality that $\delta$ is above $\eta$, $(\delta, \eta)$ admits a $\CoZ$-completion $\xi$. By \cref{prop:Co-Z}, we have that $\xi \in \opResAc'(\delta, \eta)$, and  $\eta \Accordleq \xi$. In addition, by assuming without loss of generality that $s(\delta)$ is the unique non-exclusive vertex of $(\delta,\eta)$, we have that $(\delta,\eta)$ admits a $\KE'$-completion $\kappa$. Then $\KE'(\delta,\eta) = \mathfrak{E}_{\opResAc'(\delta, \eta)}$ by \cref{prop:3EforKE}, and so $\mathfrak{E}_{\opResAc'(\delta, \eta)} \subseteq \mathfrak{B}$ by assumption $(ii)$;
        \item if $\delta$ and $\eta$ are crossing, then $(\delta,\eta)$ admits a $\V$-completion $\varsigma$. By \cref{prop:SmallEV3}, we get that $\varsigma \in \opResAc'(\delta, \eta)$, and $\delta \Accordleq \varsigma$. By assumption $(ii)$, $\varsigma \in \mathfrak{B}$, and $\delta$ is maximal in $\mathfrak{B}$  this implies that $\varsigma = \delta$. This is a contradiction. 
    \end{enumerate}
    \item If $(\delta,\eta)$ are strongly matchable, satisfy \ref{4E}, and:
    \begin{enumerate}[label=$\bullet$, itemsep=1mm]
        \item if $\delta$ and $\eta$ are not crossing, $\{\delta, \eta\}$ admits a $\CoZ$-completion $\xi$. By \cref{prop:Co-Z}, we have that $\xi \in \opResAc'(\delta, \eta)$, and  $\{\delta,\eta,\xi\} \in \Anti(\Accord',\Accordleq)$, and therefore $\{\delta,\eta,\xi\} \bleq \mathfrak{E}_{\opResAc'(\delta,\eta)}$. Moreover, by applying $\KE$-moves on $(\delta,\xi)$ and on $(\xi,\eta)$, we have that \[\KE(\delta,\xi) \curlyvee \KE(\xi,\eta)  = \mathfrak{E}_{\opResAc'(\delta, \xi)} \curlyvee \mathfrak{E}_{\opResAc'(\xi,\eta)} = \mathfrak{E}_{\opResAc'(\delta,\eta)},\] as $\xi \in \opResAc'(\delta,\eta)$. Therefore, by assumption $(ii)$, $\mathfrak{E}_{\opResAc'(\delta,\eta)} \subseteq \mathfrak{B}$.
        \item if $\delta$ and $\eta$ are crossing, then $(\delta,\eta)$ admits a $\X$-completion $(\varrho_1,\varrho_2)$. By \cref{prop:X}, we have that $\varrho_1,\varrho_2 \in \opResAc'(\delta,\eta)$, and $\{\delta, \eta, \varrho_1, \varrho_2\} \in \Anti(\Accord', \Accordleq)$. Moreover, by applying $\KE$-moves on $(\delta,\varrho_i)$ and on $(\varrho_i,\eta)$, for $i \in \{1,2\}$, we have that \[
        \begin{split}
            & \KE(\delta,\varrho_1) \curlyvee \KE(\varrho_1,\eta) \curlyvee \KE(\delta,\varrho_2) \curlyvee \KE(\varrho_2,\eta) \\
            =\ & \mathfrak{E}_{\opResAc'(\delta, \varrho_1)} \curlyvee \mathfrak{E}_{\opResAc'(\varrho_1,\eta)} \curlyvee \mathfrak{E}_{\opResAc'(\delta, \varrho_2)} \curlyvee \mathfrak{E}_{\opResAc'(\varrho_2,\eta)} \\
            =\ & \mathfrak{E}_{\opResAc'(\delta,\eta)},
        \end{split}\] as $\varrho_1,\varrho_2 \in \opResAc'(\delta,\eta)$. Therefore, by assumption $(ii)$, $\mathfrak{E}_{\opResAc'(\delta,\eta)} \subseteq \mathfrak{B}$.
    \end{enumerate}
\end{enumerate}
In all cases, those elements obtained by applying moves are in $\opResAc'(\delta,\eta)$. 

Assume that $\col_\delta$ and $\col_\eta$ are not matchable. The only move we can apply is an $\Smov$-move. For any $\rho \in \NP(\delta) \cap \NP(\eta)$, as the $\Smov(\rho)$-approximation $(\xi_\rho^{(\delta)},\varsigma_\rho^{(\eta)})$ is a pair of accordions with matchable colorations, we have \[\All^{+}(\xi_\rho^{(\delta)}, \varsigma_\rho^{(\eta)}) = \mathfrak{E}_{\opResAc'(\xi_\rho^{(\delta)},\varsigma_\rho^{(\eta)})}.\]
Then, by \cref{prop:Sapproxisenough}, we have that $\Smov(\delta,\eta) = \mathfrak{E}_{\opResAc'(\delta,\eta)}$. By our assumption on $\Theta_{\Accord'}^{-1}(\mathfrak{B})$, we obtain $\opResAc'(\delta,\eta) \subseteq \mathfrak{B}$.

It follows that $\opResAc'(\mathfrak{D})=\bigcup_{(\delta,\eta)}\opResAc'(\delta,\eta)\subseteq\mathfrak{B}$. 
Moreover, we have that \[\mathfrak{B} = \bigcup_{\delta \in \mathfrak{D}} \opResAc'(\delta) \subseteq \opResAc'(\mathfrak{D}).\]
Hence, $\mathfrak{B}$ is resolving and is move-closed by \cref{prop:moveclosedandRes}.
\end{proof}

\begin{remark} \label{rem:antichains}
In the previous proof, we note that the $\KE'$-move and $\Smov$-move must be considered for the antichains that generate the move-closed ideals of $(\Accord',\Accordleq)$.
\end{remark}

\subsection{Algorithm on the geometric model}
\label{ss:Algo}

As all the configuration of pairs of non-comparable accordions $(\delta,\eta)$ in $(\mathscr{A} \setminus \Prj(\Delta^{\gpoint}), \Accordleq)$ are known, we can make our algorithm explicit.

Let us give a first algorithm which, given some nonempty subset $\mathfrak{B} \subseteq \Accord'$,  returns the antichain $\mathfrak{D} \in \Anti(\Accord', \Accordleq)$ such that $\langle \mathfrak{B} \rangle_{\Accord'} = \langle \mathfrak{D} \rangle_{\Accord'}$.

\begin{algo} \label{algo:joinantichain} Let $\mathfrak{D}$ be a nonempty subset of $(\Accord', \Accordleq)$.
\begin{enumerate}[label = $(\arabic*)$,itemsep=1mm]
    \item Set $\mathfrak{D}_0 = \mathfrak{D}$.
    \item At the $i$th iteration of the algorithm, we denote by $\mathfrak{D}_i$ the subset we are working with.
    \item  If there exists $(\delta,\eta) \in \mathfrak{D}_i^2$ such that:
    \begin{enumerate}[label=$\bullet$, itemsep=1mm]
        \item $\col_\delta$ and $\col_\eta$ are matching; and,
        \item $s(\eta)$ and $t(\eta)$ are not exclusive to $\eta$ in $(\delta,\eta)$;
    \end{enumerate}
    then, we set: $\mathfrak{D}_{i+1} = \mathfrak{D}_i \setminus \{ \eta \}$.
    \item Otherwise, we are done.
\end{enumerate}
\end{algo}

\begin{lemma} \label{lem:joinantichain} 
    Let $(Q,R)$ be a gentle tree, and $\Surf(Q,R) = (\pmb{\Sigma}, \mathcal{M}, \Delta^{\gpoint})$ be its associated $\gpoint$-dissected marked disc. For any $ \varnothing \neq \mathfrak{D} \subseteq \Accord'$, \cref{algo:joinantichain} returns the antichain $\displaystyle \Theta_{\Accord'}^{-1} \left(\langle \mathfrak{D} \rangle_{\Accord'} \right)$.
\end{lemma}

\begin{proof}
    The result follows from \cref{prop:0E}.
\end{proof}

\begin{remark} \label{rem:joinantichain}
    For all $p \in \mathbb{N}^*$, \cref{algo:joinantichain} allows to compute $\mathfrak{B}_1 \curlyvee \cdots \curlyvee \mathfrak{B}_p$ for any $p$-tuple $(\mathfrak{B}_1, \ldots, \mathfrak{B}_p) \in \Anti(\Accord', \Accordleq)^p$, by inputting $\mathfrak{D} = \mathfrak{B}_1 \cup \cdots \cup \mathfrak{B}_p$.
\end{remark}

Now we give an algorithm which, given any $\mathfrak{D} \subseteq \Accord'$, returns the antichain $\mathfrak{E}$ such that $\opResAc'(\mathfrak{D}) = \langle \mathfrak{E} \rangle_{\Accord'}$.

\begin{algo} \label{algo:resclosalgo} Let $\mathfrak{D}$ be an antichain in $(\Accord', \Accordleq)$.
\begin{enumerate}[label = $(\arabic*)$,itemsep=1mm]
    \item Set $\mathfrak{D}_0 = \mathfrak{D}$.
    \item At the $i$th iteration of the algorithm, we denote by $\mathfrak{D}_i$ the antichains we are working with.
    \item For any $(\delta,\eta) \in \mathfrak{D}_i^2$ such that $\delta \neq \eta$, we define $\mathfrak{K}(\delta,\eta)$ as follows:
    \begin{enumerate}[label=$(3\alph*)$, itemsep=1mm]
        \item if $\NP(\delta) \cap \NP(\eta) = \varnothing$, then $\mathfrak{K}(\delta,\eta) = \varnothing$;
        \item if $\NP(\delta) \cap \NP(\eta) \neq \varnothing$, and $\col_\delta$ and $\col_\eta$ are not matchable, then $\mathfrak{K}(\delta,\eta) = \Smov(\delta,\eta)$;
        \item if $\NP(\delta) \cap \NP(\eta) \neq \varnothing$, and $\col_\delta$ and $\col_\eta$ are weakly matchable, then $\mathfrak{K}(\delta,\eta) = \W(\delta,\eta)$; and
        \item  otherwise, we set \[\mathfrak{K}(\delta,\eta) = \KE(\delta,\eta) \curlyvee \CoZ(\delta,\eta) \curlyvee \KE'(\delta,\eta) \curlyvee \V(\delta,\eta) \curlyvee \X(\delta,\eta).\]
    \end{enumerate}
    \item  We compute
    \[\mathfrak{D}_{i+1} = \mathfrak{D}_i \curlyvee \left(\bigcurlyvee_{(\delta,\eta) \in \mathfrak{D}_i^2, \delta \neq \eta} \mathfrak{K}(\delta,\eta)\right)\]
    \item If $\mathfrak{D}_{i+1} \neq \mathfrak{D}_i$, we come back to step $(3)$.
    \item Otherwise, we return $\mathfrak{D}_i$.
\end{enumerate}
\end{algo}

\begin{theorem} \label{thm:Resclosalgosucceed}
Let $(Q,R)$ be a gentle tree, and $\Surf(Q,R) = (\pmb{\Sigma}, \mathcal{M}, \Delta^{\gpoint}$) be its associated dualisable $\gpoint$-dissected marked disc. For any antichain $\mathfrak{\mathfrak{D}}$ in $(\Accord, \Accordleq)$, \cref{algo:resclosalgo} ends, and returns the unique antichain $\mathfrak{E}$ such that \[\opResAc(\mathfrak{D}) = \opResAc(\mathfrak{E}) = \bigcup_{\delta \in \mathfrak{E}} \opResAc(\delta).\]
\end{theorem}

\begin{proof}
This result follows from \cref{prop:moveclosedandRes,prop:moveclosedandantichain}, \cref{thm:EndAlgo}, and \cref{cor:maingeommotivation}.
\end{proof}
  
\subsection{An example}
\label{ss:examplePart1}

Let us consider $(Q,R)$ to be the gentle tree drawn in \cref{fig:Exsurf1}. We also give its Res-poset $(\pmb{\ind \setminus \proj}, \Resleq)$, and its $\rpoint$-dissected marked surface $(\pmb{\Sigma}, \mathcal{M}, \Prj(\Delta^{\gpoint}))$.

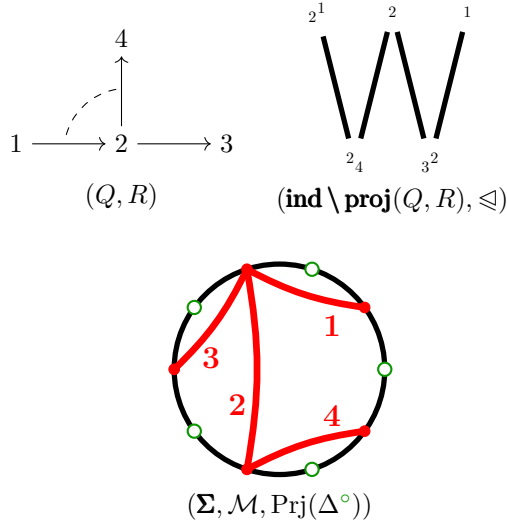
\begin{figure}[!ht]
\centering 
    \begin{tikzpicture}
    \begin{scope}[->,scale=1.4]
		\node (a) at (0,0) {$1$};
		\node (b) at (1,0) {$2$};
		\node (c) at (2,0) {$3$};
		\node (d) at (1,1) {$4$};
		
		\draw (a) -- (b);
		\draw (b) -- (c);
		\draw (b) -- (d);

		\draw[dashed,-] ([xshift=-.05cm,yshift=.35cm]b.north) arc[start angle = 100, end angle = 170, x radius=.6cm, y radius =.6cm];
		
		\node at (1,-.5) {$(Q,R)$};
    \end{scope}
    
    \begin{scope}[xshift=5cm,yshift=-1.25cm]
	
	    \node (b) at (-0.5,1) {\scalebox{0.6}{$
	    \begin{tikzpicture}[baseline={(0,-.2)},scale=0.2]
	    \node at (0,0){$2$};
	    \node at (1,-1){$4$};
	    \end{tikzpicture}$}};
	    \node (c) at (0.5,1)  {\scalebox{0.6}{$
	    \begin{tikzpicture}[baseline={(0,-.2)},scale=0.2]
	    \node at (0,0){$2$};
	    \node at (-1,-1){$3$};
	    \end{tikzpicture}$}};
	    \node (e) at (0,3) {\scalebox{0.6}{$
	    \begin{tikzpicture}[baseline={(0,-.1)},scale=0.2]
	    \node at (0,0){$2$};
	    \end{tikzpicture}$}};
	    \node (f) at (-1,3) {\scalebox{0.6}{$
	    \begin{tikzpicture}[baseline={(0,-.2)},scale=0.2]
	    \node at (0,0){$1$};
	    \node at (-1,-1){$2$};
	    \end{tikzpicture}$}};
	    \node (h) at (1,3) {\scalebox{0.6}{$\begin{tikzpicture}[baseline={(0,-.1)},scale=0.2]
	    \node at (0,0){$1$};
	    \end{tikzpicture}$}};

		\draw[line width=0.7mm,black] (b) -- (e);
		\draw[line width=0.7mm,black] (b) -- (f);
		\draw[line width=0.7mm,black] (c) -- (e);
		\draw[line width=0.7mm,black] (c) -- (h);
		
		\node at (0,.5) {$(\pmb{\ind \setminus \proj}(Q,R), \Resleq)$};
	\end{scope}
	\begin{scope}[xshift=3.5cm,yshift=-3cm,mydot/.style={
					circle,
					thick,
					fill=white,
					draw,
					outer sep=0.5pt,
					inner sep=1pt
				}, scale = 0.7]
		\tikzset{
		osq/.style={
        rectangle,
        thick,
        fill=white,
        append after command={
            node [
                fit=(\tikzlastnode),
                orange,
                line width=0.3mm,
                inner sep=-\pgflinewidth,
                cross out,
                draw
            ] {}}}}
		\draw[line width=0.7mm,black] (0,0) ellipse (2cm and 2cm);
		\foreach \X in {0,1,...,9}
		{
		\tkzDefPoint(2*cos(pi/5*\X),2*sin(pi/5*\X)){\X};
		};

		\draw[line width=0.9mm ,bend left =10,red](1) edge (3);
		\draw[line width=0.9mm ,bend left =10,red](3) edge (7);
		\draw[line width=0.9mm ,bend left =10,red](3) edge (5);
		\draw[line width=0.9mm ,bend left =10,red](7) edge (9);

		\foreach \X in {1,3,5,7,9}
		{
		\tkzDrawPoints[fill =red,size=4,color=red](\X);
		};
		
		\foreach \X in {0,2,4,6,8}
		{
		\tkzDrawPoints[size=5,dark-green,mydot](\X);
		};

		\tkzDefPoint(1,1.2){gamma};
		\tkzLabelPoint[red](gamma){\Large $\mathbf{1}$}
		\tkzDefPoint(-0.8,-0.3){gamma};
		\tkzLabelPoint[red](gamma){\Large $\mathbf{2}$}
		\tkzDefPoint(-1.3,0.6){gamma};
		\tkzLabelPoint[red](gamma){\Large $\mathbf{3}$}
		\tkzDefPoint(1,-0.5){gamma};
		\tkzLabelPoint[red](gamma){\Large $\mathbf{4}$}
		\node at (0,-2.6) {$(\pmb{\Sigma}, \mathcal{M}, \Prj(\Delta^{\gpoint}))$};
	\end{scope}
    \end{tikzpicture}
\caption{\label{fig:Exsurf1} The gentle tree $(Q,R)$ with the Res-relation order on its non-projective indecomposable representations seen in \cref{fig:Resrelex}, and the $\rpoint$-dissected marked surface associated to $(Q,R)$.}
\end{figure}

We compute two resolving subcategories to illustrate the use of \cref{algo:resclosalgo}. We identify the notation for strings and the one for accordions.

\begin{ex} \label{ex:Rescalc1} Let us consider $\delta = \begin{tikzpicture}[baseline={(0,-.2)},scale=0.2]
	    \node at (0,0){$2$};
	    \node at (-1,-1){$3$};
	    \end{tikzpicture}$, and $\eta = \begin{tikzpicture}[baseline={(0,-.2)},scale=0.2]
	    \node at (0,0){$1$};
	    \node at (-1,-1){$2$};
	    \end{tikzpicture}$. 
    
We compute $\opResAc(\delta,\eta)$. See \cref{fig:Exsurf2} to illustrate the following explanations.
\begin{figure}[!ht]
\centering 
    \begin{tikzpicture}
    \begin{scope}
	    \node (b) at (-0.5,1) {\scalebox{0.6}{$
	    \begin{tikzpicture}[baseline={(0,-.2)},scale=0.2]
	    \node at (0,0){$2$};
	    \node at (1,-1){$4$};
	    \end{tikzpicture}$}};
	    \node (c) at (0.5,1)  {\scalebox{0.6}{$\eta =
	    \begin{tikzpicture}[baseline={(0,-.2)},scale=0.2]
	    \node at (0,0){$2$};
	    \node at (-1,-1){$3$};
	    \end{tikzpicture}$}};
	    \node (e) at (0,3) {\scalebox{0.6}{$ \varsigma =
	    \begin{tikzpicture}[baseline={(0,-.1)},scale=0.2]
	    \node at (0,0){$2$};
	    \end{tikzpicture}$}};
	    \node (f) at (-1,3) {\scalebox{0.6}{$ \delta =
	    \begin{tikzpicture}[baseline={(0,-.2)},scale=0.2]
	    \node at (0,0){$1$};
	    \node at (-1,-1){$2$};
	    \end{tikzpicture}$}};
	    \node (h) at (1,3) {\scalebox{0.6}{$\begin{tikzpicture}[baseline={(0,-.1)},scale=0.2]
	    \node at (0,0){$1$};
	    \end{tikzpicture}$}};

		\draw[line width=0.7mm,black] (b) -- (e);
		\draw[line width=0.7mm,black] (b) -- (f);
		\draw[line width=0.7mm,black] (c) -- (e);
		\draw[line width=0.7mm,black] (c) -- (h);
		
		\node at (0,.5) {$(\pmb{\ind \setminus \proj}(Q,R), \Resleq)$};
	\end{scope}
	\begin{scope}[xshift=5cm,yshift=2cm,mydot/.style={
					circle,
					thick,
					fill=white,
					draw,
					outer sep=0.5pt,
					inner sep=1pt
				}, scale = 0.7]
		\tikzset{
		osq/.style={
        rectangle,
        thick,
        fill=white,
        append after command={
            node [
                fit=(\tikzlastnode),
                orange,
                line width=0.3mm,
                inner sep=-\pgflinewidth,
                cross out,
                draw
            ] {}}}}
		\draw[line width=0.7mm,black] (0,0) ellipse (2cm and 2cm);
		\foreach \X in {0,1,...,9}
		{
		\tkzDefPoint(2*cos(pi/5*\X),2*sin(pi/5*\X)){\X};
		};
		
		\draw[line width=0.9mm ,bend left =10,red](1) edge (3);
		\draw[line width=0.9mm ,bend left =10,mypurple,densely dashdotted](3) edge (7);
		\draw[line width=0.9mm ,bend left =10,red](3) edge (5);
		\draw[line width=0.9mm ,bend left =10,red](7) edge (9);
		
		\draw[line width=0.7mm ,bend right=10,blue, loosely dashed](1) edge (5);
		\draw[line width=0.7mm ,bend left=10,blue, loosely dashed](3) edge (9);
	
		\draw[line width=0.7mm ,bend right=10,orange, dash pattern={on 10pt off 2pt on 5pt off 2pt}](5) edge (9);

		\foreach \X in {3}
		{
		\tkzDrawPoints[fill =red,size=4,color=red](\X);
		};
		
		\foreach \X in {1,5,7,9}
		{
		\tkzDrawPoints[size=6,color=mypurple](\X);
		};
		
		\foreach \X in {0,2,4,6,8}
		{
		\tkzDrawPoints[size=5,dark-green,mydot](\X);
		};

		\tkzDefPoint(-1.1,0.5){gamma};
		\tkzLabelPoint[blue](gamma){\Large $\eta$}
		\tkzDefPoint(0.8,-0.1){i};
		\tkzLabelPoint[blue](i){\Large $\delta$}
		\tkzDefPoint(0,-0.2){j};
		\tkzLabelPoint[orange](j){\Large $\varsigma$}
		\node at (0,-2.6) {$(\pmb{\Sigma}, \mathcal{M}, \Prj(\Delta^{\gpoint}))$};
	\end{scope}
    \end{tikzpicture}
\caption{\label{fig:Exsurf2} In this example, $(\delta,\eta)$ admits a $\V$-completion $\varsigma$, and by \cref{V}, we have $\delta \protect\Accordlneq \varsigma$ and $\opResAc(\delta, \eta) = \opResAc(\varsigma) \cup \opResAc(\eta)$.}
\end{figure}

First, $\NP(\delta) \cap \NP(\eta) = \left\{ \rho = \begin{tikzpicture}[baseline={(0,-.2)},scale=0.2]
	    \node at (0,0){$2$};
	    \node at (1,-1){$4$};
            \node at (-1,-1){$3$};
	    \end{tikzpicture} \right\}$. We drew $\rho$ as the purple densely dashed curve.
        
Then, $\col_\delta$ and $\col_\eta$ are matchable as the common endpoint $v$ of $\rho$ and $\delta$ is neither colored in ${\osquare}$ nor in ${\psquare}$ for both colorations. We consider the induced laterality such that $\delta$ is above $\eta$; in particular, the common endpoint of $\delta$ and $\rho$ is $s(\delta)$.

The pair $(\delta,\eta)$ satisfies \ref{3E}, and $s(\delta)$ is not exclusive in $(\delta,\eta)$. This also gives  $\{\delta,\eta\} \in \Anti(\Accord', \Accordleq)$. Therefore, $(\delta,\eta)$ admits a $\V$-completion $\varsigma = \begin{tikzpicture}[baseline={(0,-.1)},scale=0.2]
	    \node at (0,0){$2$};
	    \end{tikzpicture}$. We drew $\varsigma$ as the orange dashed curve.

By \cref{prop:SmallEV3}, we have that $\delta \Accordlneq \varsigma$. We have then to restart our calculation with the pair $(\eta,\varsigma)$. We invite the reader to check that $\{\eta,\varsigma\}$ is an antichain of $(\Accord',\Accordleq)$ which satisfy $(ii)$ of \cref{prop:moveclosedandantichain}. Therefore we can conclude that $\opResAc(\delta,\eta) = \opResAc(\varsigma) \cup \opResAc(\eta)$.
\end{ex}

\begin{ex} \label{ex:Rescalc2}
Let us consider $\delta = \begin{tikzpicture}[baseline={(0,-.2)},scale=0.2]
	    \node at (0,0){$2$};
	    \node at (1,-1){$4$};
	    \end{tikzpicture}$, and $\eta = \begin{tikzpicture}[baseline={(0,-.1)},scale=0.2]
	    \node at (0,0){$1$};
	    \end{tikzpicture}$. 
    
We compute $\opResAc(\delta,\eta)$. See \cref{fig:Exsurf3} for an illustration of the following explanations.
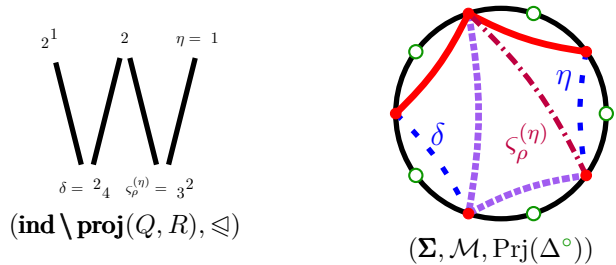
\begin{figure}[!ht]
\centering 
    \begin{tikzpicture}
    \begin{scope}
	
	    \node (b) at (-0.5,1) {\scalebox{0.6}{$ \delta =
	    \begin{tikzpicture}[baseline={(0,-.2)},scale=0.2]
	    \node at (0,0){$2$};
	    \node at (1,-1){$4$};
	    \end{tikzpicture}$}};
	    \node (c) at (0.5,1)  {\scalebox{0.6}{$ \varsigma_\rho^{(\eta)} = 
	    \begin{tikzpicture}[baseline={(0,-.2)},scale=0.2]
	    \node at (0,0){$2$};
	    \node at (-1,-1){$3$};
	    \end{tikzpicture}$}};
	    \node (e) at (0,3) {\scalebox{0.6}{$   \begin{tikzpicture}[baseline={(0,-.1)},scale=0.2]
	    \node at (0,0){$2$};
	    \end{tikzpicture}$}};
	    \node (f) at (-1,3) {\scalebox{0.6}{$ 
	    \begin{tikzpicture}[baseline={(0,-.2)},scale=0.2]
	    \node at (0,0){$1$};
	    \node at (-1,-1){$2$};
	    \end{tikzpicture}$}};
	    \node (h) at (1,3) {\scalebox{0.6}{$\eta = \begin{tikzpicture}[baseline={(0,-.1)},scale=0.2]
	    \node at (0,0){$1$};
	    \end{tikzpicture}$}};

		\draw[line width=0.7mm,black] (b) -- (e);
		\draw[line width=0.7mm,black] (b) -- (f);
		\draw[line width=0.7mm,black] (c) -- (e);
		\draw[line width=0.7mm,black] (c) -- (h);
		
		\node at (0,.5) {$(\pmb{\ind \setminus \proj}(Q,R), \Resleq)$};
	\end{scope}
	\begin{scope}[xshift=5cm,yshift=2cm,mydot/.style={
					circle,
					thick,
					fill=white,
					draw,
					outer sep=0.5pt,
					inner sep=1pt
				}, scale = 0.7]
		\tikzset{
		osq/.style={
        rectangle,
        thick,
        fill=white,
        append after command={
            node [
                fit=(\tikzlastnode),
                orange,
                line width=0.3mm,
                inner sep=-\pgflinewidth,
                cross out,
                draw
            ] {}}}}
		\draw[line width=0.7mm,black] (0,0) ellipse (2cm and 2cm);
		\foreach \X in {0,1,...,9}
		{
		\tkzDefPoint(2*cos(pi/5*\X),2*sin(pi/5*\X)){\X};
		};
		
		\draw[line width=0.9mm ,bend left =10,red](1) edge (3);
		\draw[line width=0.9mm ,bend left =10,mypurple,densely dashdotted](3) edge (7);
		\draw[line width=0.9mm ,bend left =10,red](3) edge (5);
		\draw[line width=0.9mm ,bend left =10,mypurple,densely dashdotted](7) edge (9);
		
		\draw[line width=0.7mm ,bend right=10,blue, loosely dashed](1) edge (9);
		\draw[line width=0.7mm ,bend left=10,blue, loosely dashed](5) edge (7);
	
		\draw[line width=0.7mm ,bend right=10,purple, dash pattern={on 5pt off 2pt on 1pt off 2pt}](9) edge (3);

		\foreach \X in {1,3,5,7,9}
		{
		\tkzDrawPoints[fill =red,size=4,color=red](\X);
		};
		
		\foreach \X in {0,2,4,6,8}
		{
		\tkzDrawPoints[size=5,dark-green,mydot](\X);
		};

		\tkzDefPoint(-1.2,0){gamma};
		\tkzLabelPoint[blue](gamma){\Large $\delta$}
		\tkzDefPoint(1.2,1){i};
		\tkzLabelPoint[blue](i){\Large $\eta$}
		\tkzDefPoint(0.5,0){s};
		\tkzLabelPoint[purple](s){\Large $\varsigma_{\rho}^{(\eta)}$};
		\node at (0,-2.6) {$(\pmb{\Sigma}, \mathcal{M}, \Prj(\Delta^{\gpoint}))$};
	\end{scope}
    \end{tikzpicture}
\caption{\label{fig:Exsurf3} In this example, $\eta$ and$\delta$ do not admit matchable colorations eventhough $\NP(\delta) \cap \NP(\eta)\neq \varnothing$: we use the $\Smov$-move to calculate $\opResAc(\delta,\eta)$. }
\end{figure}

First, $\NP(\delta) \cap \NP(\eta) = \left\{ \rho = \begin{tikzpicture}[baseline={(0,-.2)},scale=0.2]
	    \node at (0,0){$2$};
            \node at (-1,-1){$3$};
	    \node at (1,-1){$4$};
	    \end{tikzpicture}, \rho' = \begin{tikzpicture}[baseline={(0,-.1)},scale=0.2]
	    \node at (0,0){$4$};
	    \end{tikzpicture} \right\}$, we drew $\rho$ and $\rho'$ as purple densely dashed curves.

Then, given a coloration $\col_\eta$ of $\NP(\eta)_0$, we can check that $\col_{\eta}$ sends the endpoints of $\rho$ and $\rho'$ into $\{{\osquare}, {\psquare}\}$. Therefore, $\col_\eta$ and $\col_\delta$ are not matchable.  We will use an $\Smov$-move on $(\delta,\eta)$. 

The $\Smov(\rho)$-approximation of $(\delta,\eta)$ is $(\delta, \varsigma_\rho^{(\eta)})$ where $\varsigma_\rho^{(\eta)} =\begin{tikzpicture}[baseline={(0,-.2)},scale=0.2]
	    \node at (0,0){$2$};
            \node at (-1,-1){$3$};
	    \end{tikzpicture}$. The $\Smov(\rho')$-approximation $(\delta,\eta)$ gives the same pair of non-projective accordions. By \cref{prop:Sapproxisenough}, $\opResAc(\delta,\eta) = \opResAc(\delta) \cup \opResAc(\eta) \cup \opResAc(\delta,\varsigma_\rho^{(\eta)})$.

We note that $\delta$ and $\varsigma_\rho^{(\eta)}$ are not crossing and do not share a common endpoint, so $(\delta,\varsigma_\rho^{(\eta)})$ does not admit a $\KE$-move. Moreover, they do not cross any accordions in $\Prj(\Delta^{\gpoint})$. So $(\delta,\varsigma_\rho^{(\eta)})$ does not admit a $\CoZ$-move. Using our previous observations, we can also check that $(\delta,\varsigma_\rho^{(\eta)})$ does not admit a $\V$-move, nor an  $\X$-move. Therefore, by \cref{thm:EndAlgo}, $\opResAc(\delta, \varsigma_\rho^{(\eta)}) = \opResAc(\delta) \cup \opResAc(\varsigma_\rho^{(\eta)})$.

Hence we obtain $\opResAc(\delta,\eta) = \opResAc(\delta) \cup \opResAc(\eta)$, as $\varsigma_\rho^{(\eta)} \Accordleq \eta$.
\end{ex}

	%\section{Tilting representations}
	%\label{sec:Tilt}
	%\input{Tilt.tex}
	
	\section*{Acknowledgements}
	
    B.D. thanks the \emph{Institut des Sciences Mathematiques (UQAM)} and the \emph{Engineering and Physical Sciences Research Council (EP/W007509/1))} for their partial funding support. The authors acknowledge the \emph{CHARMS program grant (ANR-19-CE40-0017-02)} for their funding support.
    
	The authors thank Yann Palu and Baptiste Rognerud (especially for our discussions on \cref{sec:ResPoset}) for their discussions and advice throughout this work.

	%\nocite{*}
	\bibliography{Article}
	\bibliographystyle{alpha}
	
\end{document}